\documentclass[a4paper,12pt,reqno]{amsart}
\usepackage{amssymb}
\usepackage{latexsym}
\usepackage{amsfonts}
\usepackage{mathrsfs}
\usepackage{amsmath}
\usepackage{amsthm}
\usepackage{multicol}
\usepackage{constants}
\usepackage{bm}
\theoremstyle{definition}
\newtheorem{definition}{Definition}[section]
\newtheorem{proposition}[definition]{Proposition}
\newtheorem{theorem}[definition]{Theorem}
\newtheorem{corollary}[definition]{Corollary}
\newtheorem{lemma}[definition]{Lemma}
\newtheorem{remark}[definition]{Remark}

\def\N{{\mathbb{N}}}

\def\Q{{\mathbb{Q}}}
\def\R{{\mathbb{R}}}
\def\S{{\mathbb{S}}}
\def\T{{\mathbb{T}}}

\def\div{\mbox{div}\,}

\def\<{\mathop{<}}
\def\>{\mathop{>}}

\newcommand{\spt}{\mathrm{spt}\,}

\newcommand{\dist}{\mathrm{dist}\,}
\newcommand{\diam}{\mathrm{diam}\,}
\numberwithin{equation}{section}
\def\Xint#1{\mathchoice
{\XXint\displaystyle\textstyle{#1}}%
{\XXint\textstyle\scriptstyle{#1}}%
{\XXint\scriptstyle\scriptscriptstyle{#1}}%
{\XXint\scriptscriptstyle\scriptscriptstyle{#1}}%
\!\int}
\def\XXint#1#2#3{{\setbox0=\hbox{$#1{#2#3}{\int}$}
\vcenter{\hbox{$#2#3$}}\kern-.5\wd0}}

\def\dashint{\Xint-}

\makeatletter
\@namedef{subjclassname@2020}{\textup{2020} Mathematics Subject Classification}
\makeatother

\title[Existence of volume preserving MCF]
{Existence of weak solution to volume preserving mean curvature flow in higher dimensions}
\author[K. Takasao]{Keisuke Takasao \\ Department of Mathematics/Hakubi Center, Kyoto University, \\Kitashirakawa-Oiwakecho Sakyo Kyoto 606-8502, Japan}
\email{k.takasao@math.kyoto-u.ac.jp}
\keywords{volume preserving mean curvature flow, Allen--Cahn equation, phase field method}
\subjclass[2020]{Primary~35K93, Secondary~53E10}
\thanks{}
\date{}
\begin{document}
\maketitle
\begin{abstract}
In this paper, we construct a family of integral varifolds, which is
a global weak solution to the volume preserving mean curvature flow in the sense of $L^2$-flow. 
This flow is also a distributional BV-solution for a short time,
when the perimeter of the initial data is sufficiently close to that of ball with the same volume.
To construct the flow, we use the Allen--Cahn equation with non-local term
motivated by studies of Mugnai, Seis, and Spadaro, 
and Kim and Kwon. 
We prove the convergence of the solution for the Allen--Cahn equation
to the family of integral
varifolds with only natural assumptions for the initial data.
\end{abstract}

\section{Introduction}
Let $d\geq 2$ be an integer and $\Omega := \mathbb{T}^d = (\mathbb{R}/\mathbb{Z}) ^d$. 
Assume that $T>0$ and $U_t \subset \Omega$ is an  open set with the smooth boundary 
$M_t:=\partial U_t$ for any $t \in [0,T)$.
The family of the hypersurfaces $\{ M_t \} _{t \in [0,T)}$ is called the volume preserving mean curvature flow
if the normal velocity vector $\vec{v}$ satisfies
\begin{equation}
\vec{v}=\vec{h} - \left( \frac{1}{\mathscr{H}^{d-1} (M_t)} \int _{M_t} \vec{h} \cdot \vec{\nu} \, d \mathscr{H}^{d-1} \right) \vec{\nu},
\quad \text{on} \ M_t , \ t \in (0,T).
\label{vpmcf}
\end{equation}
Here, $\mathscr{H}^{d-1}$ is the $(d-1)$-dimensional Hausdorff measure, and
$\vec{h}$ and $\vec{\nu}$ are the mean curvature vector and the inner unit normal vector of $M_t$, respectively.
Note that the solution $\{M_t\} _{t \in [0,T)}$ to \eqref{vpmcf} satisfies 
\begin{equation}
\frac{d}{dt}\mathscr{H}^{d-1} (M_t) \leq 0 \quad \text{and} \quad
\frac{d}{dt} \mathscr{L}^{d} (U_t) 
= - \int _{M_t} \vec{v} \cdot \vec{\nu} \, d \mathscr{H}^{d-1}  =0, \quad t \in (0,T),
\label{vpproperty}
\end{equation}
where $\mathscr{L}^{d}$ is the $d$-dimensional Lebesgue measure.
From \eqref{vpproperty}, $\{M_t\} _{t \in [0,T)}$ has the volume preserving property, 
that is, $\mathscr{L}^d (U_t)$ is constant with respect to $t$.

When $U_0$ is convex, Gage~\cite{MR848933} and Huisken~\cite{MR921165}
proved that there exists a solution to \eqref{vpmcf} and
it converges to a sphere as $t\to \infty$.
Escher and Simonett~\cite{MR1485470} showed the short time existence of the solution to \eqref{vpmcf}
 for smooth initial data $M_0$
and they also proved that if $M_0$ is sufficiently close to a sphere in the sense of 
the little H\"{o}lder norm $h^{1+\alpha}$,
 then there exists
a global solution and
it converges to some sphere as $t\to \infty$
(see also \cite{MR2780248, MR1456315, MR2552262} for related results).
Mugnai, Seis, and Spadaro~\cite{MR3455792} studied the minimizing movement for \eqref{vpmcf}
and they proved the global existence of the flat flow, that is, 
there exist $C=C(d,U_0) >0$ and a family of Caccioppoli sets $\{ U_t \}_{t \in [0,\infty)}$ such that
$\mathscr{L}^d (U_s \bigtriangleup U_t) \leq C\sqrt{s-t}$ for any $0\leq t<s$, 
$\mathscr{H}^{d-1} (\partial ^\ast U_t)$ is monotone decreasing, and
$\mathscr{L}^d (U_t)$ is constant. Here, $\partial ^\ast U_t$ is the
reduced boundary of $U_t$.
In addition,  for $d \leq 7$, they proved the global existence of the weak solution 
to \eqref{vpmcf} in the sense of the distribution,
under the reasonable assumption for the convergence, that is,
\begin{equation}\label{eq:1.3}
\lim _{k\to \infty} \int _0 ^T \mathscr{H}^{d-1} (\partial ^\ast U_t ^k) \, dt 
= \int _0 ^T \mathscr{H}^{d-1} (\partial ^\ast U_t) \, dt,
\end{equation}
where $\{ U_t ^k \}_{t \in [0,T)}$ is the time-discretized approximate solution to \eqref{vpmcf}.
This kind of condition was introduced in \cite{MR1386964} (see also \cite{MR1205983, MR3556529}).
Laux and Swartz~\cite{MR3667702}
showed the convergence of the thresholding schemes
to the distributional BV-solutions of 
\eqref{vpmcf} under an assumption
of the convergence similar to \eqref{eq:1.3}. 
Laux and Simon~\cite{MR3847750} also proved similar results 
in the case of the phase field method.
On the other hand, 
the author~\cite{takasao2017} proved the existence of the weak solution
(family of integral varifolds) to \eqref{vpmcf}
in the sense of $L^2$-flow for $2\leq d \leq 3$ without 
any such convergence assumption, 
via the phase field method studied by Golovaty~\cite{golovaty}.
Recently, Kim and Kwon~\cite{MR4083196} proved the existence of the viscosity solution to \eqref{vpmcf}
for the case where $U_0$ satisfies a geometric condition called $\rho$-reflection.
Moreover, they also proved that the viscosity solution
converges to some sphere uniformly as $t \to \infty$.

Let $\{\delta _i \}_{i=1} ^\infty$ be a positive sequence with $\delta _i \to 0$ as $i\to \infty$
and we denote $\delta _i$ as $\delta$ for simplicity.
Suppose that $U_t ^\delta$ is an open set with smooth boundary $M_t ^\delta$ 
for any $t \in [0,T)$. 
The approximate solutions studied in \cite{MR3455792} and \cite{MR4083196}
correspond to the following mean curvature flow $\{ M_t ^\delta \}_{t \in [0,T)}$ with non-local term:
\begin{equation}\label{vpmcf.delta}
\vec{v} = \vec{h} - \lambda^\delta \vec{\nu},
\quad \text{on} \ M_t ^\delta , \ t \in (0,T),
\end{equation}
where
\[
\lambda^\delta (t) =\frac{1}{\delta} (\mathscr{L}^d (U_{0} ^\delta) -\mathscr{L}^d (U_t ^\delta)).
\] 
One can check that \eqref{vpmcf.delta} is a $L^2$-gradient flow of
$$
E ^\delta (t)=\mathscr{H}^{d-1} (M_t ^\delta) 
+\frac{1}{2\delta}(\mathscr{L}^d (U_{0}^\delta) - \mathscr{L}^d (U_t ^\delta))^2,
$$
that is, 
\[
\frac{d}{dt} E ^\delta (t) = - \int _{M_t ^\delta} \vert \vec{v} \vert ^2 \, d \mathscr{H}^{d-1} \leq 0 
\qquad \text{for any} \ 
t \in (0,T).
\]
Hence $\{ M_t ^\delta \}_{t \in [0,T)}$ satisfies
a relaxed volume preserving property, namely,
\begin{equation*}
\begin{split}
( \mathscr{L}^d (U_0 ^\delta) -  \mathscr{L}^d (U_t ^\delta) )^2
\leq
2\delta E^\delta (t) 
\leq
2\delta E ^\delta (0) =2\delta \mathscr{H}^{d-1} (M_0 ^\delta).
\end{split}
\end{equation*}
Therefore $\{ M_t ^\delta \}_{t \in [0, T)}$ converges to the solution $\{ M_t \}_{t \in [0,T)}$ to \eqref{vpmcf}
as $\delta \to 0$ formally. 
Note that we cannot directly obtain the monotonically decreasing of
$\mathscr{H}^{d-1} (M_t)$ by the energy estimates above.
However, 
if we have a natural energy estimate $\sup _i \int _0 ^T \vert \lambda ^{\delta_i} (t) \vert ^2 \, dt \leq C_T$ for some  constant $C_T>0$, 
we can expect the property 
in some sense,
because 
\[
\liminf _{i\to \infty}  
\frac{1}{2\delta _i }(\mathscr{L}^d (U_{0}^{\delta_i}) - \mathscr{L}^d (U_t ^{\delta_i}))^2
=
\liminf _{i\to \infty} 
\frac{\delta_i}{2} \vert \lambda ^{\delta_i} (t) \vert ^2
 = 0 \qquad \text{for a.e.} \ t \in [0,T) 
\]
by Fatou's lemma (see Proposition \ref{prop3.12}). 
The reason why the $L^2$-estimate is natural is because the non-local term of the
solution to \eqref{vpmcf} satisfies it
(see Proposition \ref{prop7.3}).
Mugnai, Seis, and Spadaro~\cite{MR3455792} used a minimizing movement scheme 
corresponding to \eqref{vpmcf.delta}, and 
Kim and Kwon~\cite{MR4083196} used \eqref{vpmcf.delta} to prove 
the existence of the viscosity solution to \eqref{vpmcf}.
Based on these results, in this paper we show the global existence of
the weak solution to \eqref{vpmcf}, via the phase field method corresponding to 
\eqref{vpmcf.delta}.

We denote $W(a):= \dfrac{(1-a^2)^2}{2}$ and $k(s) = \int _{0} ^s \sqrt{2W (a)} \, da= s -\dfrac{1}{3} s^3$.
Let $\varepsilon  \in (0,1)$, $T>0$, and $\alpha \in (0,1)$.
With reference to \cite{MR3455792} and \cite{MR4083196},
in this paper we consider the following Allen--Cahn equation with non-local term:
\begin{equation}
\left\{ 
\begin{array}{ll}
\varepsilon \varphi ^{\varepsilon} _t =\varepsilon \Delta \varphi ^{\varepsilon} -\dfrac{W' (\varphi ^{\varepsilon})}{\varepsilon }+ \lambda ^{\varepsilon} \sqrt{2W(\varphi ^\varepsilon)} ,& (x,t)\in \Omega \times (0, \infty),  \\
\varphi ^{\varepsilon} (x,0) = \varphi _0 ^{\varepsilon} (x) ,  &x\in \Omega,
\end{array} \right.
\label{ac}
\end{equation}
where $\lambda^{\varepsilon}$ is given by
\begin{equation}\label{lambdadef}
\lambda^{\varepsilon}(t) =\frac{1}{\varepsilon ^\alpha} 
\left( \int_{\Omega} k(\varphi ^\varepsilon _0 (x))\, dx 
- \int_{\Omega} k(\varphi ^\varepsilon (x,t))\, dx \right).
\end{equation}
Note that if $\varphi _0 ^\varepsilon$ satisfies suitable
assumptions,
the standard PDE theories imply the global existence and uniqueness of 
the solution to \eqref{ac} (see Remark \ref{rem3.3}).
Set
\begin{equation*}
\begin{split}
E^{\varepsilon} (t) \,
& =  \,\int _{\Omega} 
\left( \frac{\varepsilon \vert \nabla \varphi ^\varepsilon (x,t) \vert ^2}{2} 
+ \frac{W(\varphi ^\varepsilon (x,t))}{\varepsilon} \right) \, dx
+
\frac{1}{2\varepsilon ^\alpha}
\left( \int_{\Omega} k(\varphi ^\varepsilon _0 (x))\, dx 
- \int_{\Omega} k(\varphi ^\varepsilon (x,t))\, dx \right)^2 \\
& =:  E_S ^\varepsilon (t) + E_P ^\varepsilon (t).
\end{split}
\end{equation*}
As above, one can check that
the solution $\varphi^{\varepsilon} $ to \eqref{ac} satisfies
\begin{equation}\label{eq:1.11}
\frac{d}{dt} E^\varepsilon (t) = - \int _{\Omega} \varepsilon (\varphi ^\varepsilon _t )^2 \, dx
\leq 0 \qquad \text{for any} \ t \in (0,\infty) ,
\end{equation}
\begin{equation}\label{eq:1.12}
E^\varepsilon (T) 
+
\int _0 ^T \int _{\Omega} \varepsilon (\varphi ^\varepsilon _t )^2 \, dx dt
= 
E^\varepsilon (0) = E ^\varepsilon _S (0)
\qquad \text{for any} \ T\geq 0,
\end{equation}
and 
\begin{equation}\label{eq:1.5}
\left( \int_{\Omega} k(\varphi ^\varepsilon _0 (x))\, dx 
- \int_{\Omega} k(\varphi ^\varepsilon (x,t))\, dx \right)^2
=
2 \varepsilon ^\alpha E^\varepsilon _P (t) 
\leq
 2 \varepsilon ^\alpha E ^\varepsilon _S (0)
\quad \text{for any} \ t\in [0,\infty).
\end{equation}
Assume $ \sup _{\varepsilon \in (0,1)} E^\varepsilon _S (0) <\infty $
(this assumption corresponds to $\mathscr{H}^{d-1} (M_0) <\infty$ for \eqref{vpmcf}). 
Then, we can expect that $\varphi ^\varepsilon (x,t) \approx 1$ or $-1$
when $x$ is outside the neighborhood of 
the zero level set 
$M_t ^\varepsilon=\{x \in \Omega \mid \varphi ^\varepsilon (x,t)=0 \}$
for sufficiently small $\varepsilon$.
Then we have 
$\int_{\Omega} k(\varphi ^\varepsilon )\, dx 
\approx 
\frac23 \int_{\Omega} \varphi ^\varepsilon \, dx$ and thus
we can regard \eqref{eq:1.5} as a relaxed volume preserving property.
The function $\sqrt{2W(\varphi ^\varepsilon)}$ 
expresses that the non-local term is almost zero when $x$ is outside the neighborhood of 
$M_t ^\varepsilon$.
In addition, $\sqrt{2W(\varphi ^\varepsilon)}$ plays important roles in 
$L^\infty$-estimates and energy estimates
(see Proposition \ref{prop3.2} and Theorem \ref{thm3.11}). 

\bigskip

The first main result of this paper is that there exists a global-in-time
weak solution to \eqref{vpmcf} for any $d\geq 2$
in the sense of $L^2$-flow, under the assumptions on the regularity of $M_0$
(see Theorem \ref{mainthm1}). We employ \eqref{ac} to construct the solution.
Note that we do not require assumptions such as \eqref{eq:1.3}. 
The second main result is that, when $M_0$ is $C^1$ and 
the value $\mathscr{H}^{d-1} (M_0)/ (\mathscr{L}^d (U_0))^{\frac{d-1}{d}}$ 
is sufficiently close to that of a ball,
there exists $T_1 >0$ such that the flow has a unit density for a.e. $t \in [0,T_1)$ and is also a
distributional BV-solution up to $t=T_1$ (see Theorem \ref{mainthm2}).
To obtain the main results, we need to prove that the varifold
$V_t ^\varepsilon$ defined by the Modica--Mortola functional \cite{MR473971} 
converges to a integral varifold for a.e. $t\geq 0$ (roughly speaking, the condition \eqref{eq:1.3}
corresponds to this convergence).
For the standard Allen--Cahn equation without non-local term,
this convergence was shown by Ilmanen~\cite{ilmanen1993} and Tonegawa~\cite{MR2040901}.
Therefore we can expect the convergence for \eqref{ac} if 
$\lambda ^\varepsilon$ has suitable properties.
In fact, $\lambda^\varepsilon$ can be regarded as an error term 
when we consider the parabolic rescaled equation of \eqref{ac}. We explain this more precisely.
Define $\tilde \varphi ^\varepsilon (\tilde x,\tilde t) = \varphi ^\varepsilon (\varepsilon \tilde x, \varepsilon ^2 \tilde t)$.
Then $\tilde \varphi ^\varepsilon$ satisfies
\begin{equation}
\tilde \varphi ^\varepsilon _{\tilde t} = \Delta _{\tilde x} \tilde \varphi ^\varepsilon - W' (\tilde \varphi ^\varepsilon)
+\varepsilon \lambda ^\varepsilon (\varepsilon ^2 \tilde t) \sqrt{2W(\tilde \varphi ^\varepsilon)},
\label{eq:1.6}
\end{equation}
where 
$\Delta _{\tilde x} $ is a Laplacian with respect to $\tilde x$.
Assume $\sup_{x} \vert \varphi ^\varepsilon _0 (x) \vert < 1$. Then Proposition \ref{prop3.1} below yields
$\sup_{x,t} \vert \varphi ^\varepsilon (x,t) \vert < 1$. Thus we have
\begin{equation}
\sup _{\tilde t \geq 0} \lvert \varepsilon \lambda ^\varepsilon (\varepsilon ^2 \tilde t)  
\sqrt{2W(\tilde \varphi ^\varepsilon)} \rvert
\leq
\sup _{\tilde t \geq 0} \lvert \varepsilon \lambda ^\varepsilon (\varepsilon ^2 \tilde t) \rvert
\leq \frac{4}{3} \mathscr{L}^{d} (\Omega) \varepsilon ^{1-\alpha}
=\frac{4}{3} \varepsilon ^{1-\alpha},
\label{eq:1.7}
\end{equation}
where we used $\max_{s \in [-1,1]} \vert k (s) \vert =\frac23$. Therefore, broadly speaking, 
the non-local term 
$\varepsilon \lambda ^\varepsilon (\varepsilon ^2 \tilde t)\sqrt{2W(\tilde \varphi ^\varepsilon)}$
is a perturbation (to the best of our knowledge, for \eqref{vpmcf}, 
no phase field model with such a property has been known).
Hence we can show the rectifiability and the integrality of the varifold $V_t$
with arguments similar to that in \cite{ilmanen1993, MR2040901} 
(see also \cite{takasao-tonegawa}). However,  
the proofs are not exactly the same as those, 
because the monotonicity formula for \eqref{ac} 
is different from the standard one (see Proposition \ref{prop3.8}). 
Therefore we give the proofs in Section 4.
In addition, as another good property of $\lambda ^\varepsilon$, the $L^2$-norm  
can be controlled (see Lemma \ref{lem3.2}). This property is useful when proving
the monotonicity formula and the rectifiability of $V_t$. 

\bigskip

The most well-known phase field model for \eqref{vpmcf} 
studied by Rubinstein and Sternberg~\cite{RubinsteinSternberg}
is the following equation.
\begin{equation}
\left\{ 
\begin{array}{ll}
\varepsilon \varphi ^{\varepsilon} _t =\varepsilon \Delta \varphi ^{\varepsilon} -\dfrac{W' (\varphi ^{\varepsilon})}{\varepsilon }+ \Lambda ^{\varepsilon},
& (x,t)\in \Omega \times (0,T),  \\
\varphi ^{\varepsilon} (x,0) = \varphi _0 ^{\varepsilon} (x) ,  &x\in \Omega,
\end{array} \right.
\label{rs}
\end{equation}
where $\Lambda ^\varepsilon (t) = 
\frac{1}{\mathscr{L}^d (\Omega)}\int _\Omega \frac{W' (\varphi ^{\varepsilon} (x,t) )}{\varepsilon } \, dx$.
As above, the solution to \eqref{rs} has the volume preserving property 
$ \frac{d}{dt} \int _{\Omega} \varphi ^\varepsilon \, dx =0 $.
Chen, Hilhorst, and Logak~ \cite{MR2754215} proved that
for the smooth solution $\{M_t\} _{t \in [0,T)}$ to \eqref{vpmcf}, there exists a
family of functions $\{ \varphi _0 ^{\varepsilon _i } \}_{i =1 } ^\infty$ with $\varepsilon_i \to 0$
such that the level set $M_t ^{\varepsilon_i} = \{ x \in \Omega \mid \varphi^{\varepsilon_i} (x,t)=0 \}$
converges to $M_t$, where $\varphi^\varepsilon$ is a solution to \eqref{rs} with initial data $\varphi^\varepsilon _0$.
In addition, as mentioned above, Laux and Simon~\cite{MR3847750} proved the convergence of the vector-valued version of \eqref{rs} to 
the weak volume preserving multiphase mean curvature flow under an assumption
corresponds to \eqref{eq:1.3}.
However, it is an open problem to show its convergence for \eqref{rs} without such assumptions.
One of the difficulties is that the boundedness of 
$\sup _{\varepsilon >0} \int_0 ^T \vert \Lambda ^\varepsilon \vert ^2 \, dt$ proved by 
Bronsard and Stoth~\cite{bronsard-stoth}
does not immediately lead to
\begin{equation}\label{eq:1.10}
\sup _{\varepsilon >0} \int _0 ^T \int _{\Omega} \varepsilon 
\left( \Delta \varphi ^{\varepsilon} -\dfrac{W' (\varphi ^{\varepsilon})}{\varepsilon^2 } \right)^2 
\,dxdt <\infty.
\end{equation}
Note that \eqref{eq:1.10} corresponds to 
$\int _0 ^T \int _{M_t} \vert \vec{h} \vert ^2 \, d\mathscr{H}^{d-1} dt <\infty$ of the solution to \eqref{vpmcf}
and is important to show the rectifiability of the varifold
(see Theorem \ref{thm3.11} and Theorem \ref{thm4.6}). 
As another phase field method for \eqref{vpmcf}, the study of Brassel and Bretin~\cite{brassel-bretin} is known (see \cite[Section 1]{takasao2017} for a comparison of these equations).

\bigskip

The organization of this paper is as follows.
In Section 2, we set our notations and state the main results.
In Section 3, to obtain the existence theorem
we prove the energy estimates and $L^\infty$-estimates
for the solution to \eqref{ac}. In addition, for $d=2$ or $3$, we give a short proof for the
integrality of the limit measure $\mu_t$ constructed as a weak solution to \eqref{vpmcf}.
In Section 4, we show the integrality of $\mu_t$ for any $d \geq 2$.
In Section 5, we prove the main results.
In Section 6, we give some supplements for this paper.

\section{Preliminaries and main results}
\subsection{Notations and definitions}
For $r>0$, $d\in \N$, and $x \in \R^d$, we denote 
$B_r ^d (x) := \{ y \in \R^d \mid \vert x-y \vert <r \}$ (we often write this as $B_r (x)$ for simplicity). 
We define $\omega _d := \mathscr{L}^d (B_1 ^d (0))$.
For $d\times d$ matrix $A =(a_{ij})$ and $B=(b_{ij})$, we define
$
A\cdot B:= \sum_{i,j} a_{ij} b_{ij}.
$
For $a=(a_1 ,a_2,\dots, a_d) \in \R^d$, we define a $d\times d$ matrix $a\otimes a$
by $a\otimes a := (a_i a_j)$.
Next we recall notations and definitions from the geometric measure theory
and refer to \cite{MR3409135, giusti, simon, MR2040901} for more details.
For a Caccioppoli set $E \subset \R^d$, 
we denote the reduced boundary of $E$ by $\partial ^\ast E$.
For the characteristic function $\chi _{E}$, we denote
the total variation measure of the distributional derivative
$\nabla \chi _E$ by $\| \nabla \chi _E \|$.
Let $U \subset \mathbb{R}^d$ be an open set. We write
the space of bounded variation functions on $U$ as $BV(U)$.
For any Radon measure $\mu$ on $U$ and $\phi \in C_c (U)$, 
we often write $\int \phi \, d\mu$ as $\mu (\phi)$. For $p\geq 1$, 
we write $f \in L^p (\mu)$ if $f$ is $\mu$-measurable and $\int \vert f \vert^p \, d\mu <\infty$.
For $d,k \in\N$ with $k<d$, 
let $\mathbb{G}(d,k)$ be the space of $k$-dimensional subspace of $\R^d$. 
For an open set $U \subset \R^d$, let $G_k (U) := U \times \mathbb{G} (d,k)$.
We say $V$ is a general $k$-varifold on $U$ if $V$ is a Radon measure  on $G_k (U)$.
We denote the set of all general $k$-varifolds on $U$ by 
$\mathbb{V} _k (U)$.
For a general varifold $V \in \mathbb{V} _k (U)$, we define the weight measure $\| V\|$ by
\[
\| V \| (\phi) := \int _{G_k (U)} \phi (x) \, dV (x,S) \qquad \text{for any} \ \phi \in C_c (U). 
\]
We call $V \in \mathbb{V} _k (U)$ is rectifiable if there exist a $\mathscr{H}^k$-measurable
$k$-countably rectifiable set $M \subset U$ and $\theta \in L_{loc} ^1 (\mathscr{H}^{k}\lfloor_{M})$
such that
\[
V(\phi) =\int _M \phi (x, T_x M) \theta (x) \, d \mathscr{H}^k
\qquad \text{for any} \ \phi \in C_c (G_k(U)),
\]
where $T_x M$ is the approximate tangent space of $M$ at $x$.
Note that such $x$ does exists for $\mathscr{H}^k$-a.e. on $M$.
If $\theta \in \N$ $\mathscr{H}^k$-a.e. on $M$, we call $V$ is integral.
In addition, if $\theta =1$ $\mathscr{H}^k$-a.e. on $M$, we say 
$V$ has unit density.

For $V \in \mathbb{V} _k (U)$, we define the first variation $\delta V$ by
\[
\delta V ( \vec{\phi} ) := \int _{G_k(U)} \nabla \vec \phi (x) \cdot S \,
dV (x,S) \qquad \text{for any} \ \vec \phi \in C_c ^1 (U;\R^d).
\]
Here, we identify $S \in \mathbb{G} (d,k)$ with the corresponding
orthogonal projection of $\R^d$ onto $S$.
When the total variation $\| \delta V \|$ of $\delta V$
is locally bounded and absolutely continuous with respect to $\| V \|$,
there exists a measurable vector field $\vec h$ such that
\[
\delta V (\vec \phi ) = - \int _U \vec\phi (x) \cdot \vec h(x) \, d \| V \| (x)
\qquad \text{for any} \ \vec \phi \in C_c ^1 (U;\R^d).
\]
The vector valued function $\vec h$ is called the generalized mean curvature vector of $V$.
In addition, a Radon measure $\mu$ is called a $k$-rectifiable if 
there exists a $k$-rectifiable varifold
such that $\mu$ is represented by $\mu = \| V \|$.
Note that this $V$ is uniquely determined, so 
the first variation and the generalized mean curvature vector of $\mu$ 
is naturally determined by $V$.
The definition of an integral Radon measure is determined in the same way.

The formulation of the following is similar to that of the Brakke flow~\cite{brakke,MR3930606}.
\begin{definition}[$L^2$-flow \cite{MR2383536}]\label{defL2}
Let $T>0$, $U\subset \mathbb{R}^d$ be an open set, 
and $\{\mu_t\} _{t \in [0,T)}$ be a family of Radon measures on $U$.
Set $d\mu := d\mu_t dt$. We call $\{\mu_t\} _{t \in [0,T)}$
an $L^2$-flow with a generalized velocity vector $\vec{v}$
if  the following hold:
\begin{enumerate}
\item For a.e. $t \in (0,T)$, $\mu_t$ is $(d-1)$-integral, and also has a generalized mean curvature vector
$\vec{h} \in L^2 (\mu_t ; \mathbb{R}^d)$.
\item The vector field $\vec{v}$ belongs to $L^2 (0,T; (L^2 (\mu _t))^d)$ and 
$$
\vec{v}(x,t) \perp T_x \mu_t \quad 
\text{for} \ 
\mu\text{-a.e.} \ (x,t) \in U \times (0,T),
$$
where $T_x \mu_t \in \mathbb{G} (d,d-1)$ 
is the approximate tangent space of $\mu_t$ at $x$.
\item There exists $C_T>0$ such that
\begin{equation}\label{ineq-L2}
\left \vert \int _0 ^T \int _U (\eta _t + \nabla \eta \cdot \vec{v}) \, d\mu _t dt \right \vert
\leq C_T \| \eta \|_{C^0 (U\times (0,T))}
\end{equation}
for any $\eta \in C_c ^1 (U\times (0,T))$.
\end{enumerate} 
\end{definition}

\begin{remark}
If there exists a family of smooth hypersurfaces $\{M_t\}_{t \in [0,T)}$ with
the normal velocity vector $\vec{w}$, then \eqref{ineq-L2} holds
with $\vec{v} =\vec{w}$ and $\mu _t = \mathscr{H} ^{d-1} \lfloor_{M_t}$.
In addition, if $\vec{v}$ satisfies \eqref{ineq-L2} with $\mu _t = \mathscr{H} ^{d-1} \lfloor_{M_t}$,
then $\vec{v} =\vec{w}$.
This proof is almost identical to the proof in \cite[Proposition 2.1]{MR3930606}.
\end{remark}

The $L^2$-flow has the following property.

\begin{proposition}[See Proposition 3.3 of \cite{MR2383536}]\label{prop2.3}
Assume that $\{ \mu_t \} _{t \in (0,T)}$ is an $L^2$-flow with the 
generalized velocity vector $\vec{v}$ and set $d\mu :=d\mu_t dt$. Then
\[
(\vec{v}(x_0,t_0),1) \in T_{(x_0,t_0)} \mu
\]
at $\mu$-a.e. $ (x_0,t_0) \in \Sigma (\mu)$, 
where $T_{(x_0,t_0)} \mu \in \mathbb{G} (d+1,d)$ is the 
approximate tangent space of $\mu$ at $(x_0, t_0)$ 
and 
$\Sigma(\mu)= \{ (x,t) \mid T_{(x,t)} \mu \ 
\text{exists at} \ (x,t) \}$.
\end{proposition}

\subsection{Assumptions for initial data}
Let $U _0 \subset \subset (0,1)^d$ be a bounded open set with the following properties.
\begin{enumerate}
\item There exists $D_0 >0$ such that
\begin{equation}\label{U01}
\sup _{x \in (0,1) ^d , 0<R <1} \frac{\mathscr{H}^{d-1} (M_0 \cap B_r (x))}{\omega _{d-1} r^{d-1}}
\leq D_0,
\end{equation}
where $M_0 = \partial U_0$.
\item There exists a family of open sets $\{ U_0 ^i\}_{i=1} ^\infty$ such that
$U_0 ^i$ has a $C^3$ boundary 
$M_0 ^i = \partial U_0 ^i$ for any $i $ and the following hold:
\begin{equation}\label{U02}
\lim_{i\to \infty}
\mathscr{L}^d (U_0 \triangle U_0 ^i)=0
\qquad
\text{and}
\qquad
\lim_{i\to \infty}
\| \nabla \chi _{U_0 ^i} \|= \| \nabla \chi _{U_0} \|
\quad \text{as Radon measures}.
\end{equation}
\end{enumerate}
Note that the second assumption is satisfied when $U_0$ is a Caccioppoli set,
and both conditions are fulfilled when $M_0$ is $C^1$
(see \cite{giusti}).

\bigskip

We denote $q^\varepsilon (r) := \tanh (r/\varepsilon)$ for $r \in \R$.
Then $q^\varepsilon$ satisfies
\begin{equation}
\frac{\varepsilon (q_r ^\varepsilon (r)) ^2 }{2} 
= \frac{W(q^\varepsilon(r) )}{\varepsilon} \qquad \text{for any} \ r \in \R
\label{eq:3.16}
\end{equation}
and
\begin{equation}
q_{rr} ^\varepsilon (r)  
= \frac{W' (q^\varepsilon(r) )}{\varepsilon ^2} \qquad \text{for any} \ r \in \R.
\label{eq:3.17}
\end{equation}
In addition, \eqref{eq:3.16} yields
\begin{equation*}
\int _{\R} \left( \frac{\varepsilon (q_r ^\varepsilon (r)) ^2 }{2} 
+ \frac{W(q^\varepsilon(r) )}{\varepsilon} \right) \, dr
= \int_{\R} \sqrt{ 2W (q^\varepsilon) } q^\varepsilon _r \, dr
= \int _{-1} ^1 \sqrt{ 2W (q) } \, dq =: \sigma.
\end{equation*}
This means that the Radon measure $\mu_t ^\varepsilon$ defined below
needs to be normalized by $\sigma$.

Next we extend $U_0 ^i $ and $M_0 ^i $ periodically to $\R^d$ with period $\Omega$
and define
\begin{equation*}
r _i (x) = \left\{ 
\begin{array}{ll}
\dist (x, M_0 ^i), & \text{if} \ x \in U_0 ^i, \\
-\dist (x,M_0 ^i) , & \text{if} \  x \not \in U_0 ^i.
\end{array} \right.
\end{equation*}
Then $\vert \nabla r _i (x) \vert \leq 1$ for a.e. $x \in \R ^d$
and there exists $b_i >0$ such that
$r_i $ is $C^3$ on 
$N_{b_i} := \{ x \mid \dist (x, M_0 ^i)<b_i \}$ (see \cite{MR1279299}).
Let $d_i$ be a smooth monotone increase function such that
\begin{equation*}
d _i (r) = \left\{ 
\begin{array}{ll}
r, & \text{if} \  \vert r \vert < \frac14 b_i, \\
\frac23 b_i , & \text{if} \  r >\frac34 b_i\\
-\frac23 b_i , & \text{if} \  r <-\frac34 b_i
\end{array} \right.
\end{equation*}
and $\vert \frac{d}{dr}d_i  \vert \leq 1$.
Set $\overline{r_i} := d_i (r_i) $. Then $\overline{r_i} \in C^3 (\Omega)$, 
$\overline{r_i} =r_i$ on $N_{b_i/4}$, and
$\vert \nabla \overline{r_i} (x) \vert \leq 1$ for any $x \in \R ^d$. 
Let $\{ \varepsilon _i \}_{i=1} ^\infty$
be a positive sequence with $\varepsilon _i \to 0$
and $\displaystyle \frac{ \varepsilon _i }{b_i ^2} \to 0$ as $i \to \infty$,
and 
\begin{equation}\label{eq:2.6}
\sup _{x \in \R^d} \vert \nabla ^{j+1} \overline{r_i} (x) \vert \leq \varepsilon _i ^{-j}
\qquad \text{for any} \ i \in \N,  \ j=1,2.
\end{equation}
Note that \eqref{eq:2.6} corresponds to the condition \eqref{initial} below.
We define a periodic function $\varphi _0 ^{\varepsilon _i} \in C^3 (\Omega) $ by
\begin{equation}\label{eq:2.7}
\varphi _0 ^{\varepsilon _i} (x) := q ^{\varepsilon _i} ( \overline{r_i} (x) )
=\tanh \left( \frac{d_i (r_i (x))}{\varepsilon _i} \right)
\qquad \text{for any} \ i  \in \N.
\end{equation}
We define a Radon measure $\mu _t ^{\varepsilon_i}$ by
\begin{equation}
\mu_t ^{\varepsilon_i} (\phi)
:=
\frac{1}{\sigma}\int _{\Omega} \phi
\left( \frac{\varepsilon_i \vert \nabla \varphi ^{\varepsilon_i} (x,t) \vert ^2}{2} 
+ \frac{W(\varphi ^{\varepsilon_i} (x,t))}{\varepsilon_i} \right) \, dx, \qquad \phi \in C_c (\Omega), 
\label{mu}
\end{equation}
where $\varphi ^{\varepsilon_i}$ is the solution to \eqref{ac} 
with initial data $\varphi _0 ^{\varepsilon _i}$ defined by \eqref{eq:2.7}
and $\sigma = \int _{-1} ^1 \sqrt{2W(s)} \, ds$.

For $\varphi ^{\varepsilon _i} _0$ and $\mu _0 ^{\varepsilon _i}$, we have the following properties
(see \cite[p.\,423]{ilmanen1993} and \cite[Section 5]{liu-sato-tonegawa}).
\begin{proposition}\label{prop2.2}
There exists a subsequence $\{ \varepsilon _i \}_{i=1} ^\infty$ (denoted by the same index
and the subsequence is taken only for $\{ \varepsilon _i\}_{i=1} ^\infty $, not for $\{ M_0 ^i\} _{i=1} ^\infty$)
such that the following hold.
\begin{enumerate}
\item For any $i \in \N$ and $x \in \Omega$, we have
$\displaystyle
\frac{\varepsilon_i \vert \nabla \varphi^{\varepsilon_i} _0 (x) \vert ^2 }{2} 
\leq \frac{W(\varphi ^{\varepsilon_i} _0 (x))}{\varepsilon_i} 
$.
\item
There exists $D_1= D_1 (D_0) >0$ such that
\begin{equation}
\max \left\{
\sup _{i \in \N} \mu _0 ^{\varepsilon_i} (\Omega) 
, \sup _{i \in \N, \ x \in \Omega, \ r \in (0,1)} 
\frac{\mu _0 ^{\varepsilon_i} (B_r (x))}{\omega ^{d-1} r^{d-1}}
\right\}
\leq D_1.
\label{d1}
\end{equation}
\item $\mu _0 ^{\varepsilon_i} \to \mathscr{H}^{d-1} \lfloor _{M_0}$ as Radon measures, that is,
\[
\int _{\Omega} \phi \, d \mu _0 ^{\varepsilon _i} \to \int _{M_0} \phi \, d\mathscr{H}^{d-1}
\qquad 
\text{for any} \ \phi \in  C_c(\Omega).
\]
\item For $\psi ^{\varepsilon _i} =\frac12 ( \varphi ^{\varepsilon _i} +1)$,
$\lim _{i\to \infty} \psi ^{\varepsilon _i} = \chi _{U_0} $ in $L^1$
and $\lim _{i\to \infty} \| \nabla \psi ^{\varepsilon _i } \| = \| \nabla \chi _{U_0} \| $
as Radon measures.
\end{enumerate}
\end{proposition}
\begin{remark}
The first property (1) is obtained from $ \vert \nabla \overline{r_i}  \vert \leq 1$
(see the proof of Proposition \ref{prop3.2}).
The assumption $\frac{ \varepsilon _i }{b_i ^2} \to 0$
is used to show 
$\int _{\Omega \setminus N_{b_i /4} } 
\left( \frac{\varepsilon_i \vert \nabla \varphi ^{\varepsilon_i} _0 \vert^2}{2} 
+ \frac{W(\varphi ^{\varepsilon_i} _0 )}{\varepsilon_i} \right) \, dx\to 0
$.
\end{remark}

\subsection{Main results}
We denote the approximate velocity vector $\vec{v} ^{\, \varepsilon _i}$ by
\begin{equation*}
\vec{v} ^{\, \varepsilon _i} = \left\{ 
\begin{array}{ll}
\dfrac{- \varphi ^{\varepsilon _i} _t }{\vert \nabla \varphi ^{\varepsilon _i} \vert }
\dfrac{\nabla \varphi ^{\varepsilon _i}}{\vert \nabla \varphi ^{\varepsilon _i} \vert }, & \text{if} \ 
\vert \nabla \varphi ^{\varepsilon _i} \vert \not=0, \\
\qquad 0 , & \text{otherwise}.
\end{array} \right.
\end{equation*}

The first main result of this paper is the following.
\begin{theorem}\label{mainthm1}
Suppose that $d\geq 2$ and $U_0$ satisfies \eqref{U01} and \eqref{U02}.
For any $i \in \N$, let $\varphi _0 ^{\varepsilon _i}$ be defined so that all the claims of Proposition \ref{prop2.2} are satisfied
and
$\varphi ^{\varepsilon_i} $ be a solution to
\eqref{ac} with initial data $\varphi _0 ^{\varepsilon _i}$.
Then there exists a subsequence $\{ \varepsilon _i \}_{i=1} ^\infty$ (denoted by the same index)
such that the following hold.
\begin{enumerate}
\item[(a)] There exist a countable subset $B \subset [0,\infty)$ and
a family of $(d-1)$-integral Radon measures $\{\mu _t\}_{t \in [0,\infty)}$
on $\Omega$ such that
\[
\mu _0 = \mathscr{H}^{d-1} \lfloor_{M_0}, 
\qquad 
\mu _t ^{\varepsilon _i} \to \mu _t \ \ \ \text{as Radon measures for any} \ t\geq 0,
\]
and
\[
\mu _s (\Omega) \leq \mu _t (\Omega) \qquad \text{for any} \ s,t \in [0,\infty) \setminus B 
\ \text{with} \  0 \leq t <s <\infty. 
\]
\item[(b)]
There exists $\psi \in BV_{loc} (\Omega \times [0,\infty)) \cap C_{loc} ^{\frac12} ([0,\infty); L^1 (\Omega))$
such that the following hold.
\begin{enumerate}
\item[(b1)] $\psi ^{\varepsilon _i} \to \psi $ in $L^1 _{loc} (\Omega \times [0,\infty))$
and a.e. pointwise, where $\psi ^{\varepsilon _i}=\frac12 (\varphi^{\varepsilon _i} +1)$.
\item[(b2)] $\psi \vert_{t=0}=\chi _{U_0}$ a.e. on $\Omega$.
\item[(b3)] For any $t \in [0,\infty )$, $\psi (\cdot ,t) =1$ or $0$ a.e. on $\Omega$ and
$\psi$ satisfies the volume preserving property, that is, 
\[
\int _{\Omega} \psi (x,t) \, dx  = \mathscr{L}^{d} (U_0)
\qquad \text{for all} \ t \in [0,\infty).
\]
\item[(b4)] For any $t \in [0,\infty)$ and for any $\phi \in C_c (\Omega; [0,\infty))$, 
we have $\| \nabla \psi (\cdot ,t) \| (\phi ) \leq \mu _t (\phi)$.
\end{enumerate}
\item[(c)] For $\lambda ^{\varepsilon _i}$ given by \eqref{lambdadef}, we have
\[
\sup_{i \in \N} \int _0 ^T \vert \lambda ^{\varepsilon_i} \vert^2 \, dt <\infty \qquad \text{for any} \ T>0
\]
and there exists $\lambda \in L_{loc} ^2 (0,\infty)$ such that
$\lambda ^{\varepsilon _i} \to \lambda$ weakly in $L^2 (0,T)$ for any $T>0$.
\item[(d)] 
There exists $\vec{f} \in L_{loc} ^2 ([0,\infty); (L^2 (\mu _t))^d)$ such that
\begin{equation}\label{claim-e2}
\begin{split}
&\lim_{i \to \infty}
\frac{1}{\sigma} \int_0 ^\infty \int _{\Omega }
-\lambda ^{\varepsilon _i} \sqrt{2W (\varphi ^{\varepsilon _i})}
\nabla \varphi ^{\varepsilon _i} \cdot \vec{\phi} \, dxdt \\
= & \int_{0} ^\infty \int _{\Omega} \vec{f} \cdot \vec{\phi} \, d \mu_t dt 
= 
\int_{0} ^\infty \int _{\Omega} -\lambda \vec{\nu}\cdot \vec{\phi} 
\, d \| \nabla \psi (\cdot,t) \| dt
\end{split}
\end{equation}
for any $\vec{\phi} \in C_c (\Omega \times [0,\infty); \R^d)$,
where $\vec \nu$ is the inner unit normal vector of $\{ \psi (\cdot,t) =1 \}$
 on $\spt \| \nabla \psi (\cdot,t) \|$.

\item[(e)] The family of Radon measures $\{ \mu _t \}_{t \in [0,\infty)}$
is an $L^2$-flow with a generalized velocity vector $\vec{v} = \vec{h} +\vec{f}$,
where $\vec{h} \in L_{loc} ^2 ([0,\infty) ; (L ^2 (\mu_t) )^d) $ 
is the generalized mean curvature vector of $\mu _t$.
Moreover, for any $\vec{\phi} \in C_c (\Omega \times [0,\infty); \R^d)$,
\begin{equation}\label{claim-e}
\lim_{i \to \infty}
\int_0 ^\infty \int _{\Omega }
\vec{v}^{ \, \varepsilon _i} \cdot \vec{\phi} \, d\mu _t ^{\varepsilon _i}dt
= \int_{0} ^\infty \int _{\Omega} \vec{v} \cdot \vec{\phi} \, d \mu_t dt.
\end{equation}
\end{enumerate}

\end{theorem}

\begin{remark}\label{rem2.6}
From (d) and (e), we have
\[
\int_{0} ^\infty \int _{\Omega} \vec{v} \cdot \vec{\phi} \, d \mu_t dt
=\int_{0} ^\infty \int _{\Omega} \left( \vec{h} 
- \lambda \frac{d \| \nabla \psi (\cdot,t) \|}{d\mu _t} \vec{\nu} \right) \cdot \vec{\phi} \, d \mu_t dt
\]
for any $\vec{\phi} \in C_c (\Omega \times (0,\infty); \R^d)$, 
where $\frac{d \| \nabla \psi (\cdot,t) \|}{d\mu _t}$ is the Radon--Nikodym derivative.
Hence we have $\vec{v} = \vec{h} -\lambda \vec{\nu}$ in the sense of $L^2$-flow
if $\mu _t = \| \nabla \psi (\cdot,t) \|$ for a.e. $t$ 
(from Theorem \ref{mainthm2} below, this is correct for a short time if the initial data
is sufficiently close to a ball).
Since $\mu_t$ is integral for a.e. $t$, for such $t$, we have
$\left( \frac{d \| \nabla \psi (\cdot,t) \|}{d\mu _t} \right) ^{-1} \in \N$ for $\mu_t$-a.e. $x \in \Omega$,
if
$\frac{d \| \nabla \psi (\cdot,t) \|}{d\mu _t}\not =0$.
\end{remark}

\bigskip

Set  
$U_t := \{ x \in \Omega \mid \psi (x,t)=1 \}$ for $t>0$.
Let $B \subset \subset (0,1) ^d$ be an open ball.
We also show that if $U_0 \approx B$ in the following sense,
then there exists $T_1 >0$ such that 
$\{\partial ^\ast U _t\} _{t \in [0,T_1)}$ is a distributional solution to \eqref{vpmcf} 
in the framework of BV functions.

\begin{theorem}\label{mainthm2}
For any $r \in (0,\frac{1}{4})$, there exists $\delta_1 >0$ depending only on $d$ and $r$ 
with the following properties.
Assume that $ U_0 \subset (\frac14 ,\frac34)^d $ 
satisfies $\mathscr{L}^d (U_0)=\mathscr{L}^d (B_{r} (0))$ 
and has a $C^1$ boundary $M_0$ with 
$\mathscr{H}^{d-1} (M_0) \leq 2 \mathscr{H}^{d-1} (\partial B_{r} (0))$ and
\begin{equation}\label{iso}
\mathscr{H}^{d-1} (M_0) - d \omega _d ^{\frac{1}{d}} (\mathscr{L}^d (U_0))^{\frac{d-1}{d}} \leq \delta_1.
\end{equation}
Then there exists $T_1 =T_1(d,r,M_0) >0$ such that the following hold.
\begin{enumerate}
\item[(a)] For a.e. $t \in [0,T_1)$, 
$
\mu _t = \| \nabla \psi (\cdot,t) \| = 
\mathscr{H}^{d-1} \lfloor_{\partial ^{\ast} U_t}
$, where $\{ \mu _t\} _{t \in [0,\infty)}$ is the $L^2$-flow 
with initial data $\mu_0 =\mathscr{H}^{d-1}\lfloor_{M_0}$, given by Theorem \ref{mainthm1}.
\item[(b)] 
Let $\vec{v}$, $\vec{h}$, $\vec{\nu}$, and $\lambda$ be functions given by Theorem \ref{mainthm1}.
Then $\{ \partial ^\ast U_t \}_{t \in [0,T_1)}$ is a distributional solution to \eqref{vpmcf} with initial data $\partial U_0 = M_0$
in the following sense.
\begin{enumerate}
\item[(b1)] For any $t \in [0,T_1)$, $\mathscr{L}^d (U_t) = \mathscr{L}^d (U_0)$.
\item[(b2)] For a.e. $t\in [0,T_1)$, $\vec{h}$ is
also a generalized mean curvature vector
of $\mathscr{H}^{d-1} \lfloor _{\partial^\ast U_t}$.
\item[(b3)] 
For any $\vec{\phi} \in C_c (\Omega \times [0,T_1);\R ^d)$, we have
\[
\int _0 ^{T_1} \int _{\partial ^\ast U_t} \{ \vec{v} -\vec{h} + \lambda \vec{\nu} \} \cdot \vec{\phi} \, 
d\mathscr{H} ^{d-1} dt =0.
\]
\item[(b4)] 
For any
$\phi \in C_c ^1 (\Omega \times (0,T_1))$, we have
\[
\int _0 ^{T_1} \int _{U_t} \phi _t \, dxdt 
= \int _0 ^{T_1} \int _{\partial ^\ast U_t} \vec{v} \cdot \vec{\nu} \phi \, d\mathscr{H}^{d-1} dt.
\]
\item[(b5)] (Additional volume preserving property). For a.e. $t \in [0,T_1)$, we have
\[
\int _{\partial ^\ast U_t} \vec{v} \cdot \vec{\nu} \, d \mathscr{H}^{d-1}
=
\int _{\Omega} \vec{v} \cdot \vec{\nu} \, d \| \nabla \psi (\cdot ,t)\| =0.
\]
\item[(b6)] For a.e. $t \in [0,T_1)$, we have
\[
\lambda (t) = \frac{1}{\mathscr{H}^{d-1} (\partial ^\ast U_t)} 
\int _{\partial ^\ast U_t} \vec{h} \cdot \vec{\nu} \, d \mathscr{H}^{d-1}.
\]
\end{enumerate}
\end{enumerate}
\end{theorem}

\begin{remark}
The isoperimetric inequality tells us that
$d \omega _d ^{\frac{1}{d}} (\mathscr{L}^d (U))^{\frac{d-1}{d}} \leq \mathscr{H}^{d-1} (\partial^\ast U) $
for any Caccioppoli set $U \subset \R^d$ with $\mathscr{L}^d (U) <\infty$ 
and the equality holds if and only if 
there exists a ball $B \subset \R^d$ such that $\mathscr{L}^d (U \triangle B )=0$
(see \cite{MR2147710,MR1242977} and references therein).
Therefore the assumption \eqref{iso} means that $U_0$ is sufficiently close to
a ball in some sense.
On the other hand, $M_0$ does not have to be close to a sphere in $C^0$
(for example, $U_0$ does not have to be connected).
\end{remark}
\begin{remark}
The property (b4) claims that $\vec v$ is a normal velocity vector in a weak sense,
since
\[\frac{d}{dt} \int _{U_t} \phi \, dx
= \int _{U_t} \phi_t \, dx - \int _{\partial U _t} \vec v\cdot \vec \nu \phi \, d\mathscr{H}^{d-1} 
\]
holds for any $\phi \in C_c ^1 (\Omega \times (0,T_1))$ and $t \in (0,T_1)$, where 
$\{ U_t \}_{t\in [0,T_1)}$ is a family of open sets and the smooth boundary 
$\partial U_t$ moves by the normal velocity vector $\vec{v}$.
By \eqref{vpproperty}, we can regard (b5) as a volume preserving property
in a weak sense.
\end{remark}


\section{Energy and pointwise estimates}
In this section we show standard estimates for \eqref{ac} such as 
the uniform $L^2$-estimate for $\lambda ^\varepsilon$ and the monotonicity formula.
\subsection{Assumptions}
Let $\{ \varepsilon _i \}_{i=1} ^\infty$
be a positive sequence with $\varepsilon _i \to 0$ as $i \to \infty$.
In this section, we assume that there exist $D_1>0$ and $\omega >0$ such that
\eqref{d1}
and
\begin{equation}
\frac{2}{3} - \left \vert  \int_{\Omega} k(\varphi ^{\varepsilon_i} _0 (x))\, dx \right \vert > \omega >0,
\label{omega}
\end{equation}
hold for any $i \in \N$. The set $\{ x \in \Omega \mid \varphi ^{\varepsilon_i} _0 (x) = 0 \}$ corresponds to the initial data $M_0$
of \eqref{vpmcf}, and \eqref{omega} yields that 
$\mathscr{L}^d (\{ x \in \Omega \mid \varphi ^{\varepsilon_i} _0 \approx 1 \}) >0$ formally,
since $\int_{\Omega} k(\pm 1)\, dx = \pm \frac23 \mathscr{L}^d(\Omega) =\pm \frac23 $.
For some $\Cl{const:initial} >0$, we also assume that the initial data $\varphi _0 ^{\varepsilon_i}$ of the solution to \eqref{ac} satisfies
\begin{equation}
\varphi ^{\varepsilon_i} _0 \in C^3 (\Omega), \quad \sup _{x \in \Omega} 
\vert \varphi ^{\varepsilon_i} _0 (x) \vert <1,
\quad \text{and} \quad \varepsilon_i ^j \sup _{x \in \Omega} \vert \nabla ^j \varphi ^{\varepsilon_i} _0 (x) \vert \leq \Cr{const:initial}
\label{initial}
\end{equation}
for any $i \in \N$ and $j=1,2,3$. In addition, 
to control the discrepancy measure $\xi _t ^\varepsilon$ defined below, we assume
\begin{equation}
\frac{\varepsilon_i \vert \nabla \varphi^{\varepsilon_i} _0 (x) \vert ^2 }{2} 
\leq \frac{W(\varphi ^{\varepsilon_i} _0 (x))}{\varepsilon_i} \qquad \text{for any} \ x \in \Omega
\ \text{and} \
i\in\N.
\label{eq:3.14}
\end{equation}
Note that the function $\varphi _0 ^{\varepsilon _i}$ defined by \eqref{eq:2.7} 
satisfies all the assumptions above, for sufficiently large $i$.
Throughout this paper, we often write $\varepsilon$ as $\varepsilon_i$ for simplicity.

\subsection{Pointwise estimates}
The comparison principle implies the following estimate.
\begin{proposition}\label{prop3.1}
The solution $\varphi ^\varepsilon$ to \eqref{ac} with \eqref{initial} satisfies 
\begin{equation}
\vert \varphi ^\varepsilon (x,t) \vert <1,\qquad x \in \Omega , \ t\geq 0.
\label{eq:3.4}
\end{equation}
\end{proposition}
\begin{remark}\label{rem3.3}
The estimate \eqref{eq:3.4} implies $\sqrt{2W (\varphi ^\varepsilon)} =1-(\varphi ^\varepsilon)^2 $.
By a priori estimates including Proposition \ref{prop3.2} below, 
standard PDE theories imply the global existence and uniqueness of the classical solution to \eqref{ac}
with initial data $\varphi ^\varepsilon _0$ satisfying \eqref{initial}.
\end{remark}

\begin{proof}
Suppose that 
$t_0:= \inf \{ t \in [0,\infty) \mid \sup_{x \in \Omega} \varphi ^\varepsilon (x,t) \geq 1 \} <\infty$.
Then $t_0 > 0 $ since $\sup _{x \in \Omega} \varphi ^\varepsilon _0 (x) <1 $. 
We may assume that there exists $t_1 \in (t_0, \infty)$ such that
$\sup_{x \in \Omega} \varphi ^\varepsilon (x,t) \leq 2$ for any $t <t_1$.
Let $\varphi ^\varepsilon _+ $ be a solution to
\begin{equation}\label{eq:super}
\varepsilon (\varphi ^{\varepsilon} _+)_t =\varepsilon \Delta \varphi_+ ^{\varepsilon} -\dfrac{W' (\varphi ^{\varepsilon} _+ )}{\varepsilon }+ L^\varepsilon \sqrt{2W(\varphi ^\varepsilon _+)} ,
\qquad (x,t) \in \Omega \times (0,t_1)
\end{equation}
with initial data $\varphi ^\varepsilon _+ (x,0) = \sup _{x \in \Omega} \varphi ^\varepsilon _0 (x) $,
where $L^\varepsilon := 2 \varepsilon ^{-\alpha} \max _{ \vert s \vert \leq 2} \vert k(s) \vert $.
Note that $\sup_{ t \in (0, t_1)} \vert \lambda^\varepsilon (t) \vert \leq L^\varepsilon $,
where $\lambda^\varepsilon$ is given by the solution $\varphi ^\varepsilon$ to \eqref{ac},
and this implies that $\varphi ^{\varepsilon} _+$ is a supersolution to \eqref{ac} if we regard
$\lambda^\varepsilon$ as a given function.
Since the initial data is constant and 
$W'(s), \sqrt{2W(s)} \to 0$ as $s\to 1$, 
one can easily check that the solution $\varphi_+ ^\varepsilon$ to
\eqref{eq:super} depends only on $t$ and satisfies $\varphi _+ ^\varepsilon (t) <1$ for any $t \in  (0,t_1)$.
Therefore the comparison principle implies that 
$\varphi ^\varepsilon (x,t) \leq \varphi _+ ^\varepsilon (t) <1$ for any $(x,t) \in \Omega \times (0,t_1)$.
This yields a contradiction. Hence $\varphi ^\varepsilon (x,t) <1$ 
for any $(x,t) \in \Omega \times [0,\infty)$ and the remained claim can be proved similarly.
\end{proof}

In addition, by Proposition \ref{prop3.1} and the maximum principle we have the following Proposition (see \cite{ilmanen1993}). 
\begin{proposition}\label{prop3.2}
If the solution $\varphi ^\varepsilon$ to \eqref{ac} satisfies \eqref{initial} 
and \eqref{eq:3.14},
then we have
\begin{equation}
\frac{\varepsilon \vert \nabla \varphi^\varepsilon (x,t) \vert ^2 }{2} 
\leq \frac{W(\varphi ^\varepsilon (x,t))}{\varepsilon} ,\qquad x \in \Omega, \ t\geq 0. 
\label{eq:3.15}
\end{equation}
\end{proposition}

\begin{proof}
By \eqref{eq:3.4}, we can define a function $r^\varepsilon$ by
\[
r^\varepsilon (x,t) = ( q^\varepsilon )^{-1} (\varphi ^\varepsilon (x,t)), \qquad x \in \Omega, \ t\geq 0,
\]
since $q^\varepsilon : \R \to (-1, 1)$ is one to one and surjective. 
We compute that
\begin{equation*}
\begin{split}
\varepsilon q_r ^\varepsilon r_t ^\varepsilon
&=
\varepsilon q_r ^\varepsilon \Delta r^\varepsilon
+ \varepsilon q_{rr} ^\varepsilon  \vert \nabla r ^\varepsilon \vert ^2
-\dfrac{W' (q ^\varepsilon )}{\varepsilon } 
+ \lambda^{\varepsilon} \sqrt{2W(q ^\varepsilon )}\\
&=
\sqrt{2W(q ^\varepsilon )} \Delta r^\varepsilon
+ \dfrac{W' (q ^\varepsilon )}{\varepsilon } ( \vert \nabla r ^\varepsilon \vert ^2 -1) 
+ \lambda^{\varepsilon} \sqrt{2W(q ^\varepsilon )},
\end{split}
\end{equation*}
where we used \eqref{eq:3.16} and \eqref{eq:3.17}.
Then we obtain
\begin{equation}
r^\varepsilon _t
=
\Delta r^\varepsilon
- \dfrac{2 q ^\varepsilon (r^\varepsilon) }{\varepsilon } ( \vert \nabla r ^\varepsilon \vert ^2 -1) 
+ \lambda^{\varepsilon} ,
\label{ac-r}
\end{equation}
where we used $W' (q^\varepsilon) / \sqrt{2W (q^\varepsilon)} =-2 q^\varepsilon$.
We compute
\begin{equation*}
\frac12 \partial _t \vert \nabla r^\varepsilon \vert^2
=
\frac12 \Delta \vert \nabla r^\varepsilon \vert^2 - \vert \nabla r^\varepsilon \vert ^2 
-\nabla r ^\varepsilon \cdot \nabla \left( \dfrac{2 q ^\varepsilon (r^\varepsilon) }{\varepsilon } 
( \vert \nabla r ^\varepsilon \vert ^2 -1) \right) ,
\end{equation*}
where we used $\nabla \lambda ^\varepsilon =0$.
Set $w^\varepsilon = \vert \nabla r ^\varepsilon \vert ^2 -1$. Then $w^\varepsilon$ satisfies
\begin{equation*}
w ^\varepsilon _t
\leq
\Delta w^\varepsilon  
- \dfrac{4 q ^\varepsilon (r^\varepsilon) }{\varepsilon } \nabla r^\varepsilon \cdot \nabla w^\varepsilon
-2\left(\nabla r^\varepsilon \cdot \nabla \dfrac{2 q ^\varepsilon (r^\varepsilon)}{\varepsilon} \right) w ^\varepsilon.
\end{equation*}
In addition, we have $w ^\varepsilon (x,0) \leq 0$, because 
\[
\frac{\varepsilon \vert \nabla \varphi^\varepsilon _0 (x) \vert ^2 }{2} 
- \frac{W(\varphi ^\varepsilon _0 (x))}{\varepsilon}
=
\frac{W(q ^\varepsilon (r^\varepsilon (x,0)) )}{\varepsilon}
(\vert \nabla r ^\varepsilon (x,0) \vert ^2 -1) \leq 0
\]
by \eqref{eq:3.14}. Hence the maximum principle implies $w^\varepsilon (x,t) \leq 0$ for any $x \in \Omega$ and $t \in [0,\infty)$,
and we obtain \eqref{eq:3.15} by 
$\frac{\varepsilon \vert \nabla \varphi^\varepsilon \vert ^2 }{2} 
- \frac{W(\varphi ^\varepsilon)}{\varepsilon}
= \frac{W(q ^\varepsilon)}{\varepsilon} w^\varepsilon \leq 0.
$
\end{proof}
\subsection{Energy estimates}
By $E_S ^\varepsilon (t)=\sigma \mu _t ^\varepsilon (\Omega)$, 
\eqref{eq:1.11}, \eqref{eq:1.12}, and \eqref{d1}, we can easily obtain the following estimates.
\begin{proposition}
For any $\varepsilon >0$ and $T>0$, we have
\begin{equation}
\mu _{T} ^\varepsilon (\Omega) 
+
\frac{1}{\sigma} \int _{0} ^{T} \int_{\Omega} \varepsilon (\varphi _t ^\varepsilon )^2 \, dxdt
\leq
\mu _{0} ^\varepsilon (\Omega) 
\label{eq:3.1}
\end{equation}
and
\begin{equation}
\sup _{\varepsilon \in (0,1)} \mu _T ^\varepsilon (\Omega) 
\leq \sup _{\varepsilon \in (0,1)} \mu _0 ^\varepsilon (\Omega) 
\leq D_1.
\label{eq:3.2}
\end{equation}
\end{proposition}
\begin{remark}
Generally, ``$\mu _{s} ^\varepsilon (\Omega) \leq \mu _{t} ^\varepsilon (\Omega)$
 for any $0\leq t <s <\infty$'' can not be shown from the energy estimates above. 
However, there exists a countable set $B$ such that 
$\mu _{s} (\Omega) \leq \mu _{t} (\Omega)$
holds for any $t,s \in [0,\infty) \setminus B$ with $t<s$,
where $\mu _{t} (\Omega) =\lim _{\varepsilon \to 0} \mu _{t} ^\varepsilon (\Omega) $ 
(see Proposition \ref{prop3.12}). 
\end{remark}
Set $D' _1 := \sup _{\varepsilon \in (0,1)} \mu _0 ^\varepsilon (\Omega)$. Note that $D' _1 \leq D_1$.
By an argument similar to that in \cite{bronsard-stoth}, we have
the following lemma.
\begin{lemma}\label{lem3.2}
There exist constants $\Cl{const:3.1-2}= \Cr{const:3.1-2}(\omega,d ,D' _1)>0$,
$\Cl{const:3.1}= \Cr{const:3.1}(\omega,d ,D' _1)>0$, and $\epsilon_1 =\epsilon_1 (\omega, d, D' _1, \alpha) >0$
such that
\begin{equation}
\int_ {0} ^{T} \vert \lambda ^\varepsilon (t) \vert ^2 \, dt 
\leq \Cr{const:3.1-2} (\mu _0 ^\varepsilon (\Omega) -\mu _{T} ^\varepsilon (\Omega) + T )
\quad \text{for any} \ \varepsilon  \in (0,\epsilon_1) \ \text{and} \ T>0,
\label{eq:3.5.additional}
\end{equation}
and
\begin{equation}
\sup _{\varepsilon \in (0,\epsilon_1)} 
\int_ {t_1} ^{t_2} \vert \lambda ^\varepsilon (t) \vert^2 \, dt 
\leq \Cr{const:3.1} (1+t_2 -t_1) 
\qquad \text{for any} \  0 \leq t_1 < t_2 <\infty.
\label{eq:3.5}
\end{equation}

\end{lemma}

\begin{proof}
Let $\vec{\zeta}=(\zeta ^1 ,\zeta^2,\dots , \zeta^d) : \Omega \times [0,\infty) \to \mathbb{R}^d$ 
be a smooth periodic test function.
Multiply \eqref{ac} by $\nabla \varphi ^\varepsilon \cdot \vec{\zeta}$ and integrate over $\Omega$.
Then we have
\begin{equation}
\begin{split}
&\int _{\Omega} \varepsilon \varphi ^\varepsilon _t \nabla \varphi ^\varepsilon \cdot \vec \zeta \, dx
+ \sum _{i,j=1} ^d \int_{\Omega} 
\varepsilon \varphi _{x_i} ^\varepsilon \varphi _{x_j} ^\varepsilon \zeta ^j _{x_i} \, dx
- \int_{\Omega} \left( \frac{\varepsilon \vert \nabla \varphi ^\varepsilon \vert^2}{2} 
+\frac{W(\varphi ^\varepsilon)}{\varepsilon}\right) \div \vec \zeta \, dx \\
=& \, -\lambda ^\varepsilon \int _{\Omega} k(\varphi ^\varepsilon) \div \vec \zeta \, dx,
\end{split}
\label{eq:3.6}
\end{equation}
by the integration by parts. Here we used
$\nabla k (\varphi ^\varepsilon) = \sqrt{2W(\varphi ^\varepsilon)} \nabla \varphi ^\varepsilon$
and 
\[
\int _\Omega \Delta \varphi ^\varepsilon \nabla \varphi ^\varepsilon \cdot \vec \zeta \, dx
= -\sum _{i,j=1} ^d \int_{\Omega} 
 \varphi _{x_i} ^\varepsilon \varphi _{x_j} ^\varepsilon \zeta ^j _{x_i} \, dx
+ \int _\Omega \frac{ \vert \nabla \varphi ^\varepsilon \vert ^2}{2} \div \vec \zeta \, dx.
\]
The Cauchy--Schwarz inequality, \eqref{eq:3.1}, and \eqref{eq:3.6} imply
\begin{equation}
\begin{split}
&\left \vert
\int _{\Omega} \varepsilon \varphi ^\varepsilon _t \nabla \varphi ^\varepsilon \cdot \vec \zeta \, dx
+ \sum _{i,j=1} ^d \int_{\Omega} 
\varepsilon \varphi _{x_i} ^\varepsilon \varphi _{x_j} ^\varepsilon \zeta ^j _{x_i} \, dx
- \int_{\Omega} \left( \frac{\varepsilon \vert \nabla \varphi ^\varepsilon \vert^2}{2} 
+\frac{W(\varphi ^\varepsilon)}{\varepsilon}\right) \div \vec \zeta \, dx
\right \vert \\
\leq & \, 
\Cl{const:3.2} \| \vec \zeta (\cdot, t) \|_{C^1 (\Omega)} \left( (D'_1) ^\frac{1}{2} 
\left( \int_\Omega \varepsilon (\varphi _t ^\varepsilon)^2 \, dx \right) ^{\frac12}
+ D' _1 \right),
\end{split}
\label{eq:3.7}
\end{equation}
where $\Cr{const:3.2}>0$ depends only on $d$. 
Let $\eta \in C ^\infty _c (B_1 (0))$ 
be a smooth nonnegative function with $\int _{B_1 (0)} \eta \, dx =1$
and define the standard mollifier $\eta _\delta$ by $\eta _\delta (x) = \delta ^{-d} \eta (x/\delta)$
for $\delta>0$.
Let $u=u(x,t)$ be a periodic solution to
\begin{equation*}
\left\{ 
\begin{array}{ll}
-\Delta u &= k (\varphi ^\varepsilon) \ast \eta_\delta - \dashint _\Omega (k (\varphi ^\varepsilon) \ast \eta_\delta) \qquad \text{in} \ \Omega, \\
\int _\Omega u \, dx &=0.
\end{array} \right.
\end{equation*}
Note that 
\[
\int _{\Omega} \left\{ k (\varphi ^\varepsilon) \ast \eta_\delta - \dashint _\Omega (k (\varphi ^\varepsilon) \ast \eta_\delta) \right\} =0
\]
and there exists $C>0$ depending only on $\mathscr{L}^d (\Omega)$ such that
\[
\left\| k (\varphi ^\varepsilon (\cdot ,t)) \ast \eta_\delta - \dashint _\Omega (k (\varphi ^\varepsilon (\cdot ,t)) \ast \eta_\delta) \right\| _{C^1 (\Omega)} \leq C(1+ \delta ^{-1}), \qquad t\geq 0,
\]
where we used $\| \varphi ^\varepsilon \|_{L^\infty} \leq 1$.
Therefore the standard PDE arguments imply the existence and uniqueness of the solution $u$ and
\[
\| u(\cdot, t) \|_{C^{2,\beta} (\Omega)} \leq \Cl{const:3.3}, \qquad t\geq 0,
\] 
where $\beta \in (0,1)$ and
$\Cr{const:3.3}>0$ depends only on $\beta$, $d$, and $\delta$. 
Set $\vec \zeta (x,t) = \nabla u(x,t)$. Then, by \eqref{eq:3.6} and \eqref{eq:3.7}, 
we have
\begin{equation}
\begin{split}
\vert \lambda ^\varepsilon \vert 
\left \vert \int _{\Omega} k(\varphi ^\varepsilon) \div \vec \zeta \, dx \right \vert
\leq  \Cr{const:3.2} \Cr{const:3.3} \left( (D' _1) ^\frac{1}{2} 
\left( \int_\Omega \varepsilon (\varphi _t ^\varepsilon)^2 \, dx \right) ^{\frac12}
+ D' _1 \right).
\end{split}
\label{eq:3.8}
\end{equation}
We compute
\begin{equation}
\begin{split}
& -\int _{\Omega} k(\varphi ^\varepsilon) \div \vec \zeta \, dx  
= \int _{\Omega} k(\varphi ^\varepsilon) (-\Delta u )\, dx \\
= & \, 
\int _\Omega k(\varphi^ \varepsilon) 
\left\{ k (\varphi ^\varepsilon) \ast \eta_\delta - \dashint _\Omega (k (\varphi ^\varepsilon) \ast \eta_\delta) \right\} \, dx\\
= & \, 
\frac49 \mathscr{L}^d (\Omega) 
+ \int _{\Omega} (k (\varphi ^\varepsilon)) ^2 -\frac49 \, dx
+ \int _{\Omega} k (\varphi ^\varepsilon) \{ k(\varphi ^\varepsilon) \ast \eta _\delta -k(\varphi ^\varepsilon) \} \, dx \\
& \, -\frac{1}{\mathscr{L}^d (\Omega)} \left( \int _{\Omega} k (\varphi ^\varepsilon) \, dx \right)^2
+ \frac{1}{\mathscr{L}^d (\Omega)} \int _{\Omega} k (\varphi ^\varepsilon) \, dx \left( \int _{\Omega} k (\varphi ^\varepsilon) \, dx
 -\int _{\Omega} k (\varphi ^\varepsilon) \ast \eta_\delta \, dx \right).
\end{split}
\label{eq:3.9}
\end{equation}
By $(k(s))^2 -\frac49 \geq -W(s)$ for any $s \in [-1,1]$, we have
\begin{equation}
\int _{\Omega} (k (\varphi ^\varepsilon)) ^2 -\frac49 \, dx \geq -\varepsilon \sigma \mu _t ^\varepsilon (\Omega) \geq -\varepsilon \sigma D' _1.
\label{eq:3.10}
\end{equation}
By using
\[
\int_{\Omega} \vert \nabla (k(\varphi ^\varepsilon)) \vert \, dx
=
\int_{\Omega} \sqrt{2W (\varphi ^\varepsilon)} \vert \nabla \varphi ^\varepsilon \vert \, dx
\leq \sigma \mu _t ^\varepsilon (\Omega) \leq \sigma D' _1,
\]
$\| \varphi ^\varepsilon \|_{L^\infty} \leq 1$, and Proposition \ref{prop7.2}, we have
\begin{equation}
\left\vert \int _{\Omega} k (\varphi ^\varepsilon) \{ k(\varphi ^\varepsilon) \ast \eta _\delta 
-k(\varphi ^\varepsilon) \} \, dx \right \vert
\leq \Cl{const:3.4} \delta,
\label{eq:3.11}
\end{equation}
and 
\begin{equation}
\left \vert 
\frac{1}{\mathscr{L}^d (\Omega)} \int _{\Omega} k (\varphi ^\varepsilon) \, dx \left( \int _{\Omega} k (\varphi ^\varepsilon) \, dx
 -\int _{\Omega} k (\varphi ^\varepsilon) \ast \eta_\delta \, dx \right)
 \right \vert
\leq \Cr{const:3.4} \delta,
\label{eq:3.12}
\end{equation}
where $\Cr{const:3.4} >0$ depends only on $D' _1$ and $\mathscr{L}^d (\Omega)$.
Set $\delta = \frac{\omega ^2}{4\Cr{const:3.4} \mathscr{L}^d (\Omega)}$.
By \eqref{eq:1.5}, \eqref{omega},
\eqref{eq:3.9}, \eqref{eq:3.10}, \eqref{eq:3.11}, and \eqref{eq:3.12}, there exists $\epsilon_1 >0$
depending only on $\alpha$, $D' _1$, $\mathscr{L}^d (\Omega)$, and $\omega$ such that 
\begin{equation}
\begin{split}
& -\int _{\Omega} k(\varphi ^\varepsilon) \div \vec \zeta \, dx  \\
\geq & \, 
\frac49 \mathscr{L}^d (\Omega)
-\frac{1}{\mathscr{L}^d (\Omega)} \left( \int _{\Omega} k (\varphi ^\varepsilon) \, dx \right)^2
-\varepsilon \sigma D' _1 -2 \Cr{const:3.4} \delta \\
\geq & \, 
\frac{1}{\mathscr{L}^d (\Omega)}
\left( \omega ^2 -\frac{4\sqrt{2}}{3} \varepsilon ^{\frac{\alpha}{2}} (D' _1)^{\frac12} \right)
-\varepsilon \sigma D' _1 -2 \Cr{const:3.4} \delta \\
\geq & \, 
\frac{1}{4\mathscr{L}^d (\Omega)}
\omega ^2
\end{split}
\label{eq:3.13}
\end{equation}
holds for any $\varepsilon \in (0,\epsilon_1)$,
where we used 
$(\int _{\Omega} k (\varphi _0 ^\varepsilon) \, dx)^2
-(\int _{\Omega} k (\varphi ^\varepsilon) \, dx)^2 
\leq 
\frac{4\sqrt{2}}{3} \varepsilon ^{\frac{\alpha}{2}} (D' _1)^{\frac12}$ by \eqref{eq:1.5}.
From \eqref{eq:3.1}, \eqref{eq:3.2}, \eqref{eq:3.8}, and \eqref{eq:3.13}, we obtain 
\eqref{eq:3.5.additional} and \eqref{eq:3.5}.
\end{proof}

\begin{remark}
For the classical solution to the volume preserving mean curvature flow, we can obtain a similar estimate
for the non-local term (see Proposition \ref{prop7.3}).
\end{remark}

We define the discrepancy measure $\xi _t ^\varepsilon$ on $\Omega$ by
\begin{equation}
\xi _t ^\varepsilon ( \phi ) 
:= 
\frac{1}{\sigma} \int _{\Omega} \phi (x) \left( \frac{\varepsilon \vert \nabla \varphi^\varepsilon (x,t) \vert ^2 }{2} 
- \frac{W(\varphi ^\varepsilon (x,t))}{\varepsilon} \right) \, dx,
\qquad \phi \in C_c (\Omega).
\label{xi}
\end{equation}
In addition, we denote
$$
\xi _\varepsilon (x,t) = \frac{\varepsilon \vert \nabla \varphi^\varepsilon (x,t) \vert ^2 }{2} 
- \frac{W(\varphi ^\varepsilon (x,t))}{\varepsilon}.
$$
Proposition \ref{prop3.2} implies the following lemma.
\begin{lemma}\label{lem3.5}
Assume \eqref{eq:3.14}. Then 
$\xi _\varepsilon (x,t) \leq 0$ for any $ (x,t) \in \Omega\times [0,\infty)$.
In addition, $\xi _t ^\varepsilon $ is a non-positive measure for any $t \geq 0$.
\end{lemma}
We denote the backward heat kernel $\rho =\rho _{(y,s)} (x,t)$ by
\[
\rho _{(y,s)} (x,t) = \frac{1}{(4\pi (s-t)) ^{\frac{d-1}{2}}} e ^{-\frac{\vert x-y\vert^2}{4(s-t)}},
\qquad 
x,y \in \R^d, \ 0\leq t <s.
\]
With exactly the same proof as in \cite[p.\,2028]{takasao2017}, we obtain the following estimates
similar to the monotonicity formula obtaind by Huisken~\cite{huisken1990}
(for convenience, we call the following as the monotonicity formula).
\begin{proposition}[See \cite{takasao2017}]\label{prop3.8}
Let $\xi _\varepsilon (x,0) \leq 0$ for any $ x \in \Omega$. Assume \eqref{eq:3.14}. Then 
\begin{equation}\label{eq:3.21}
\frac{d}{dt} \int_{\mathbb{R}^d} \rho _{(y,s)} (x,t) \, d\mu _t ^\varepsilon (x)
\leq 
\frac{1}{2(s-t)} \int_{\mathbb{R}^d} \rho _{(y,s)} (x,t) \, d\xi _t ^\varepsilon (x)
+
\frac12 (\lambda ^\varepsilon)^2 \int_{\mathbb{R}^d} \rho _{(y,s)} (x,t) \, d\mu _t ^\varepsilon (x)
\end{equation}
holds for any $0\leq t<s<\infty$ and for any $y \in \R^d$.
Here, $\mu _t ^\varepsilon$ and $\xi ^\varepsilon _t$ are extended periodically to $\mathbb{R}^d$.
In addition, we have
\begin{equation}\label{eq:3.20}
\begin{split}
\int_{\mathbb{R}^d} \rho _{(y,s)} (x,t) \, d\mu _t ^\varepsilon (x) \Big\vert_{t=t_2} 
\leq & \, 
\left( \int_{\mathbb{R}^d} \rho _{(y,s)} (x,t) \, d\mu _t ^\varepsilon (x) \Big \vert_{t=t_1} \right)
e^{\frac{1}{2} \int _{t_1} ^{t_2} \vert \lambda ^{\varepsilon} \vert ^2 \, dt }
\\
\leq & \, \left( \int_{\mathbb{R}^d} \rho _{(y,s)} (x,t) \, d\mu _t ^\varepsilon (x) \Big \vert_{t=t_1} \right)
e^{\Cr{const:3.1} (t_2 -t_1 +1)}
\end{split}
\end{equation}
for any $y \in \mathbb{R}^d$, $0\leq t_1 < t_2 <\infty$, and $\varepsilon \in (0,\epsilon_1)$.
\end{proposition}
\begin{remark}\label{rem3.9}
Ilmanen~\cite{ilmanen1993} proved the monotone decreasing of 
$\int_{\mathbb{R}^d} \rho _{(y,s)} (x,t) \, d\mu _t ^\varepsilon (x)$
with respect to $t$,
for the solution to the Allen--Cahn equation without the non-local term
under suitable assumptions.
In general, one can show that
the Brakke flow with smooth initial data has unit density for a short time
by using the monotonicity formula (see \cite{takasao-tonegawa}). However, 
in order to show a similar conclusion for our problem, 
it is necessary that $\mu _0 ^\varepsilon (\Omega) -\mu _ {t} ^\varepsilon (\Omega)$ is small enough,
due to \eqref{eq:3.5.additional} (see Lemma \ref{lem5.3} below).
\end{remark}
As a corollary of the monotonicity formula, we can obtain the 
following upper bounds of the densities of $\mu _t ^\varepsilon$.
\begin{corollary}[See \cite{ilmanen1993, MR3656492}]
There exists $0<D_2 <\infty$ depending only on $d$, $\Cr{const:3.1}$, $D_1$, and $T$ such that
\begin{equation}\label{eq:3.22}
\mu _t ^\varepsilon (B_R (y)) \leq D_2 R^{d-1}
\end{equation}
for all $y \in \mathbb{R}^d$, $R \in (0,1)$, $\varepsilon \in (0,\epsilon_1)$, and $t \in [0, T]$.
\end{corollary}

\begin{proof}
Using the same calculation as \eqref{eq:7.15} below, we have
\begin{equation}\label{eq:3.70}
\int _{\R^d} \rho _{(y,s)} (x,0) \, d \mu _0 ^\varepsilon
\leq \frac{D_1 \omega _{d-1}}{\pi^{\frac{d-1}{2}}} \int _0 ^1 \left( \log \frac{1}{k} \right) ^{\frac{d-1}{2}} \, dk
=D_1
\end{equation}
for any $s>0$ and $y \in \R^d$,
where we used $
\int _0 ^1 \left( \log \frac{1}{k} \right)^{\frac{d-1}{2}} \, dk
=\Gamma (\frac{d-1}{2} +1)
=\pi ^{\frac{d-1}{2}} /\omega_{d-1}
$.
By \eqref{eq:3.20} and \eqref{eq:3.70},
\begin{equation}\label{eq:3.71}
\int _{\R^d} \rho _{(y,s)} (x,t) \, d \mu _t ^\varepsilon
\leq D_1 e^{\Cr{const:3.1} (T+1)}
\end{equation}
for any $t \in [0,T)$ with $0<t<s$ and $y \in \R^d$. Set $R=2\sqrt{s-t}$. 
We compute 
\begin{equation}\label{eq:3.72}
\int _{\R^d} \rho _{(y,s)} (x,t) \, d \mu _t ^\varepsilon
=\frac{1}{\pi ^{\frac{d-1}{2} } R^{d-1}} \int _{\R^d} e^{- \frac{\vert x-y \vert^2}{R^2}} \, d \mu _t ^\varepsilon
\geq \frac{1}{\pi ^{\frac{d-1}{2} } R^{d-1}} \int _{B_R (y)} e^{- 1} \, d \mu _t ^\varepsilon.
\end{equation}
Therefore we have \eqref{eq:3.22} by \eqref{eq:3.71} and \eqref{eq:3.72}.
\end{proof}

By the integration by parts, we have the following estimate.
\begin{lemma}
For any nonnegative test function $\phi \in C_c ^2 (\Omega)$, there exists
$\Cl{const:3.10} >0$ depends only on $D_1$, $\| \phi \|_{C^2 (\Omega)}$, $\omega$, and $d$
such that
\begin{equation}\label{eq:3.23}
\int_0 ^T \left\vert \frac{d}{dt} \mu _t ^\varepsilon (\phi) \right\vert \, dt \leq \Cr{const:3.10}(1+T)
\qquad \text{for any} \ \varepsilon \in (0,\epsilon_1) \ \text{and} \ T>0.
\end{equation}
\end{lemma}

\begin{proof}
By the integration by parts, we have
\begin{equation}
\begin{split}
& \frac{d}{dt} \int _{\Omega} \phi
\left( \frac{\varepsilon \vert \nabla \varphi ^\varepsilon \vert^2}{2} 
+ \frac{W(\varphi ^\varepsilon )}{\varepsilon} \right) \, dx \\
= & \,
- \int _{\Omega} \varepsilon \phi (\varphi ^\varepsilon _t) ^2 \, dx
+\lambda ^\varepsilon \int _{\Omega} \phi \sqrt{2 W (\varphi ^\varepsilon)} \varphi ^\varepsilon _t \, dx
- \int _{\Omega} \varepsilon (\nabla \phi \cdot \nabla \varphi ^\varepsilon ) \varphi ^\varepsilon _t \, dx.
\end{split}
\end{equation}
Hence,
\begin{equation}\label{eq:3.26}
\begin{split}
& \left\vert \frac{d}{dt} \int _{\Omega} \phi
\left( \frac{\varepsilon \vert\nabla \varphi ^\varepsilon \vert^2}{2} 
+ \frac{W(\varphi ^\varepsilon )}{\varepsilon} \right) \, dx \right\vert \\
\leq & \,
\vert \lambda ^\varepsilon \vert^2 \int _{\Omega} \phi \frac{2W(\varphi ^\varepsilon)}{\varepsilon} \, dx
+ \int _{\Omega} \frac{\vert \nabla \phi \vert^2}{\phi} \varepsilon \vert \nabla \varphi ^\varepsilon \vert^2 \, dx
+ 2 \int _{\Omega} \varepsilon \phi ( \varphi ^\varepsilon _t ) ^2 \, dx
\\
\leq & \,
C \sigma D_1 
\left( 
1+ \vert \lambda ^\varepsilon \vert^2 + \int _{\Omega} \varepsilon ( \varphi ^\varepsilon _t ) ^2 \, dx
\right),
\end{split}
\end{equation}
where $C>0$ depends only on $\|\phi\|_{C^2 (\Omega)}$. Note that 
there exists $c=c(d) >0$ such that $\sup \frac{\vert \nabla \phi \vert^2}{\phi} 
\leq c(d) \| \nabla ^2 \phi \|_{C^0 (\Omega)}$ by Cauchy's mean-value theorem.
Thus \eqref{eq:3.1}, \eqref{eq:3.5}, and \eqref{eq:3.26} imply \eqref{eq:3.23}.
\end{proof}

With an argument similar to \cite{ilmanen1993}, 
we can show the following proposition by using 
\eqref{eq:3.23}.
\begin{proposition}\label{prop3.11}
There exist a subsequence $\{ \varepsilon _{i_j} \}_{j=1} ^\infty$ and
a family of Radon measures $\{ \mu _t \}_{t\geq 0}$ such that
\begin{equation}\label{eq:3.24}
\mu _t ^{\varepsilon _{i_j}} \to \mu _t \qquad \text{as Radon measures on} \ \Omega
\end{equation}
for any $t \in [0,\infty)$ and for any $d\geq 2$.
In addition, there exists a countable set $B \subset [0,\infty)$ such that
$\mu _t (\Omega)$ is continuous on $[0,\infty) \setminus B$.
\end{proposition}

\begin{proof}
Let $\{ \phi_k \} _{k=1} ^\infty \subset C_c (\Omega) $ be a dense subset 
with $\phi _k \in C_c ^2 (\Omega)$ for any $k$, and 
for any $x \in \Omega \cap \Q^d$ and $r \in (0,1) \cap \Q$, there exists $k \in \N$
such that $\phi _k \in C_c ^2 (B_r (x))$.
Let 
$f _k ^i (t)=\mu _0 ^{\varepsilon _i} (\phi _k) 
+ \int _0 ^t \left( \frac{d}{ds} \mu _s ^{\varepsilon _i} (\phi _k) \right)_+ \, ds$
and
$ g _k ^i (t) = \int _0 ^t \left( \frac{d}{ds} \mu _s ^{\varepsilon _i} (\phi _k) \right)_- \, ds $. Then
\[
\mu _t ^{\varepsilon _i} (\phi _k) = f_k ^i (t) - g_k ^i (t),
\]
and $f _k ^i (t)$ and $g _k ^i (t)$ are non decreasing functions with
\[
0 \leq f_k ^i (t) \leq C _k
\qquad 
\text{and}
\qquad
0 \leq g_k ^i (t) \leq C _k
\]
for any $i$ and $t \in [0,T)$, where $C_k >0$ depends only on 
$D_1$, $\| \phi_k \|_{C^2 (\Omega)}$, $\omega$, and $d$, by \eqref{eq:3.23}.
Then Helly's selection theorem implies that there exist a subsequence $\varepsilon_i \to 0$
(denoted by the same index),
$f_k ,g_k :[0,T) \to [0,\infty)$ such that
$\lim _{i\to \infty } f_k ^i (t) = f_k (t)$ and $\lim _{i\to \infty } g_k ^i (t) = g_k (t)$
for any $t \in [0,T)$.
 Therefore we have
\begin{equation}\label{eq:3.27}
\lim _{i \to \infty} \mu _t ^{\varepsilon _i} (\phi _k) = f_k (t) - g_k (t)
\end{equation}
for any $t \in [0,T)$. By this and the diagonal argument, we can choose
a subsequence such that \eqref{eq:3.27} holds for any $t \in [0,T)$
and $k \in \N$.
On the other hand, for any $t \in [0,T)$, the compactness of Radon measures
yields that there exist $\mu _t $ and a subsequence $\varepsilon _i \to 0$ (depending on $t$) such that 
$\mu _t ^{\varepsilon_i} \to \mu _t$ as Radon measures. 
However, $\mu_t$ is uniquely determined
by \eqref{eq:3.27}. Hence we obtain \eqref{eq:3.24} for any $t \in [0,T)$. By the diagonal argument
with $T\to \infty$,
we have \eqref{eq:3.24} for any $t \in [0,\infty)$.

From the similar argument as above, 
there exists monotone increasing functions $f$ and $g$ such that
$\mu _t (\Omega) = f(t) - g(t)$. By the monotonicity, there exists a countable set $B$
such that $f$ and $g$ are continuous on $[0,\infty ) \setminus B$.
This concludes the proof.
\end{proof}

\begin{proposition}\label{prop3.12}
Let $B$ be the countable set given by Proposition \ref{prop3.11}.
For any $t,s \in [0,\infty) \setminus B$ with $t<s$, we have
$\mu _s (\Omega) \leq \mu _t (\Omega)$.
\end{proposition}

\begin{proof}
From Proposition \ref{prop3.11}, 
we may assume that $\mu _t ^{\varepsilon _i} (\Omega) \to \mu _t (\Omega)$ for any 
$t \in [0,\infty)$.
By \eqref{eq:1.11} and $E ^{\varepsilon _i} (t) \leq E_S ^{\varepsilon _i} (0) \leq D_1$,
Helly's selection theorem yields that there exist a subsequence $\varepsilon_i \to 0$
(denoted by the same index) and a monotone decreasing function $E (t)$ such that
$E^{\varepsilon _i} (t) \to E (t)$ for any $t \in [0,\infty)$. 
For any $T>0$, the estimate \eqref{eq:3.5} and Fatou's lemma imply
\[ 
\int _0 ^T \liminf_{i\to \infty } E_P ^{\varepsilon _i} (t) \, dt
\leq \liminf_{i\to \infty } \int _0 ^T E_P ^{\varepsilon _i} (t) \, dt
=
\liminf_{i\to \infty } \int _0 ^T \frac{\varepsilon _i ^\alpha}{2} \vert \lambda^{\varepsilon _i} \vert^2 \, dt=0.
\]
Therefore $\liminf_{i\to \infty } E_P ^{\varepsilon _i} (t)=0$ a.e. $t \geq 0$
and hence $ E(t) = \sigma \mu _t (\Omega) $ for a.e. $t\geq 0$.
By this, the monotonicity of $E (t)$, and the continuity of $\mu _t (\Omega) $ on $[0,\infty) \setminus B$, 
we obtain the claim.
\end{proof}

We define a Radon measure $\mu $ on $\Omega \times [0,\infty)$ by $d\mu:=d\mu _t dt$.
By the boundedness of $\sup _i \mu _t ^{\varepsilon_i} (\Omega)$, 
the dominated convergence theorem implies
\[
\lim _{i\to \infty} \int _0 ^T \int _{\Omega} \phi \, d\mu _t ^{\varepsilon _i} dt
= \int _{\Omega \times [0,T)} \phi \,  d\mu 
\qquad \text{for any} \ \phi \in C_c (\Omega \times [0,T)).
\]
For measures $\mu$ and $\mu _t$, we have the following property.
\begin{proposition}
\label{prop3.13}
There exists a countable set $\tilde B \subset [0,\infty)$ such that
\begin{equation}\label{eq:3.50}
\spt \mu_t \subset \{ x \in \Omega \mid (x,t) \in \spt \mu \}
\end{equation}
for any $t \in (0,\infty) \setminus \tilde B$.
\end{proposition}


\begin{proof}
Let $f_k$ and $g_k$ be monotone increase functions given by Proposition \ref{prop3.11}. 
Then there exists a countable set $\tilde B$ such that $f_k$ and $g_k$ are continuous
on $[0,\infty) \setminus \tilde B$ for any $k$.
Suppose that there exists $t_0 \in [0,\infty) \setminus \tilde B$ such that 
$x \in \spt \mu_{t_0}$ and $(x,t_0) \not \in \spt \mu$. 
Then we may assume that there exists $k$ such that $x \in \spt \phi _k $ and 
$\mu (\phi _k \times (t_0 -\delta , t_0 +\delta))=0$ for sufficiently small $\delta>0$,
where $\phi _k$ is a function given by Proposition \ref{prop3.11}.
From $x \in \spt \mu _{t_0}$, $\mu _{t_0 } (\phi _k)>0$ and there exists $\delta' >0$
such that $ \mu _{t } (\phi _k)>0$ for any $t \in (t_0 -\delta', t_0 +\delta')$ by the
continuity of $f_k$ and $g_k$. However, this contradicts $\mu (\phi _k \times (t_0 -\delta , t_0 +\delta))=0$.
Therefore we obtain \eqref{eq:3.50} for $t \in [0,\infty) \setminus \tilde B$.
\end{proof}

\subsection{Integrarity of $\mu_t$ for $d \leq 3$}

In the case of $d\leq 3$, we can use the results of \cite{roger-schatzle}.
For $d\geq 4$, we employ the arguments of \cite{ilmanen1993, MR3348119, takasao-tonegawa}
in Section 4 below.
\begin{theorem}\label{thm3.11}
Assume that $d=2$ or $3$ and \eqref{eq:3.24}. Then $\mu _t$ is integral for a.e. $t \geq 0$.
\end{theorem}

\begin{proof}
The estimates \eqref{eq:3.1}, \eqref{eq:3.2}, and \eqref{eq:3.5} imply
\begin{equation}\label{eq:3.25}
\begin{split}
& \, \int _0 ^T \int _{\Omega} \varepsilon 
\left( \Delta \varphi ^{\varepsilon} -\dfrac{W' (\varphi ^{\varepsilon})}{\varepsilon^2 } \right)^2 
\,dxdt \leq 
\int _0 ^T \int _{\Omega} \varepsilon 
( \varphi ^{\varepsilon} _t )^2
\,dxdt
+
\int _0 ^T \vert \lambda ^\varepsilon \vert^2  \int _{\Omega} 
\dfrac{2 W (\varphi ^{\varepsilon})}{\varepsilon } 
\,dxdt \\
\leq & \, 
\sigma \mu _0 ^\varepsilon (\Omega)
+ 2 D_1 \Cr{const:3.1} (1+T)
\leq \sigma D_1 + 2 D_1 \Cr{const:3.1} (1+T)
\end{split}
\end{equation}
for any $T>0$. 
Then Fatou's lemma yields
\begin{equation*}
\begin{split}
& \int _0 ^T \liminf_{i\to \infty } \int _{\Omega} \varepsilon_i 
\left( \Delta \varphi ^{\varepsilon_i} -\dfrac{W' (\varphi ^{\varepsilon_i})}{\varepsilon^2 _i } \right)^2 
\,dx dt \\
\leq & \, 
\liminf _{i\to \infty} \int _0 ^T \int _{\Omega} \varepsilon_i 
\left( \Delta \varphi ^{\varepsilon_i} -\dfrac{W' (\varphi ^{\varepsilon_i})}{\varepsilon^2 _i } \right)^2 
\,dxdt <\infty.
\end{split}
\end{equation*}
Therefore
\[
\liminf_{i\to \infty } \int _{\Omega} \varepsilon_i \left(
\Delta \varphi ^{\varepsilon_i} -\dfrac{W' (\varphi ^{\varepsilon_i})}{\varepsilon^2 _i } \right)^2 
\,dx  < \infty \qquad \text{for a.e. } \ t\geq 0.
\]
By this, $2\leq d \leq 3$, and \eqref{eq:3.2}, 
$\mu _t$ is integral for a.e. $t \geq 0$ (see \cite[Theorem 5.1]{roger-schatzle}).
\end{proof}


\section{Rectifiability and integrality of $\mu_t$}
We already proved the rectifiability and integrality of $\mu_t$ with $d\leq 3$ in Theorem \ref{thm3.11}.
Next we consider the case of $d\geq 2$
and basically follow \cite{ilmanen1993, MR3348119, takasao-tonegawa}.

\subsection{Assumptions}
We assume \eqref{d1} and \eqref{omega}--\eqref{eq:3.14} again in this section.
Let $\{ \varepsilon _i \}_{i=1} ^\infty$ 
be a positive sequence such that 
$\varepsilon _i \to 0 $ as $i \to \infty$.
By the weak compactness of the Radon measures and
Proposition \ref{prop3.11}, we may assume that
there exist Radon measures $\mu$, $\vert \xi \vert$ and a family of Radon measures $\{ \mu _t\} _{t \in [0,T)}$
such that
\[
\mu ( \phi ) = \lim _{i \to \infty} \int _ 0 ^T \mu _t ^{\varepsilon _i} (\phi) \, dt,
\qquad
\vert \xi \vert ( \phi ) = \lim _{i \to \infty} \int _ 0 ^T \int _{\Omega} \sigma^{-1}
\vert \xi _{\varepsilon _i} \vert \phi \, dxdt,
\qquad \phi \in C_c (\Omega \times (0,T))
\]
and
\[
\mu_t ( \phi ) = \lim _{i \to \infty} \mu _t ^{\varepsilon _i} (\phi) ,
\qquad \phi \in C_c (\Omega ), \ t \in [0,T).
\]
%
\begin{remark}
In the discussion above, we proved that there exists 
$\mu_t = \lim _{\varepsilon \to 0} \mu _t ^\varepsilon$ for any $t\geq 0$, 
however such a property does not necessarily hold for $\xi _t ^\varepsilon$.
\end{remark}

%

By the standard PDE theories and the rescaling arguments, we obtain the following lemma.
The proof is almost the same as \cite[Lemma 4.1]{takasao-tonegawa}. So, we skip this.
\begin{lemma}\label{lem:holder}
There exists $\Cl{const:3.5}>0$ depending only on $d$ and $\Cr{const:initial}$ such that
\begin{equation}
\sup _{\Omega \times [0,T) } \varepsilon \vert \nabla \varphi ^\varepsilon \vert 
+ \sup _{x,y \in \Omega, \ t \in [0,T)} 
\frac{\varepsilon ^{\frac{3}{2}} \vert \nabla \varphi ^\varepsilon(x,t) - \nabla \varphi ^\varepsilon(y,t) \vert }{\vert x-y \vert^{\frac12}} \leq \Cr{const:3.5}
\end{equation}
for any $\varepsilon \in (0,1)$.
\end{lemma}

\subsection{Vanishing of $\xi$}
First we show $\vert \xi \vert=0$ for any $d \geq 2$.
\begin{lemma}\label{lem4.1}
Assume $(x',t') \in \spt \mu$ and $\alpha_1 \in (0,1)$. Then there exist a sequence
$\{ (x_j ,t_j) \}_{j=1} ^\infty$ and a subsequence $\{ \varepsilon _{i_j} \}_{j=1} ^\infty$
such that
$\vert (x_j ,t_j)-(x', t') \vert < \frac{1}{j}$ and 
$\vert \varphi ^{\varepsilon_{i_j}} (x_j ,t_j) \vert <\alpha_1 $ for all $j$.
\end{lemma}

\begin{proof}
Define $Q_r = \overline{ B_r (x') \times (t' -r , t' +r)}$ for $r>0$.
If the claim is not true, then there are $r>0$ and $N>1$ such that
$\inf _{Q_r} \vert \varphi ^{\varepsilon_i } \vert \geq \alpha_1 $ for any $i >N$.
Without loss of generality, we may assume that 
$\inf _{Q_r} \varphi ^{\varepsilon_i } \geq \alpha_1 $ for any $i >N$.
For $s \in [\alpha_1 ,1)$, we have
$W(s) = \frac{1}{4s}W'(s) (s^2 -1) \leq \frac{1}{4\alpha_1 }W'(s) (s^2 -1)$.
Assume that $\phi \in C_c ^\infty (B_r (x'))$ satisfies
$0 \leq \phi \leq 1$ and $\phi=1 $ on $B_{r/2} (x')$.
We compute
\begin{equation*}
\begin{split}
\int _{Q_r} \phi ^2 \frac{W(\varphi ^\varepsilon)}{\varepsilon ^2} \, dx dt
\leq & \, \frac{1}{4\alpha_1} \int _{Q_r} \phi ^2 \frac{W'(\varphi ^\varepsilon)}{\varepsilon ^2} 
( (\varphi ^\varepsilon)^2 - 1 ) \, dx dt \\
= & \, \frac{1}{4\alpha_1}  \int _{Q_r} \phi ^2 
\left( -\varphi _t ^\varepsilon +\Delta \varphi ^{\varepsilon} 
+ \lambda ^{\varepsilon} \frac{\sqrt{2W(\varphi ^\varepsilon)}}{\varepsilon} \right) 
( (\varphi ^\varepsilon)^2 - 1 ) \, dx dt.
\end{split}
\end{equation*}
We compute
\begin{equation*}
\begin{split}
\left\vert \int _{Q_r} \phi ^2 \varphi _t ^\varepsilon  
( (\varphi ^\varepsilon)^2 - 1 ) \, dx dt\right\vert
=
\left\vert \int _{t'- r} ^{t' +r} \frac{d}{dt} \int _{B_r (x')} \phi ^2  
\left( \frac13 (\varphi ^\varepsilon)^3 - \varphi ^\varepsilon \right) \, dx dt\right\vert
\leq C,
\end{split}
\end{equation*}
where $C>0$ depends only on $r$. Here we used 
$\| \varphi ^\varepsilon \| _{L^\infty} \leq 1$.
By $\inf _{Q_r} \varphi ^{\varepsilon_i } \geq \alpha_1 $, 
the integration by parts, and Young's inequality,
\begin{equation*}
\begin{split}
\int _{Q_r} \phi ^2 \Delta \varphi  ^\varepsilon  
( (\varphi ^\varepsilon)^2 - 1 ) \, dx dt
= & \,
\int _{Q_r} 
-2\phi (\nabla \phi \cdot \nabla \varphi ^\varepsilon) ( (\varphi ^\varepsilon)^2 - 1 )
-2 \phi ^2 \vert \nabla \varphi ^\varepsilon \vert ^2 \varphi ^\varepsilon 
 \, dx dt \\
\leq & \,
\int _{Q_r} 
\alpha_1 \phi ^2 \vert \nabla \varphi ^\varepsilon \vert ^2
+ \frac{1}{\alpha_1} \vert \nabla \phi \vert^2 ( (\varphi ^\varepsilon)^2 - 1 )^2
-2 \phi ^2 \vert \nabla \varphi ^\varepsilon \vert ^2 \alpha_1 
 \, dx dt  \\
 \leq & \,
\frac{1}{\alpha_1} \int _{Q_r} 
\vert \nabla \phi \vert^2 
 \, dx dt 
 \leq C,
\end{split}
\end{equation*}
where $C>0$ depends only on $\alpha_1$, $r$, and $\| \nabla \phi \|_{L^\infty}$.
By $\sqrt{2W (\varphi ^\varepsilon)} ( (\varphi ^\varepsilon)^2 - 1 ) =-2 W (\varphi ^\varepsilon) $,
\begin{equation*}
\begin{split}
& \left \vert \int _{Q_r} \phi ^2 \lambda ^{\varepsilon} \frac{\sqrt{2W(\varphi ^\varepsilon)}}{\varepsilon}  
( (\varphi ^\varepsilon)^2 - 1 ) \, dx dt\right\vert
\leq
\int _{Q_r} \phi ^2 \vert \lambda ^{\varepsilon} \vert \frac{2W(\varphi ^\varepsilon)}{\varepsilon}  
\, dx dt \\
\leq & \,
\int _{t' -r} ^{t' +r} \vert \lambda ^{\varepsilon} \vert 
\int_{B_r (x')}  \frac{2W(\varphi ^\varepsilon)}{\varepsilon}  
\, dx dt \\
\leq & \, 2 \sigma D_1 \sqrt{2r} 
\left(\int _{t' -r} ^{t' +r} \vert \lambda ^{\varepsilon} \vert^2 \, dt\right) ^{\frac12} 
\leq 2 \sigma D_1 \sqrt{2r} \Cr{const:3.1} ^{\frac12} (1+ t')^{\frac12}.
\end{split}
\end{equation*}
Therefore
there exists $C>0$ depending only on $\alpha_1$, 
$r$, $\Cr{const:3.1}$, $\| \nabla \phi \|_{L^\infty}$, $t'$, and $D_1$ 
such that
\[
\int _{Q_r} \phi ^2 \frac{W(\varphi ^\varepsilon)}{\varepsilon ^2} \, dx dt
\leq C.
\]
By \eqref{eq:3.15}, $\mu _t ^\varepsilon (B_{r/2} (x')) \leq 2\sigma ^{-1} \int _{B_{r/2} (x')} \frac{W}{\varepsilon} \, dx$.
Thus
\[
\int _{t' -r} ^{t'+r} \mu _t ^\varepsilon (B_{r/2} (x')) \, dt 
\leq
2 \sigma ^{-1} \int _{Q_r} \phi ^2 \frac{W(\varphi ^\varepsilon)}{\varepsilon } \, dx dt
\leq 2 \sigma ^{-1} \varepsilon C,
\]
where $C>0$ depends only on $\alpha_1$, $r$, $\Cr{const:3.1}$, $\| \nabla \phi \|_{L^\infty}$, $t'$, and $D_1$.
However, this implies $(x',t')  \not \in \spt \mu$. This is a contradiction.
\end{proof}

Set
\begin{equation}\label{defrhor}
\rho _{y} ^r (x) := \frac{1}{( \sqrt{2 \pi } r ) ^{d-1} }e^{-\frac{\vert x-y \vert^2}{2 r^2}}, \qquad r>0, \ x,y\in \mathbb{R}^d.
\end{equation}
Note that $\rho _{(y,s)}(x,t)=\rho _y ^r (x)=\rho _x ^r (y)$ for $r=\sqrt{2(s-t)}$. 

\begin{lemma}\label{lem4.2}
There exist $\gamma _1, \eta_1, \eta _2 \in (0,1)$ depending only on 
$d$, $W$, $T$, $D_2$, and $\Cr{const:3.1}$ such that the following hold. For 
$t,s \in [0,T/2)$ with $0<s-t\leq \eta_1$, we denote $r = \sqrt{2(s-t)}$
and $t' = s+ r^2/2$. If $x \in \Omega$ satisfies
\begin{equation}
\int_{\R^d } \rho ^r _{x} (y) \, d \mu _s (y)
=
\int_{\R^d } \rho _{(x,t')} (y,s) \, d \mu _s (y) <\eta_2,
\label{eq:4.4}
\end{equation}
then $(B_{\gamma _1 r} (x) \times \{ t' \}) \cap \spt \mu = \emptyset$.
\end{lemma}

\begin{proof}
First we remark that $0\leq t < s <t' <T$, $s=\frac{t+t'}{2}$, and
$r=\sqrt{2(s-t)}=\sqrt{2(t'-s)}$.
Assume that $x \in \Omega$ satisfies \eqref{eq:4.4},
$(x' ,t') \in \spt \mu$, and $x' \in B_{\gamma _1 r} (x)$. 
We choose $\gamma _1$, $\eta_1$, and $\eta_2$ later.
Let $\alpha _1 \in (0,1)$ be a constant.
By Lemma \ref{lem4.1}, there exist a sequence $\{ (x_j ,t_j) \}_{j=1} ^\infty$
and a subsequence $\varepsilon _j \to 0$ such that
$\lim _{j\to \infty} (x_j ,t_j) = (x' ,t' )$ and
$ \vert \varphi ^{\varepsilon _j} (x_j ,t_j) \vert <\alpha_1$ for all $j$. 
Then we may assume that for $\alpha' = (\alpha_1 + 1)/2 >\alpha_1$, there exists 
$\gamma _2 =\gamma _2 (W,\alpha_1) >0$
such that
$
\frac{W(\varphi ^{\varepsilon_j} (y ,t_j) )}{\varepsilon _j}
\geq \frac{W(\alpha')}{\varepsilon _j}
$
for any $j$ and for any $y \in B_{\gamma_2 \varepsilon _j } (x_j)$,
because $W(\alpha_1) >W(\alpha')$ and
\[
\vert \varphi ^{\varepsilon_j} (y ,t_j) -\varphi ^{\varepsilon_j} (x_j ,t_j) \vert
\leq \sup _{z \in \Omega} \| \nabla \varphi ^{\varepsilon_j} (z ,t_j) \| \vert y-x_j \vert
\leq \varepsilon _j ^{-1} W(0) \vert y-x_j \vert \leq W(0) \gamma _2
\]
for any $y \in B_{\gamma_2 \varepsilon _j } (x_j)$, where we used \eqref{eq:3.15}.
Thus,
there exists $\eta _3 = \eta _3 (d, \gamma _2)>0$ such that
\begin{equation*}
\begin{split}
\eta_3 \leq \int _{B_{\gamma _2 \varepsilon _j (x_j)}} 
\frac{W(\alpha')}{\varepsilon _j} \rho _{(x_j, t_j +\varepsilon _j ^2)} (y,t_j) \, dy
\leq \int _{B_{\gamma _2 \varepsilon _j (x_j)}} 
\frac{W(\varphi ^{\varepsilon_j} (y ,t_j) )}{\varepsilon _j} \rho _{(x_j, t_j +\varepsilon _j ^2)} (y,t_j) \, dy.
\end{split}
\end{equation*}
Here we used 
\[
\inf _{y \in B_{\gamma _2 \varepsilon} (x_j)} \rho _{x_j, t_j +\varepsilon _j ^2} (y, t_j) 
> \Cl{const:4.2} \varepsilon _j ^{1-d} >0,
\]
where $\Cr{const:4.2}>0$ depends only on $d$ and $\gamma _2$.
By the monotonicity formula, we have
\begin{equation*}
\begin{split}
\eta_3 
\leq
\int _{\R ^d} 
\rho _{(x_j, t_j +\varepsilon _j ^2)} (y,t_j) \, d\mu_{t_j} ^{\varepsilon_j} (y)
\leq
e^{\Cr{const:3.1} (T +1)}
\int _{\R ^d} 
\rho _{(x_j, t_j +\varepsilon _j ^2)} (y,s) \, d\mu_{s} ^{\varepsilon_j}(y) .
\end{split}
\end{equation*}
Choose $\eta _2 = \eta_2 (d,\gamma_2, T, \Cr{const:3.1}) >0$ such that
\begin{equation*}
\begin{split}
2 \eta_2 
\leq
\int _{\R ^d} 
\rho _{(x_j, t_j +\varepsilon _j ^2)} (y,s) \, d\mu_{s} ^{\varepsilon_j} (y)
\end{split}
\end{equation*}
and letting $j\to \infty$, we have
\begin{equation}\label{eq:est-lem4.2}
\begin{split}
2 \eta_2 
\leq
\int _{\R ^d} 
\rho _{(x', t')} (y,s) \, d\mu_{s} (y).
\end{split}
\end{equation}
Changing the center of the backward heat kernel by using \eqref{rhoest1},
we have
\begin{equation*}
\begin{split}
\eta_2 
\leq
\int _{\R ^d} 
\rho _{(x, t')} (y,s) \, d\mu_{s}(y)
\end{split}
\end{equation*}
when $\vert x-x' \vert \leq \gamma _1 r$. Here $\gamma _1$ depends only on $\eta_2$ and $D_2$.
This is a contradiction to \eqref{eq:4.4}. Therefore $(x', t') \not \in \spt \mu$.
\end{proof}

We can also show the following using the estimate \eqref{eq:est-lem4.2}.

\begin{lemma}\label{lem4.5}
There exists $\Cl{const:cor4.5} >0$ depending only on 
$d$, $T$, $\Cr{const:3.1}$, and $D_2$
such that
\begin{equation}\label{eq:3.51}
\mathscr{H}^{d-1} (\spt \mu_t \cap U) \leq \Cr{const:cor4.5} \liminf_{r \downarrow 0} \mu _{t-r^2} (U)
\end{equation}
for any $t \in (0,T) \setminus \tilde B$ and for any open set $U \subset \Omega$, where
$\tilde B$ is the countable set given by Proposition \ref{prop3.13}.
\end{lemma}

\begin{proof}
We need only prove \eqref{eq:3.51} for any compact set $K \subset U$.
Let $X_t:= \{ x \in K \mid (x,t) \in \spt \mu \}$ with $t \in (0,T) \setminus \tilde B$. For any $x \in X_t$,
by \eqref{eq:est-lem4.2}, we have
\begin{equation*}
\begin{split}
2 \eta_2
\leq
\int _{\R ^d} 
\rho _{(x, t)} (y,t-r^2) \, d\mu_{t-r^2} (y).
\end{split}
\end{equation*}
for sufficiently small $r>0$. By \eqref{rhoest2}, we deduce that
\begin{equation*}
\begin{split}
& \int _{\R ^d} 
\rho _{(x, t)} (y,t-r^2) \, d\mu_{t-r^2} (y) \\
\leq & \,
\int _{B _{Lr} (x)} 
\rho _{(x, t)} (y,t-r^2) \, d\mu_{t-r^2} (y)
+2^{d-1} e^{- \frac{3L^2}{8}} D_2
\end{split}
\end{equation*}
for any $L>0$. Therefore for sufficiently large $L>0$
depending only on $d$, $\gamma_2$, $T$, $\Cr{const:3.1}$, and $D_2$, 
we have
\begin{equation*}
\begin{split}
\eta_2
\leq
\int _{B _{Lr} (x)} 
\rho _{(x, t)} (y,t-r^2) \, d\mu_{t-r^2} (y)
\leq (4\pi )^{-\frac{d-1}{2}} r^{1-d} \mu_{t-r^2} (B_{Lr} (x)),
\end{split}
\end{equation*}
where we used $\rho _{(x,t)} (y,t-r^2) \leq (4\pi )^{-\frac{d-1}{2}} r^{1-d} $.
Hence there exists $\Cl{const:cor4.5-1} >0$ depending only on
$d$, $\gamma_2$, $T$, $\Cr{const:3.1}$, and $D_2$ such that
\begin{equation}\label{eq:lem4.5-1}
\begin{split}
\omega _{d-1} r^{d-1} \leq \Cr{const:cor4.5-1} \mu_{t-r^2} (B_{Lr} (x))
\end{split}
\end{equation}
holds for any sufficiently small $r >0$.
Set $\mathcal{B} :=\{ \overline{B} _{Lr} (x) \subset U \mid x \in X_t \}$.
By the Besicovitch covering theorem, there exists a finite sub-collection
$\mathcal{B}_1$, $\mathcal{B}_2$, \dots, $\mathcal{B}_{N(d)}$ such that
each $\mathcal{B}_i$ is a family of the disjoint closed balls and
\begin{equation}\label{eq:lem4.5-2}
X_t \subset \cup _{i=1} ^{N(d)} \cup _{\overline{B} _{Lr} (x_j) \in \mathcal{B}_i} \overline{B} _{Lr} (x_j).
\end{equation}
Let $\mathscr{H}^{d-1} _\delta$ be defined in \cite[Chapter 2]{MR3409135}.
Note that $\mathscr{H}^{d-1} = \lim _{\delta \downarrow 0} \mathscr{H}^{d-1} _\delta$.
By \eqref{eq:lem4.5-1} and \eqref{eq:lem4.5-2}, we compute
\begin{equation*}
\begin{split}
\mathscr{H}^{d-1} _{2Lr} (X_t) 
\leq & \,
\sum _{i=1} ^{N(d)} 
\sum _{\overline{B} _{Lr} (x_j) \in \mathcal{B}_i} \omega_{d-1} (Lr) ^{d-1} \\
\leq & \, 
\sum _{i=1} ^{N(d)} 
\sum _{\overline{B} _{Lr} (x_j) \in \mathcal{B}_i} L^{d-1}
\Cr{const:cor4.5-1} \mu_{t-r^2} (\overline{B}_{Lr} (x_j)) \\
\leq & \, 
\sum _{i=1} ^{N(d)} L^{d-1}
\Cr{const:cor4.5-1} \mu_{t-r^2} (U)
= N(d) L^{d-1} \Cr{const:cor4.5-1} \mu_{t-r^2} (U).
\end{split}
\end{equation*}
Letting $r\downarrow 0$, we have 
$\mathscr{H}^{d-1} (X_t) \leq N(d) L^{d-1} \Cr{const:cor4.5-1} \liminf _{r\downarrow 0} \mu_{t-r^2} (U) $.
By this and \eqref{eq:3.50}, we obtain \eqref{eq:3.51}.
\end{proof}


By Lemma \ref{lem4.2}, we obtain the following lemma.

\begin{lemma}[see \cite{ilmanen1993,MR3348119}]\label{lem4.3}
For $T \in [1,\infty)$, let $\eta_2$ be a constant as in Lemma \ref{lem4.2}.
Set
\[
Z_T := 
\left\{
(x,t) \in \spt \mu \mid 
0\leq t \leq T/2, \ 
\limsup_{s\downarrow t}
\int_{\R^d} \rho _{(y,s)} (x,t) \, d \mu _s (y) \leq \eta_2 /2
\right \}.
\]
Then $\mu (Z_T) =0$ holds.
\end{lemma}

\begin{proof}
Let $\eta _1$, $\eta _2$, and $\gamma _1$ be constants as in Lemma \ref{lem4.2}.
For $\tau \in (0, \eta _1)$, we denote
\[
Z^{\tau} :=
\left\{
(x,t) \in \spt \mu \mid 
0\leq t \leq T/2, \ 
\int_{\R^d} \rho _{(y,s)} (x,t) \, d \mu _s (y) < \eta_2,
\ \text{for any} \ s \in (t, t+\tau ]
\right \}.
\]
Let $\{ \tau _m \}_{m=1} ^\infty$ be a positive sequence with $\tau _m \to 0$ as $m \to \infty$.
Then $Z_T \subset \cup _{m=1} ^{\infty} Z ^{\tau _m}$.
Therefore we need only show $\mu (Z ^\tau) =0$ for any $\tau \in (0,\eta_1)$.
Set
\[
P_\tau (x,t) := 
\{ (x' , t') 
\mid \tau > \vert t-t' \vert >\gamma _1 ^{-2} \vert x-x' \vert ^2\}, \quad x \in \Omega, \ t \in [0, T/2).
\]
We now show that if $(x,t) \in Z^\tau$, then
\begin{equation}\label{eq:4.7}
P_\tau (x,t) \cap Z^\tau =\emptyset.
\end{equation}
Assume that $(x', t') \in P_\tau (x,t) \cap Z^\tau$ for a contradiction. 
First we consider the case of $t' >t$. 
Set $s=\frac{t'+t}{2}$ and $r=\sqrt{t' -t} = \sqrt{2(s-t)}$.
Since $(x,t) \in Z^\tau$,
\[
\int_{\R^d} \rho _{x} ^r (y) \, d \mu _s (y) 
=
\int_{\R^d} \rho _{(y,s)} (x,t) \, d \mu _s (y) < \eta_2.
\]
Therefore Lemma \ref{lem4.2} yields
$(x', t') \not \in \spt \mu$, because $x' \in B_{\gamma _1 r} (x)$ by the definition of $P_\tau (x,t)$.
This yields a contradiction.
In the case of $t' <t$, we can show $(x,t) \not \in \spt \mu$ similarly. This is a contradiction.
Therefore \eqref{eq:4.7} holds.

For $(x_0, t_0) \in \Omega \times [\tau /2 , T/2]$, we denote
\[
Z^{\tau ,x_0 ,t_0} 
= Z^\tau \cap \left( B_{\frac{\gamma_1}{2} \sqrt{\tau}} (x_0) \times (t_0 -\tau /2 , t_0 +\tau /2)\right).
\]
We can choose a countable set $\{ (x_j , t_j ) \}_{j=1} ^\infty$ such that
$ Z^\tau \subset \cup _{j=1} ^\infty Z^{\tau ,x_j ,t_j} $.
Thus we need only prove
$\mu (Z^{\tau ,x_0 ,t_0}) =0$. 
Let $P: \R^{d+1} \to \R^d$ be a projection such that $P(x,t)=x$.
For $\rho \in (0,1)$ and $r \leq \rho$, 
let $\{\overline{B}_{r/5} (x_\lambda) \}_{\lambda\in\Lambda} $ be a covering of
$P(Z^{\tau ,x_0 ,t_0}) \subset B_{\frac{\gamma_1}{2} \sqrt{\tau}} (x_0)$.
Then, we may choose a countable covering  
$\mathcal{F}=\{ \overline{B}_{r} (x_i) \}_{i=1} ^\infty$ 
of $P(Z^{\tau ,x_0 ,t_0})$
with
$(x_i, t_i) \in Z^{\tau ,x_0 ,t_0}$ for some $t_i$, by Vitali's covering theorem.
Let $A$ be a set of centers of all balls in $\{ \overline{B}_{r} (x_i) \}_{i=1} ^\infty$.
Then, by Besicovitch's covering theorem, there exist $N(d)$ and  
subcollections $\mathcal{F} _1$, $\mathcal{F} _2$, \dots, $\mathcal{F} _{N(d)} \subset \mathcal{F}$
of disjoint balls such that
\begin{equation}\label{eq:4.5}
A \subset \cup _{k=1} ^{N (d)} \cup_{B_{k,i} \in \mathcal{F}_k} B_{k,i} . 
\end{equation}
Note that $\mathcal{F}_k$ is finite ($\mathcal{F} _k = \{ B_{k,1}, \dots ,B_{k,n_k} \}$) and
\[
\mathscr{L} ^d (\cup _{i=1} ^{n_k} B_{k,i}) 
=
\sum _{i=1} ^{n_k} \mathscr{L} ^d (B_{k,i})
\leq \mathscr{L}^d ( \overline{B}_{\frac{\gamma_1}{2} \sqrt{\tau} +\rho } (x_0) )
\]
since each balls in $\mathcal{F} _k$ are disjoint and
$B_{k,i} \subset \overline{B}_{\frac{\gamma_1}{2} \sqrt{\tau} +\rho } (x_0)$.
Therefore
\begin{equation}\label{eq:4.6}
\sum_{k=1} ^{N(d)} \sum _{i=1} ^{n_k}
\omega _d r^d 
=\sum _{k=1} ^{N(d)} \sum _{B_{k,i} \in \mathcal{F} _k} \mathscr{L} ^d (B_{k,i})
\leq N(d) \mathscr{L}^d ( \overline{B}_{\frac{\gamma_1}{2} \sqrt{\tau} +\rho } (x_0) ) =:N'.
\end{equation}
If $(x,t) \in Z^{\tau ,x_0 ,t_0}$, then there exists 
$B_{k,i} = \overline{B}_r (x_{k,i}) \in \mathcal{F}_k$ for some $k$ and $i$
such that $x \in \overline{B}_{2r} (x_{k,i})$ and
 $ \vert t_{k,i} -t \vert \leq \gamma_1 ^{-1} \vert x_{k,i}  -x \vert ^2 \leq 4 \gamma_1 ^{-1}  r ^2$ by \eqref{eq:4.7}
 and \eqref{eq:4.5} 
 (note that we should change the radius because
 $A$ is not a covering of $Z^{\tau ,x_0 ,t_0}$).
Hence, we have
\begin{equation*}
\begin{split}
Z^{\tau ,x_0 ,t_0}
\subset 
\cup _{k=1} ^{N(d)} \cup _{i=1} ^{n_k} 
\overline{B}_{2r} (x_{k,i}) \times (t_{k,i} - 4 r ^2 \gamma _1 ^{-2}, t_{k,i} + 4 r ^2 \gamma _1 ^{-2})
\end{split}
\end{equation*}
By this, \eqref{eq:3.22}, and \eqref{eq:4.6} we obtain
\begin{equation*}
\begin{split}
\mu (Z^{\tau ,x_0 ,t_0})
\leq & \,
\sum _{k=1} ^{N(d)} \sum _{i=1} ^{n_k} 
\mu (\overline{B}_{2r} (x_i) \times (t_i - 4 r ^2 \gamma _1 ^{-2}, t_i + 4 r ^2 \gamma _1 ^{-2})) \\
\leq & \, 
\sum _{k=1} ^{N(d)} \sum _{i=1} ^{n_k}
D_2 (2r) ^{d-1} \times 8 \gamma _1 ^{-2} r ^2
\leq 2^{d+2} \gamma _1 ^{-2} \omega_d ^{-1} N' D_2 \rho,
\end{split}
\end{equation*}
where we used \eqref{eq:4.6}.
Letting $\rho \to 0$, we have $\mu (Z^{\tau ,x_0 ,t_0})=0$. Thus $\mu (Z_T) =0$ holds. 
\end{proof}

\begin{theorem}[see \cite{ilmanen1993}]\label{thm4.5}
We see that $\vert \xi \vert =0$ and $\lim _{i \to \infty} \vert \xi _t ^{\varepsilon _i} \vert (\Omega)=0$ for a.e. $t\in [0,T)$.
\end{theorem}

\begin{proof}
First we show that
\begin{equation}\label{eq:4.8}
\int _{\Omega \times (0,s) } \frac{\rho _{(y,s)} (x,t) }{s-t} \, d \vert \xi \vert (x,t) \leq C
\end{equation}
for some $C>0$. By \eqref{eq:3.15} and \eqref{eq:3.20},
 integrating \eqref{eq:3.21} on $(0, s-\delta)$ with $\delta>0$,
we obtain
\begin{equation*}
\begin{split}
& \int _0 ^{s-\delta}
\frac{1}{2\sigma (s-t)} \int_{\mathbb{R}^d} \rho _{(y,s)} (x,t) 
\left \vert
\frac{W(\varphi ^\varepsilon)}{\varepsilon}
-
\frac{\varepsilon \vert \nabla \varphi^\varepsilon \vert ^2 }{2} \right\vert
\, dx dt \\
\leq & \,
\left( 1+ e^{\Cr{const:3.1} (s+1)}  \frac12 \int_{0} ^{s-\delta} \vert \lambda ^\varepsilon \vert ^2 \, dt \right)
\int _{\R ^d} \rho _{(y,s)} (x,0) \, d \mu _0 ^\varepsilon.
\end{split}
\end{equation*}
Letting $\delta \to 0$ and $\varepsilon \to 0$, we obtain \eqref{eq:4.8}.
Next, integrating \eqref{eq:4.8} on $\Omega \times (0,T)$ by $d\mu _s ds$ we have
\[
\int _{\Omega \times (0,T)}
\left(
\int _{\Omega \times (t,T)}
\frac{\rho _{(y,s)} (x,t) }{s-t}  \, d\mu _s (y) ds
\right)
d \vert \xi \vert (x,t) 
\leq CD_1 T,
\]
where we used Fubini's theorem. Then this boundedness implies
\begin{equation}\label{eq:4.8.1}
\int _{\Omega \times (t,T)}
\frac{\rho _{(y,s)} (x,t) }{s-t}  \, d\mu _s (y) ds
<\infty
\qquad \text{for} \ \vert \xi \vert \text{-a.e.} \ (x,t) \in \Omega \times (0,T).
\end{equation}
Next we claim
\begin{equation}\label{eq:4.9}
a(x,t):= \limsup_{s\downarrow t} \int _{\Omega} \rho _{(y,s)} (x,t) \, d\mu _s (y)=0
\qquad \text{for} \ \vert \xi \vert \text{-a.e.} \ (x,t) \in \Omega \times (0,T).
\end{equation}
Define $\beta:= \log (s-t)$ and
\[
h (s) := \int _{\Omega} \rho _{(y,s)} (x,t) \, d \mu _s (y).
\]
Assume that $(x,t)$ satisfies \eqref{eq:4.8.1}. Then 
\begin{equation}\label{eq:4.9.1}
\int _{-\infty} ^{\log (T-t)} h (t+ e^\beta) \, d \beta
<\infty.
\end{equation}
Let $\theta \in (0,1]$ and $\{ \beta _i \} _{i=1} ^\infty$ 
be a negative monotone decreasing sequence
such that
\[
\beta _i \downarrow -\infty, \quad 
0< \beta_i - \beta_{i+1} \leq \theta, \quad \text{and} \quad
h(t + e^{\beta_i}) \leq \theta.
\]
For any $\beta \in (-\infty, \beta_1)$, choose $i$ such that $\beta \in [\beta_{i} ,\beta_{i-1})$ holds.
One can check that
\begin{equation}\label{eq:4.9.4}
\sup _{y \in B_{Mr} (x)} 
\frac{\rho _{(y,t+2 e^\beta - e^{\beta_i })} (x,t )}{ \rho _{(y,t+ e^{\beta_i })} (x,t )}
\leq e^{M^2 (1-e ^{\beta -\beta_i})} \leq  e^{M^2 (1-e ^{\theta})}
\end{equation}
for $M>0$,
where $r=\sqrt{2(2 e ^{\beta} -e ^{\beta_i})}$.
We compute
\begin{equation}\label{eq:4.9.2}
\begin{split}
& h(t+e^\beta)
= \int _{\Omega} \rho _{(y,t+e^\beta)} (x,t) \, d \mu_{t+e^\beta} (y)
= \int _{\Omega} \rho _{(y,t+2 e^\beta)} (x,t + e^\beta) \, d \mu_{t+e^\beta} (y)\\
\leq & \,
e^{\Cr{const:3.1} (\beta-\beta_i +1) }
\int _{\Omega} \rho _{(y,t+2 e^\beta)} (x,t + e^{\beta_i }) \, d \mu_{t+e^{\beta_i}} (y) \\
\leq & \,
e^{2 \Cr{const:3.1} }
\int _{\Omega} \rho _{(y,t+2 e^\beta - e^{\beta_i })} (x,t ) \, d \mu_{t+e^{\beta_i}} (y) \\
\leq & \,
e^{2 \Cr{const:3.1} }
\int _{B_{Mr} (x)} \rho _{(y,t+2 e^\beta - e^{\beta_i })} (x,t ) \, d \mu_{t+e^{\beta_i}} (y) 
+ e^{2 \Cr{const:3.1} } 2^{d-1} e^{-\frac{3M^2}{8}} D_2\\
\leq & \,
e^{2 \Cr{const:3.1} } e^{M^2 (1-e ^{\theta})}
\int _{B_{Mr} (x)} \rho _{(y,t+ e^{\beta_i })} (x,t ) \, d \mu_{t+e^{\beta_i}} (y) 
+ e^{2 \Cr{const:3.1} } 2^{d-1} e^{-\frac{3M^2}{8}} D_2\\
\leq & \,
e^{2 \Cr{const:3.1} } e^{M^2 (1-e ^{\theta})} \theta
+ e^{2 \Cr{const:3.1} } 2^{d-1} e^{-\frac{3M^2}{8}} D_2,
\end{split}
\end{equation}
where we used \eqref{eq:3.20}, \eqref{rhoest2}, and
\begin{equation*}
\begin{split}
\int _{\Omega} \rho _{(y,t+e^{\beta_i })} (x,t ) \, d \mu_{t+e^{\beta_i}} (y)
=h(t+e^{\beta_i}) \leq \theta.
\end{split}
\end{equation*}
Thus, for any $\delta >0$, we can choose $\theta \in (0,1]$ and $M>0$
such that $h (t + e^{\beta}) \leq \delta$ for any $\beta < \beta_1$.
This proves \eqref{eq:4.9}.
Set 
\[
A:= \{ (x,t) \in \Omega \times (0,T) \mid a(x,t)=0 \} 
\ \ \text{and} \ \
B:=\{ (x,t) \in \Omega \times (0,T) \mid a(x,t)>0 \}.
\] 
Then 
$\Omega \times (0,T)= A\cup B$ and 
$\vert \xi \vert (B)=0$ by \eqref{eq:4.9}.
Moreover, Lemma \ref{lem4.3} and \eqref{rhoest2} imply
$\mu (A) =0$ and thus $\vert \xi \vert (A) =0$, because 
$\vert \xi \vert$ is absolute continuous with respect to $\mu$.
Therefore $\vert \xi \vert (\Omega\times (0,T)) =0$.
The rest of the claim can be shown from the dominated convergence theorem.
\end{proof}

\subsection{Rectifiability}

Next we show the rectifiability of $\mu_t$.

\begin{definition}
For $\phi \in C_c (G_{d-1} (\Omega))$, we define
$V_t ^\varepsilon \in \mathbb{V} _{d-1} (\Omega)$ by
\begin{equation}\label{eq:4.19}
V_t ^\varepsilon (\phi):=
\int _{\Omega \cap \{ \vert \nabla \varphi ^\varepsilon (x,t) \vert \not=0 \}}
\phi \left(x, I -\frac{\nabla \varphi ^\varepsilon (x,t) }{ \vert \nabla \varphi ^\varepsilon (x,t) \vert} 
\otimes \frac{\nabla \varphi ^\varepsilon (x,t) }{\vert \nabla \varphi ^\varepsilon (x,t) \vert}
\right) \, d\mu _t ^\varepsilon (x).
\end{equation}
Here, $\varphi ^\varepsilon$ is a solution to \eqref{ac}.
\end{definition}

Note that the first variation of $V_t ^\varepsilon$ is given by
\begin{equation*}
\begin{split}
\delta V_t ^\varepsilon (\vec{\phi})
= & \,
\int _{G_{d-1} (\Omega)} 
\nabla \vec{\phi} (x) \cdot S \, dV_t ^\varepsilon (x,S) \\
= & \,
\int _{\Omega \cap \{ \vert \nabla \varphi ^\varepsilon (x,t) \vert \not=0 \}}
\nabla \vec{\phi} (x) \cdot \left(I -\frac{\nabla \varphi ^\varepsilon (x,t) }{\vert \nabla \varphi ^\varepsilon (x,t) \vert} 
\otimes \frac{\nabla \varphi ^\varepsilon (x,t) }{\vert \nabla \varphi ^\varepsilon (x,t) \vert}
\right) \, d\mu _t ^\varepsilon (x)
\end{split}
\end{equation*}
for $\vec{\phi} \in C_c ^1 (\Omega; \R^d)$.
By the integration by parts, we have
\begin{equation}\label{eq:4.20}
\begin{split}
\delta V_t ^\varepsilon (\vec{\phi})
= & \,
\int _{\Omega} (\vec{\phi} \cdot \nabla \varphi ^\varepsilon)
\left( \varepsilon \Delta \varphi ^\varepsilon -\frac{W'(\varphi ^\varepsilon)}{\varepsilon} \right)
\, dx
-\int _{\Omega \cap \{ \vert \nabla \varphi ^\varepsilon (x,t) \vert \not=0 \}}
\frac{W(\varphi ^\varepsilon)}{\varepsilon} \div \vec{\phi} \, dx \\
& + \int _{\Omega \cap \{ \vert \nabla \varphi ^\varepsilon (x,t) \vert \not=0 \}}
\nabla \vec{\phi} \cdot 
\left(
\frac{\nabla \varphi ^\varepsilon (x,t) }{\vert \nabla \varphi ^\varepsilon (x,t) \vert} 
\otimes \frac{\nabla \varphi ^\varepsilon (x,t) }{\vert\nabla \varphi ^\varepsilon (x,t) \vert}
\right)
\xi _\varepsilon
 \, dx.
\end{split}
\end{equation}
Note that the second and third terms of the right hand side converges to $0$
for a.e. $t \in [ 0,T)$ by Theorem \ref{thm4.5}.
By \eqref{eq:3.1} and \eqref{eq:3.5}, we have
\[
\sup _{i \in \N } \int _0 ^T \int _{\Omega} \varepsilon_i 
\left( \Delta \varphi ^{\varepsilon_i } -\dfrac{W' (\varphi ^{\varepsilon_i})}{\varepsilon_i ^2 } \right)^2 
\,dxdt \leq C
\]
for some $C>0$ (see the proof of Theorem \ref{thm3.11}).
Thus Fatou's lemma implies
\begin{equation}\label{eq:4.21}
\liminf_{i \to \infty}
\int _\Omega \varepsilon _i 
\left( \Delta \varphi ^{\varepsilon_i} -\frac{W'(\varphi ^{\varepsilon_i})}{\varepsilon_i ^2} \right)^2
\, dx < \infty
\end{equation}
for a.e. $t \in[ 0, T)$. Hence, \eqref{eq:4.20} and \eqref{eq:4.21} show that
\begin{equation}\label{eq:4.22}
\begin{split}
& \liminf _{i\to \infty}
\vert \delta V_t ^{\varepsilon _i} (\vec{\phi}) \vert \\
\leq & \, \liminf_{i\to \infty}
\left( \int _{\Omega} \varepsilon_i \vert \nabla \varphi ^{\varepsilon _i} \vert^2 \, dx \right) ^{\frac12} 
\left(
\int _\Omega \varepsilon _i 
\left( \Delta \varphi ^{\varepsilon_i} -\frac{W'(\varphi ^{\varepsilon_i})}{\varepsilon_i ^2} \right)^2
\, dx
\right) ^{\frac12} \\
\leq & \,
D_1 ^{\frac12} 
\liminf_{i\to \infty}
\left(
\int _\Omega \varepsilon _i 
\left( \Delta \varphi ^{\varepsilon_i} -\frac{W'(\varphi ^{\varepsilon_i})}{\varepsilon_i ^2} \right)^2
\, dx
\right) ^{\frac12}
<\infty
\end{split}
\end{equation}
for a.e. $t \in[ 0, T)$ and for any $\vec{\phi} \in C_c ^1 (\Omega;\R^d)$ with $\sup \vert \vec{\phi} \vert \leq 1$.
Let $t \in[ 0, T) \setminus \tilde B$ satisfy \eqref{eq:4.22}, 
where $\tilde B$ is given by Proposition \ref{prop3.13}.
Taking a subsequence $i_j \to \infty$, 
there exists a varifold $V_t$ such that
$V_t ^{\varepsilon _{i_j}} \to V_t$ as Radon measures and 
$\delta V_t$ is a Radon measure by \eqref{eq:4.22}. 
In addition, Proposition \ref{prop3.13}, Lemma \ref{lem4.5}, 
and the standard measure theoretic argument imply
\[
V_t =V_t \lfloor_{\{ x \in \Omega \mid \limsup_{r\downarrow 0} r^{1-d} \| V_t \| (B_r (x)) >0 \} 
\times \mathbb{G}(d,d-1)}.
\] 
Therefore Allard's rectifiability theorem yields the following theorem.
\begin{theorem}\label{thm4.6}
For a.e. $t \geq 0$, $\mu_t$ is rectifiable.
In addition, for a.e. $t \geq 0$, $\mu _t$ has a generalized 
mean curvature vector $\vec{h} (\cdot,t)$ with
\[
\delta V_t (\vec{\phi}) =- \int _{\Omega} \vec{\phi} \cdot h (\cdot,t) \, d\mu _t
=\lim _{i \to \infty}
\int _{\Omega} (\vec{\phi} \cdot \nabla \varphi ^{\varepsilon_i})
\left( \varepsilon_i \Delta \varphi ^{\varepsilon_i} -\frac{W'(\varphi ^{\varepsilon_i})}{\varepsilon_i} \right)
\, dx
\]
and
\[
\int _\Omega \phi \vert \vec{h} \vert^2 \, d\mu _t
\leq
\frac{1}{\sigma} \liminf_{i\to \infty}
\int _{\Omega} \varepsilon _i \phi 
\left( 
\Delta \varphi ^{\varepsilon_i} -\frac{W'(\varphi ^{\varepsilon_i})}{\varepsilon_i ^2} 
\right)^2
\, dx <\infty
\]
for any $\phi \in C_c (\Omega ;[0,\infty))$ and $\vec \phi \in C_c (\Omega ;\R^d)$.
\end{theorem}
Detailed proof of this is in \cite{ilmanen1993, takasao-tonegawa}, so we omit it
(however, the essential part has already been discussed above).

\subsection{Integrality}
To prove the integrality, we mainly follow \cite{MR1803974, takasao-tonegawa, MR2040901}.
The propositions that are directly applicable to our problem are in Appendix 
for readers' convenience. Let $\{ 	r _i\} _{i=1} ^\infty$ be a positive
sequence with $r _i \to 0$ and $\frac{\varepsilon _i}{r_i} \to 0$ as $i \to \infty$.
Set $u ^{\tilde \varepsilon} (\tilde x, \tilde t) = \varphi ^\varepsilon (x,t)$
and $g ^{\tilde \varepsilon} (\tilde t) = r \lambda ^\varepsilon (t)$
for $\tilde x= \frac{x}{r}$, $\tilde t= \frac{t}{r^2}$, 
and $\tilde \varepsilon = \frac{\varepsilon}{r}$.
Then, $u ^{\tilde \varepsilon}$ is a solution to
\begin{equation}\label{eq:u}
\tilde \varepsilon u ^{\tilde \varepsilon} _{\tilde t} =\tilde \varepsilon \Delta_{\tilde x} u ^{\tilde \varepsilon} 
-\dfrac{W' (u^{\tilde \varepsilon})}{\tilde \varepsilon }
+ g ^{\tilde \varepsilon} \sqrt{2W(u ^{\tilde \varepsilon})} .
\end{equation}
We remark that the monotonicity formula \eqref{eq:3.20} and the upper bound
of the density \eqref{eq:3.22} hold for 
$d\tilde \mu_{\tilde t} ^{\tilde \varepsilon} (\tilde x)= \sigma ^{-1} 
(\frac{\tilde \varepsilon \vert \nabla_{\tilde x} u ^{\tilde \varepsilon} \vert^2}{2} +
\frac{W(u ^{\tilde \varepsilon})}{\tilde \varepsilon} )\, d\tilde x$, 
because the value
\[
\int_{\mathbb{R}^d} \rho _{(y,s)} (x,t) \, d\mu _t ^\varepsilon (x)
\]
is invariant under this rescaling, and for any $s>0$ we have
\[ 
\frac{1}{s^{d-1}}
\int _{B_s (0)} \left( \frac{\tilde \varepsilon \vert \nabla_{\tilde x} u ^{\tilde \varepsilon} \vert^2}{2} +
\frac{W(u ^{\tilde \varepsilon})}{\tilde \varepsilon} \right) \, d\tilde x
=
\frac{1}{(sr)^{d-1}}\int _{B_{sr} (0)} \left( \frac{\varepsilon \vert\nabla \varphi ^{\varepsilon}\vert^2}{2} +
\frac{W(\varphi ^{\varepsilon})}{\varepsilon} \right) \, dx
\leq \sigma D_2
\]
by \eqref{eq:3.22}. 
We subsequently drop $\tilde \cdot$ for simplicity.
First we consider the energy estimate on 
$\{x \in B_1 (0) \mid \vert u ^\varepsilon (x,t) \vert \geq 1 -b  \}$.

\begin{proposition}[See \cite{MR2040901}]\label{prop4.7}
For any $s >0$ and $a \in (0,T)$, there exist positive constants $b$ and $\epsilon_2$
depending only on $D_1$, $D_2$, $\Cr{const:3.1}$, $a$, $\alpha$, and $s$ such that
\[
\int_{\{ x \in B_1 (0) \mid \vert u ^\varepsilon (x,t) \vert \geq 1-b \}}
\frac{W( u ^\varepsilon (x,t))}{\varepsilon } \, dx \leq s
\]
for all $t \in (a, T)$ whenever $\varepsilon \in (0,\epsilon_2)$.
\end{proposition}
To prove Proposition \ref{prop4.7}, we prepare following two lemmas.
\begin{lemma}[See \cite{MR2040901}]\label{lem4.8}
For any $\delta \in (0,T)$, there exist positive constants $\Cl{const:lem4.7}$ and $\epsilon_3$ depending only
on $d$, $\delta$, $\alpha$, and $\Cr{const:initial}$ with the following property.
Assume that there exist $(x_0 ,t_0) \in B_1 (0) \times (\delta ,T)$
and $\gamma \in (0,\frac23]$ such that
\begin{equation}\label{eq:4.23}
u ^\varepsilon (x_0 ,t_0) <1 -\varepsilon ^\gamma
\qquad 
(\text{or} \ u ^\varepsilon (x_0 ,t_0) > -1 + \varepsilon ^\gamma)
\end{equation}
and
\begin{equation}\label{eq:4.24}
1\leq \tilde r := \Cr{const:lem4.7} \gamma \vert \log \varepsilon \vert
\leq \varepsilon ^{-1} \min \left\{ \sqrt{\frac{\delta}{2}} , \ \frac12 \right\}.
\end{equation}
Then 
\[
\inf _{B_{\varepsilon \tilde r} (x_0) \times (t_0 -\varepsilon ^2 \tilde r^2, t_0)} u ^\varepsilon
< \frac12 
\qquad
\left( 
\text{resp.} \ 
\sup _{B_{\varepsilon \tilde r} (x_0) \times (t_0 -\varepsilon ^2 \tilde r^2, t_0)} u ^\varepsilon
> - \frac12 
\right) 
\]
for any $\varepsilon \in (0,\epsilon_3)$.
\end{lemma}

\begin{proof}
We may assume that 
$B_{\varepsilon \tilde r} (x_0) \times (t_0 -\varepsilon ^2 \tilde r^2, t_0) \subset B_2 (0)
\times (0,T)$ by \eqref{eq:4.24}.
We consider the rescaling of \eqref{eq:u}
by $\tilde x= \frac{x-x_0}{\varepsilon}$ and $\tilde t= \frac{t-t_0}{\varepsilon ^2}$.
Then we obtain
\begin{equation}\label{eq:4.25}
\tilde u ^{\varepsilon} _{\tilde t} 
=
\Delta_{\tilde x} \tilde u ^{\varepsilon} 
- W' (\tilde u ^{\varepsilon})
+ \varepsilon \tilde g ^{\varepsilon} \sqrt{2W(\tilde u ^\varepsilon)} ,
\qquad
(x,t) \in B_{\tilde r} (0) \times (-\tilde r ^2 ,0),
\end{equation}
where $\tilde u ^{\varepsilon} (\tilde x,\tilde t) = u ^\varepsilon (x,t)$ and
$\tilde g ^{\varepsilon} (\tilde t) = g ^\varepsilon (t)$.
Note that \eqref{eq:1.7} and $\mathscr{L}^d (\Omega)=1$ yield
\begin{equation}\label{eq:4.26}
 \| \varepsilon \tilde g ^{\varepsilon} \|_{L^\infty} \leq \frac43 \varepsilon ^{1- \alpha}
\end{equation}
for $\alpha \in (0,1)$.
Let $\psi $ be a function with
\begin{equation}\label{eq:4.27}
\left\{ 
\begin{array}{ll}
\psi _{\tilde t} \, & \geq \Delta _{\tilde x} \psi -\frac{1}{10} \psi \qquad \text{on} \ \R^d \times (-\infty ,0), \\
\psi (\tilde x,\tilde t) \,  & \geq   
e^{\frac{\vert \tilde x \vert+\vert \tilde t \vert}{\Cl{const:lem4.7-2}}} \qquad \text{on} \ (\R^d \times (-\infty ,0))
\setminus B_1 (0,0), \\
\psi (0,0) \, & = 1,
\end{array} \right.
\end{equation}
for some constant $\Cr{const:lem4.7-2} >0$.
For example, $\psi = e^{-\frac{\tilde t}{100} -1} 
e^{\frac{1}{100d}\sqrt{1+\vert \tilde x \vert^2}}$ satisfies \eqref{eq:4.27}.
Set $\tilde r := \Cr{const:lem4.7-2} \gamma \vert \log \varepsilon \vert$. We may assume that $\tilde r\geq 1$
for sufficiently small $\varepsilon$. 
Note that
\begin{equation}\label{eq:4.28}
1- \varepsilon ^\gamma e^{\frac{\tilde r}{\Cr{const:lem4.7-2}}} =0.
\end{equation}
The assumption \eqref{eq:4.23} is equivalent to
\begin{equation}\label{eq:4.29}
\tilde u ^\varepsilon (0 ,0) <1 -\varepsilon ^\gamma.
\end{equation}
For a contradiction, we assume that
\begin{equation}\label{eq:4.30}
\inf _{B_{\tilde r} (0) \times (-\tilde r ^2 , 0)} \tilde u ^\varepsilon \geq \frac12.
\end{equation}
Set $\phi^\varepsilon := 1 - \varepsilon ^\gamma \psi$. Then \eqref{eq:4.27} and \eqref{eq:4.29} imply
\[
\phi ^\varepsilon _{\tilde t} \leq \Delta _{\tilde x} \phi ^\varepsilon
+\frac{1}{10} (1- \phi ^\varepsilon) \qquad \text{on} \ \R^d \times (-\infty ,0)
\]
and
\begin{equation}\label{eq:4.31}
\phi ^\varepsilon (0,0) =1- \varepsilon ^\gamma \psi (0,0) 
=1 -\varepsilon ^\gamma
> \tilde u ^\varepsilon (0 ,0).
\end{equation}
Moreover, by $\tilde r \geq 1$, 
\[
\psi \geq e^{\frac{\vert \tilde x \vert + \vert \tilde t \vert}{\Cr{const:lem4.7-2}}} \geq e^{\frac{\tilde r}{\Cr{const:lem4.7-2}}}
\qquad \text{on} \ \partial ( B_{\tilde r} (0) \times (-\tilde r ^2 ,0)).
\]
Therefore
\begin{equation}\label{eq:4.32}
\phi ^\varepsilon = 1 - \varepsilon ^\gamma \psi \leq 1- \varepsilon ^\gamma e^{\frac{\tilde r}{\Cr{const:lem4.7-2}}}
= 0 <\frac12 \leq \tilde u ^\varepsilon
\qquad \text{on} \ \partial ( B_{\tilde r} (0) \times (-\tilde r ^2 ,0))
\end{equation}
by \eqref{eq:4.28} and \eqref{eq:4.30}.
We consider a function $w = \phi ^\varepsilon - \tilde u^\varepsilon $ on 
$B_{\tilde r} (0) \times (-\tilde r ^2 ,0)$.
By \eqref{eq:4.31} and \eqref{eq:4.32},
$w$ attains its positive maximum at an interior point $(x',t') \in B_{\tilde r} (0) \times (-\tilde r ^2 ,0)$,
and hence $w_{\tilde t} - \Delta _{\tilde x} w \geq 0$ and $w >0$ at $(x',t')$.
At $(x',t')$, we compute that
\begin{equation*}
\begin{split}
0 \leq & \,  w_{\tilde t} - \Delta _{\tilde x} w 
\leq \frac{1}{10} (1- \phi ^\varepsilon) 
+ W' (\tilde u ^\varepsilon ) - \varepsilon \tilde g ^\varepsilon 
\sqrt{2W (\tilde u ^\varepsilon)} \\
= & \,  \frac{1}{10} (1- \phi ^\varepsilon)
-2 \tilde u ^\varepsilon 
(1- (\tilde u ^\varepsilon)^2) 
- \varepsilon \tilde g ^\varepsilon (1- (\tilde u ^\varepsilon)^2) \\
\leq & \,  \frac{1}{10} (1- \phi ^\varepsilon)
 +(-1 + \frac83 \varepsilon^{1-\alpha})
(1- (\tilde u ^\varepsilon)^2) \\
\leq & \,  \frac{1}{10} (1- \phi ^\varepsilon)
 + \frac32 (-1 + \frac83 \varepsilon^{1-\alpha})
(1- \phi ^\varepsilon) <0
\end{split}
\end{equation*}
for sufficiently small $\varepsilon$,
where we used \eqref{eq:4.26} 
and $1>\phi ^\varepsilon > \tilde u ^\varepsilon \geq \frac12$ at $(x',t')$.
This is a contradiction. The other case can be proved similarly.
\end{proof}

\begin{lemma}[See \cite{MR2040901}]\label{lem4.9}
For any $\delta \in (0,T)$, 
there exist positive constants $\Cl{const:lem4.8}$ and $\epsilon _4$
depending only on $\delta$, $\alpha$, $d$, $\Cr{const:3.1}$, and $D_2$
such that the following holds.
For $t \in (\delta, T)$ and $r \in (0,\frac12)$, set
\[
Z_{r,t_0} :=
\left\{ 
x_0 \in B_1 (0) \mid \inf _{B_r (x_0) \times ( t_0 -r^2 ,t_0)} \vert u ^\varepsilon \vert <\frac12 
\right\}.
\]
Then for any $\varepsilon \in (0,\epsilon_4)$, we have
\begin{equation}\label{eq:4.33}
\mathscr{L}^{d} (Z_{r,t_0}) \leq \Cr{const:lem4.8} r, \qquad
\varepsilon \leq r <\frac12.
\end{equation}
\end{lemma}

\begin{proof}
First we claim that
there exist some constants $\epsilon _4$, $\Cl{const:lem4.9-1}$, and $\Cl{const:lem4.9-2}$
such that if $x_0 \in Z_{r, t_0}$ and $\varepsilon \in (0,\epsilon _4)$ then
\begin{equation}\label{eq:4.34}
\sigma \mu _{t_0 -2r^2} ^\varepsilon (B_{\Cr{const:lem4.9-1} r } (x_0))
=
\left. \int _{B_{\Cr{const:lem4.9-1} r } (x_0)} 
\frac{\varepsilon |\nabla u ^\varepsilon|^2}{2} +
\frac{W(u ^\varepsilon)}{\varepsilon} \, dx \right\vert_{t=t_0-2r^2} \geq 
\Cr{const:lem4.9-2} r^{d-1}.
\end{equation}
holds for any $r \in [\varepsilon ,\frac12)$.
We may assume that $(x_1, t_1) \in B_r (x_0) \times (t_0 -r^2, t_0)$ with 
$\vert u ^\varepsilon (x_1 ,t_1) \vert <\frac12$.
By the monotonicity formula \eqref{eq:3.20}, for any $\varepsilon \in (0,\epsilon_1)$ we have
\begin{equation}\label{eq:4.35}
\int_{\mathbb{R}^d} \rho _{(x_1, t_1 +\varepsilon ^2)} (x,t) \, d\mu _t ^\varepsilon (x) \Big\vert_{t=t_1}
\leq \left( \int_{\mathbb{R}^d} \rho _{(x_1, t_1 +\varepsilon ^2)} (x,t) \, d\mu _t ^\varepsilon (x) 
\Big\vert_{t=t_0 -2r^2} \right)
e^{\Cr{const:3.1} (3r^2 +1)}.
\end{equation}
By $\vert u ^\varepsilon (x_1 ,t_1)\vert <\frac12$, repeating the proof of Lemma \ref{lem4.2},
there exists $ \eta =\eta (\alpha ,d) >0$ such that
\begin{equation}\label{eq:4.36}
\eta \leq \int_{\mathbb{R}^d} \rho _{(x_1, t_1 +\varepsilon ^2)} (x,t) \, d\mu _t ^\varepsilon (x) \Big\vert_{t=t_1}.
\end{equation}
Then \eqref{eq:4.35}, \eqref{eq:4.36}, and \eqref{rhoest2} imply
\[
\eta' \leq 
\int_{B_R (x_1) } \rho _{(x_1, t_1 +\varepsilon ^2)} (x,t) \, d\mu _t ^\varepsilon (x) 
\Big\vert_{t=t_0 -2r^2}
+ 2^{d-1} e^{-\frac{3R^2}{16(t_1 +\varepsilon ^2 -t_0 +2 r^2)}} D_2,
\]
where $\eta' =\eta'(\alpha,d,\Cr{const:3.1}) >0$.
By $\vert t_1 -t_0 \vert <r^2$ and $\varepsilon \leq r$, 
we have $e^{-\frac{3R^2}{16(t_1 +\varepsilon ^2 -t_0 +2 r^2)}} \leq e^{-\frac{3R^2}{64 r^2}}$.
Thus there exists $\gamma >0$ depending only on $\alpha,d,\Cr{const:3.1}, D_2$ such that
\[
\frac{\eta'}{2} \leq 
\int_{B_{\gamma r} (x_1) } \rho _{(x_1, t_1 +\varepsilon ^2)} (x,t_0 -2r^2) \, d\mu _{t_0 -2r^2} ^\varepsilon (x).
\]
Note that since $t_1 +\varepsilon ^2 - (t_0 -2r^2) \geq 2r^2$
there exists $C>0$ depending only on $d$ such that
\[
\rho _{(x_1, t_1 +\varepsilon ^2)} (x,t_0 -2r^2)
\leq \frac{C}{r^{d-1}}.
\]
Hence we obtain \eqref{eq:4.34} for some $\Cr{const:lem4.9-1}$, and $\Cr{const:lem4.9-2}$.
Finally we prove \eqref{eq:4.33}. 
The inequality \eqref{eq:4.34} yields that there exists $\Cl{const:lem4.9-3}>0$ 
depending only on $\alpha,d,\Cr{const:3.1}, D_2$ such that
\begin{equation}\label{eq:4.37}
\mathscr{L} ^d (\overline{B}_{\Cr{const:lem4.9-1} r} (x_0)) \leq 
r \Cr{const:lem4.9-3} \mu _{t_0 -2 r^2} ^\varepsilon (\overline{B}_{\Cr{const:lem4.9-1} r} (x_0))
\end{equation}
for any $x _0 \in Z_{r,t_0}$ and $r \in [\varepsilon ,\frac12)$.
Set $\tilde r := \Cr{const:lem4.9-1} r$.
By an argument similar to that in the proof of Lemma \ref{lem4.3}, 
there exist $\mathcal{F}_1,\dots ,\mathcal{F} _{N(d)}$ such that
$N(d)$ depends only on $d$,
$\mathcal{F} _k= \{ \overline{B}_{\tilde r} (x_{k,1}), \dots, \overline{B}_{\tilde r} (x_{k,n_k}) \}$
is a family of disjoint closed balls for any $k$, and
\[
Z_{r, t_0} \subset \cup _{k=1} ^{N(d)} \cup _{i=1} ^{n_k} \overline{B}_{2\tilde r} (x_{k,i}), \qquad
x_{k,i} \in Z_{r,t_0} \quad \text{for any} \ k  \ \text{and} \ i. 
\]
Therefore
\begin{equation*}
\begin{split}
\mathscr{L}^d (Z_{r,t_0}) \leq & \,
\sum_{k=1} ^{N(d)} \sum _{i=1} ^{n_k} \mathscr{L}^d ( \overline{B} _{2\tilde r} (x_{k,i}) )
=
2^d \sum_{k=1} ^{N(d)} \sum _{i=1} ^{n_k} \mathscr{L}^d ( \overline{B} _{\tilde r} (x_{k,i}) ) \\
\leq & \, 
2^d \sum_{k=1} ^{N(d)} \sum _{i=1} ^{n_k} 
r \Cr{const:lem4.9-3} \mu _{t_0 -2 r^2} ^\varepsilon (\overline{B}_{\tilde r} (x_{k,i}))
=
2^d r \Cr{const:lem4.9-3} 
\sum_{k=1} ^{N(d)}
\mu _{t_0 -2 r^2} ^\varepsilon \left( \cup _{i=1} ^{n_k} \overline{B}_{\tilde r} (x_{k,i}) \right) \\
\leq & \, 
2^d r \Cr{const:lem4.9-3} 
\sum_{k=1} ^{N(d)}
\mu _{t_0 -2 r^2} ^\varepsilon (B_{1+ \Cr{const:lem4.9-1}/2} (0)) \leq 2^d r \Cr{const:lem4.9-3} N(d) D_2,
\end{split}
\end{equation*}
where we used \eqref{eq:3.22}, \eqref{eq:4.37}, and the property that $\mathcal{F} _k$ 
is a family of disjoint balls.
Hence, we obtain \eqref{eq:4.33}.
\end{proof}

\begin{proof}[Proof of Proposition \ref{prop4.7}]
First, we restrict $b\in (0,1)$ to be small enough so that 
\begin{equation}\label{eq:4.38}
1-\sqrt{b} >\frac12, \qquad \log \sqrt{b} \leq -1, \qquad \Cr{const:lem4.7} \vert \log b \vert \geq 1.
\end{equation}
and restrict $\varepsilon $ to be small enough to use lemmas 
\ref{lem4.8} and \ref{lem4.9}.
We choose a positive integer $J$ such that
\begin{equation}\label{eq:4.39}
\varepsilon ^{\frac{1}{2^{J+1}}} \in (b, \sqrt{b} ].
\end{equation}
Then, \eqref{eq:4.24}, \eqref{eq:4.38}, and \eqref{eq:4.39} imply
\begin{equation}\label{eq:4.40}
1\leq \Cr{const:lem4.7} \vert \log b \vert \leq \frac{1}{2^J}\Cr{const:lem4.7} \vert \log \varepsilon \vert .
\end{equation}
Set $t_0 \in (\delta ,T)$ and 
\[
A_j :=
\left\{ 
x \in B_1 (0) \mid
1- \varepsilon ^{\frac{1}{2^{j+1}}} 
\leq 
\vert u ^\varepsilon (x, t_0) \vert 
\leq
1- \varepsilon ^{\frac{1}{2^{j}}} 
\right\}, \qquad
j=1,\dots, J.
\]
For $x_0 \in A_j$, we use Lemma \ref{lem4.8} with 
$\gamma =\frac{1}{2^j}$.
Note that \eqref{eq:4.24} holds with 
$\tilde r = \frac{1}{2^j} \Cr{const:lem4.7} \vert \log \varepsilon \vert$
by \eqref{eq:4.40}.
Then we obtain 
\begin{equation*}
\inf _{B_{\varepsilon \tilde r} (x_0) \times (t_0 -\varepsilon ^2 \tilde r^2, t_0)} \vert u ^\varepsilon \vert
< \frac12
\end{equation*}
and hence
\begin{equation}\label{eq:4.41}
A_j \subset Z_{\varepsilon \tilde r ,t_0}.
\end{equation}
By \eqref{eq:4.24}, we have $\varepsilon \leq \varepsilon \tilde r <\frac12$ 
for sufficiently small $\varepsilon$. Therefore \eqref{eq:4.33} and \eqref{eq:4.41} yield
\begin{equation}\label{eq:4.42}
\mathscr{L}^{d} (A_j) 
\leq
\mathscr{L}^{d} (Z_{\varepsilon \tilde r,t_0}) \leq 
\frac{1}{2^j} \Cr{const:lem4.7} \Cr{const:lem4.8} \varepsilon \vert \log \varepsilon \vert
\end{equation}
for any $j=1,\dots ,J$. On the other hand,
since $\vert  u ^\varepsilon (x,t_0) \vert \geq 1- \varepsilon ^{\frac{1}{2^{j+1}}}$ for 
any $x \in A_j$, we obtain
\begin{equation}\label{eq:4.43}
\frac{W (u ^\varepsilon (x,t_0))}{\varepsilon}
\leq 
\frac{W (1- \varepsilon ^{\frac{1}{2^{j+1}}})}{\varepsilon}
\leq \Cl{const:prop4.7-1} \varepsilon ^{\frac{1}{2^{j}} -1}
\end{equation}
for some constant $\Cr{const:prop4.7-1}$ depending only on $W$.
We define
$
Y:= \{
x \in B_1(0) \mid
1-b \leq \vert u ^\varepsilon (x,t_0) \vert \leq 1-\sqrt{\varepsilon}
\}
$. Note that
\begin{equation}\label{eq:4.44}
Y \subset \cup _{j=1} ^J A_j
\end{equation}
by \eqref{eq:4.39}. Set $\Cl{const:prop4.7-2} = \Cr{const:lem4.7} \Cr{const:lem4.8} \Cr{const:prop4.7-1}$.
Then from \eqref{eq:4.39},
\eqref{eq:4.42}, \eqref{eq:4.43}, and \eqref{eq:4.44} we have
\begin{equation}\label{eq:4.45}
\begin{split}
&\int _Y \frac{W ( u ^\varepsilon (x,t_0))}{\varepsilon} \, dx
\leq
\sum _{j=1} ^J \int _{A_j} \frac{W (u ^\varepsilon (x,t_0))}{\varepsilon} \, dx
\leq \Cr{const:prop4.7-2}
\vert \log \varepsilon \vert \sum _{j=1} ^J 2^{-j} \varepsilon ^{2^{-j}} \\
\leq & \, 
\Cr{const:prop4.7-2}
\vert \log \varepsilon \vert  \int _1 ^{J+1}  2^{-t} \varepsilon ^{2^{-t}} \, dt
=
\Cr{const:prop4.7-2}
\frac{ \varepsilon ^{\frac{1}{2^{J+1}}} - \sqrt{\varepsilon} }{\log 2} 
\leq 
\Cr{const:prop4.7-2} \frac{\sqrt{b}}{\log 2},
\end{split}
\end{equation}
where we used that a function
$p(t) = 2^{-t} \varepsilon ^{2^{-t}}$ satisfies
$p' (t) >0$ for $t \in [1 , J +1]$. Note that 
$2^{-J-1} \log \varepsilon \leq \log \sqrt{b} \leq -1$ by \eqref{eq:4.38} and \eqref{eq:4.39}.
Using the same argument above, we can show that
\begin{equation}\label{eq:4.46}
\begin{split}
& \int 
_{\{ x\in B_1 (0) \mid 1-\sqrt{\varepsilon} \leq \vert u^\varepsilon (x,t_0) \vert \leq 1 -\varepsilon^{\frac{2}{3}} \}}
\frac{W(u ^\varepsilon)}{\varepsilon} \, dx \\
\leq & \, \Cr{const:prop4.7-1} \mathscr{L} ^d
(\{ x \in B_1 (0) \mid 1-\sqrt{\varepsilon} \leq \vert u^\varepsilon (x,t_0) \vert \leq 1 -\varepsilon^{\frac{2}{3}} \}) \\
\leq & \, \frac23 \Cr{const:prop4.7-2} \varepsilon \vert \log \varepsilon\vert,
\end{split}
\end{equation}
where we used Lemma \ref{lem4.8} with $\gamma= \frac23$. 
Since $\vert u ^\varepsilon \vert \leq 1$, we have
\begin{equation}\label{eq:4.47}
\begin{split}
& \int 
_{\{ x \in B_1 (0) \mid 1-\varepsilon ^{\frac23} \leq \vert u^\varepsilon (x,t_0) \vert  \}}
\frac{W( u ^\varepsilon)}{\varepsilon} \, dx \\
\leq & \, \frac{W(1-\varepsilon ^{\frac23})}{\varepsilon} 
\mathscr{L}^d 
(\{ x \in B_1 (0) \mid 1-\varepsilon ^{\frac23} \leq \vert u^\varepsilon (x,t_0) \vert  \})\\
\leq & \, \varepsilon ^{\frac13} \mathscr{L}^d (B_1 (0)). 
\end{split}
\end{equation}
By \eqref{eq:4.45}, \eqref{eq:4.46}, and \eqref{eq:4.47}, Proposition \ref{prop4.7}
holds for sufficiently small $b$ and $\varepsilon$.
\end{proof}

Now we prove the integrality of $\mu_t$.

\begin{theorem}
For a.e. $t>0$, there exist a countably $(d-1)$-rectifiable set $M_t$
and $\mathscr{H}^{d-1}$-measurable function $\theta _t : M_t \to \N$ with
$\theta _t \in L_{loc} ^{1} (\mathscr{H}^{d-1} \lfloor_{M_t})$ such that
$\mu _t = \theta _t \mathscr{H}^{d-1} \lfloor_{M_t}$ holds.
\end{theorem}

\begin{proof}
Set $H ^\varepsilon := \Delta \varphi ^\varepsilon -\frac{W' (\varphi ^\varepsilon)}{\varepsilon ^2}$.
Then for a.e. $t _0 > 0$, we can choose a subsequence $\{ V_{t_0} ^{\varepsilon _{i_j}} \} _{j=1} ^\infty$ 
such that $V_{t_0} ^{\varepsilon _{i_j}} \to V_{t_0}$,
\begin{equation}\label{eq:4.48}
\lim _{j\to \infty} \int_{\Omega} \vert \xi _{\varepsilon_{i_j}} (x,t_0) \vert \, dx =0,
\end{equation}
and
\begin{equation}\label{eq:4.49}
c_H (t_0) := \sup _{j \in \N} 
\int_{\Omega} 
\varepsilon _{i_j} \vert H ^{\varepsilon_{i_j}} \nabla \varphi ^{\varepsilon _{i_j}} \vert (x,t_0) \, dx <\infty
\end{equation}
hold by Theorem \ref{thm4.5} and \eqref{eq:4.21}.
Note that $V_{t_0}$ is a countably $(d-1)$-rectifiable varifold and determined by $\mu _{t_0}$ uniquely
from Theorem \ref{thm4.6}. We fix such $t_0 >0$ and show the claim for $\mu _{t_0}$.
In this proof, even if we take a subsequence $\varepsilon _{i_j}$, 
we always write $\varepsilon _{i_j}$ by $\varepsilon _i$ for simplicity.
Set
\begin{equation*}
A_{i,m} :=
\left\{
x \in \Omega \mid
\int _{B_r (x)} \varepsilon _i \vert H^{\varepsilon _i} \nabla \varphi ^{\varepsilon _i} \vert (x,t_0) \, dx
\leq m \mu _{t_0} ^{\varepsilon _i} (B_r (x)) \ \ \text{for any} \ r \in \left(0, \frac12\right)
\right\}
\end{equation*}
and
\begin{equation*}
A_m := \{
x \in \Omega \mid
\text{there exists} \ x_i \in A_{i,m} \ \text{for any} \ i \in \N  \ \text{such that} \ x_i \to x
\}
\end{equation*}
for any $m \in \N$.
Then the Besicovitch covering theorem implies
\begin{equation}\label{eq:4.50}
\mu _{t_0} ^{\varepsilon_i} (\Omega \setminus A_{i,m}) \leq \frac{c(d) c_H (t_0)}{m},
\end{equation}
where $c(d) >0$ is a constant depending only on $d$. Set
$
A:= \cup _{m=1} ^\infty A_m.
$
Next we prove 
\begin{equation}\label{eq:4.51}
\mu _{t_0} (\Omega \setminus A)=0.
\end{equation}
If \eqref{eq:4.51} is not true, there exists a compact set $K \subset \Omega \setminus A$
with $\mu _{t_0} (K) >\frac12 \mu _{t_0} (\Omega \setminus A) >0$.
Since $A_1 \subset A_2 \subset A_3 \subset \cdots$, 
we have $K \subset \Omega \setminus A_m$
for any $m \in \N$.
For any $x \in K$, there is a neighborhood $B_r (x)$ such that $B_r (x) \cap A_{i,m} =\emptyset$
for sufficiently large $i$, by the definition of $A_m$. This and the compactness of $K$
imply that there exist an open set $O_m$ and $i _0 \in \N$ such that $K \subset O_m$ and
$O_m \cap A_{i,m} =\emptyset$ for any $i \geq i_0$.
Let $\phi_m \in C_c (O_m)$ be a nonnegative test function such that
$0\leq \phi _m \leq 1$ and $\phi_m =1$ on $K$.
We compute
\begin{equation}\label{eq:4.52}
\begin{split}
\mu _{t_0} (K) \leq \int_\Omega \phi _m \, d \mu _{t_0}
=&\, \lim _{i\to \infty} \int_\Omega \phi_m \, d \mu _{t_0} ^{\varepsilon _i} 
=\lim _{i\to \infty} \int_{\Omega \setminus A_{k,m}} \phi_m \, d \mu _{t_0} ^{\varepsilon _i}\\
\leq &\,
\liminf_{i\to \infty} \mu _{t_0} ^{\varepsilon _i} (\Omega \setminus A_{k,m})
\end{split}
\end{equation}
for any $k \geq i_0$. Combining \eqref{eq:4.50} and \eqref{eq:4.52}, we obtain
$\mu _{t_0} (K) =0$. Therefore we have proved \eqref{eq:4.51}.

By the rectifiability of $\mu_{t_0}$ and \eqref{eq:4.51}, for $\mu_{t_0}$ a.e. $x \in \spt \mu_{t_0}$,
it has an approximate tangent space $P$ and $x \in A_m$ for some $m$.
Fix such $x$. We may assume that $x=0$ and $P= \{ x \in \R^d \mid x_d=0 \}$
by a parallel translation and a rotation.
Set $\theta := \lim _{r \downarrow 0} \frac{\mu _{t_0} (B_r (0))}{\omega_{d-1} r^{d-1}}$.
We need only prove $\theta \in \N$. Let $\Phi_{r} (x) = \frac{x}{r}$ for $r>0$ and
$(\Phi _r)_\# V_{t_0}$ be the usual push forward of the varifold. Then for any positive sequence 
$r_i \to 0$, we have $\lim _{i\to \infty} (\Phi _{r_i})_\# V_{t_0} = \theta \vert P \vert$, where 
$\vert P \vert$ is the unit density varifold generated by $P$.
By the assumption $0 \in A_m$, there exists $\{ x_i \}_{i=1} ^\infty$ such that
$x_i \in A_{i,m}$ and $x_i \to 0$ as $i\to \infty$.
Passing to a subsequence if necessary, we may assume that
\begin{equation}\label{eq:4.53}
\lim_{i\to \infty} \frac{x_i}{r_i}=0, \qquad \lim_{i\to \infty} \frac{\varepsilon_i}{r_i}=0, 
\end{equation}
and
\begin{equation}\label{eq:4.54}
\lim _{i\to \infty} (\Phi _{r_i})_\# V_{t_0} ^{\varepsilon _i} = \theta \vert P \vert.
\end{equation}
Set $u ^{\tilde \varepsilon_i } (\tilde x, \tilde t) = \varphi ^{\varepsilon_i} (x,t)$
and $g ^{\tilde \varepsilon_i } (\tilde t) = r_i \lambda ^{\varepsilon_i} (t)$
for $\tilde x= \frac{x}{r_i}$, $\tilde t= \frac{t-t_0}{r^2 _i}$, 
and $\tilde \varepsilon_i = \frac{\varepsilon_i}{r_i}$ (another functions 
$\tilde \xi _{\tilde \varepsilon _i}$ and 
$\tilde H ^{\tilde \varepsilon _i}$ are defined in the same way).
Note that $\tilde x_i := \frac{x_i}{r_i} \to 0$ and $\tilde \varepsilon _i \to 0$ by \eqref{eq:4.53}
and $u ^{\tilde \varepsilon_i }$ is a solution to \eqref{eq:u} with 
$\tilde \varepsilon _i$ instead of $\tilde \varepsilon$.
We compute
\[
\int _{B_3 (0)} \vert \tilde \xi _{\tilde \varepsilon _i} (\tilde x,0) \vert \, d \tilde x
= \frac{1}{r_i ^{d-1}} 
\int _{B_{3r_i} (0)} \vert \xi _{\varepsilon _i} (x,0) \vert \, d x.
\]
Thus, by \eqref{eq:4.48} we may assume that
\begin{equation}\label{eq:4.55}
\lim _{i\to \infty} \int _{B_3 (0)} \vert \tilde \xi _{\tilde \varepsilon _i} (\tilde x,0) \vert \, d \tilde x=0,
\end{equation}
passing to a subsequence if necessary. We compute
\begin{equation}\label{eq:4.56}
\begin{split}
\tilde \varepsilon _i \int _{B_3 (0)}
\vert \tilde H ^{\tilde \varepsilon _i} \nabla_{\tilde x} \tilde u ^{\tilde \varepsilon_i} (\tilde x, 0 )\vert \, d \tilde x
= & \,
\frac{\varepsilon _i}{r_i ^{d-2}}
\int _{B_{3r_i} (0)}
\vert H ^{ \varepsilon _i} \nabla \varphi^{\varepsilon_i} (x, t_0 ) \vert \, d x
\leq \frac{m}{r_i ^{d-2}} \mu _{t_0} ^{\varepsilon _i} (B_{4r_i} (x_i)) \\
\leq & \, 
4^{d-1} m \omega_{d-1} D _2 r_i \to 0 \quad \text{as} \ i\to \infty,
\end{split}
\end{equation}
where we used \eqref{eq:3.22}, \eqref{eq:4.53}, and $x_i \in A_{i,m}$.
Let $\tilde V_{\tilde t} ^{\tilde \varepsilon _i} $ be a varifold defined by \eqref{eq:4.19}
with $u ^{\tilde \varepsilon_i}$ instead of $\varphi ^\varepsilon$.
Then $\tilde V_0 ^{\tilde \varepsilon _i} = (\Phi _{r_i})_\# V_{t_0} ^{\varepsilon _i}$.
Next we show 
\begin{equation}\label{eq:4.57}
\lim_{i\to \infty} \int _{B_3(0)} (1 - (\nu_d ^i) ^2 ) \tilde \varepsilon_i
\vert \nabla _{\tilde x} u ^{\tilde \varepsilon _i} \vert ^2 \, d \tilde x \Big\vert _{\tilde t=0}=0,
\end{equation}
where $\nu^i = (\nu _1 ^i ,\nu _2 ^i ,\dots , \nu _d ^i) = 
\frac{\nabla _{\tilde x} u ^{\tilde \varepsilon _i}}{\vert \nabla _{\tilde x} u ^{\tilde \varepsilon _i} \vert }$.
For $S \in \mathbb{G}(d,d-1)$, set $\psi (S) :=1- \nu_d ^2$, where
$\nu \in \S ^{d-1}$ be one of the unit normal vectors to $S$.
Then $\psi :\mathbb{G}(d,d-1) \to \R$ is well-defined, continuous, and $\psi (P)=0$.
Hence, for any $\phi \in C_c (\R^d)$, $\phi \psi \in C_c (G_{d-1}(\R^d) )$ and
\begin{equation}\label{eq:4.58}
\begin{split}
\lim _{i\to \infty} \tilde V_0 ^{\tilde \varepsilon _i} (\phi \psi) 
= & \, 
\int \phi (\tilde x ) (1 - ( \nu _d ^i) ^2) \, d \|\tilde  V_0 ^{\tilde \varepsilon _i} \| (\tilde x)\\
= & \, \lim _{i\to \infty} (\Phi _{r_i})_\# V_{t_0} ^{\varepsilon _i} (\phi \psi)
= \theta \vert P \vert (\phi \psi) \\
= & \, \theta \int \phi (\tilde x) \psi (P) \, d \mathscr{H} ^{d-1} (\tilde x)=0,
\end{split}
\end{equation}
where we used \eqref{eq:4.54} and $\psi (P)=0$.
Thus \eqref{eq:4.58} proves \eqref{eq:4.57}.
We subsequently fix the subsequence and drop $\tilde \cdot$ and time variable 
(for example, we write $u^{\tilde \varepsilon_i} (\tilde x, 0)$ as $u^{\varepsilon _i} $) 
for simplicity.
We assume that $N \in \N$ is a smallest positive integer grater than $\theta$, namely,
\begin{equation}\label{eq:4.59}
\theta \in [N-1,N).
\end{equation}
Let $s >0$ be an arbitrary number. Then Proposition \ref{prop4.7} and \eqref{eq:3.15}
imply that there exists $b >0$ such that
\begin{equation}\label{eq:4.60}
\int_{\{ x \in B_3 (0) \mid \vert u ^{\varepsilon_i} (x) \vert \geq 1-b \}}
\frac{\varepsilon_i \vert \nabla u ^{\varepsilon_i} \vert^2}{2} +
\frac{W( u ^{\varepsilon_i})}{\varepsilon_i } \, dx \leq s
\end{equation}
for sufficiently large $i$. Note that we may use Proposition \ref{prop4.7} with $t=0$
since $\tilde t = \frac{t -t_0}{r_i ^2}$ in this proof.
For these $s >0$, $b>0$, and $c>0$ given by Lemma \ref{lem:holder}, 
we choose $\varrho$ and $L$ given by propositions \ref{prop7.5} and \ref{prop7.6} in Appendix with $R=2$
(we may restrict $\varrho$ to be small if necessary).
We choose $a=L\varepsilon _i$ as a constant in Proposition \ref{prop7.5}. Set 
\begin{equation}\label{eq:4.61}
\begin{split}
G_i := & \,  B_2 (0) \cap \{ \vert u^{\varepsilon _i} \vert \leq 1-b \} \\
&\, \cap 
\left\{  
x \mid \int _{B_r(x)} \varepsilon _i \vert H^{\varepsilon _i} \nabla u ^{\varepsilon _i} \vert 
+ \vert \xi _{\varepsilon _i} \vert
+(1-\nu _d ^2) \varepsilon _i \vert \nabla u ^{\varepsilon _i} \vert ^2 \, dx \leq
\varrho \mu _0 ^{\varepsilon _i} (B_r (x)) \quad \text{if} \ \varepsilon _i L \leq r \leq 1
\right\}
\end{split}
\end{equation}
for sufficiently large $i$. The Besicovitch covering theorem, \eqref{eq:4.55}, \eqref{eq:4.56}, and \eqref{eq:4.57} yield
\begin{equation}\label{eq:4.62}
\begin{split}
& \mu_0 ^{\varepsilon_i} 
( (B_2 \cap \{ \vert u^{\varepsilon _i} \vert \leq 1-b \}) \setminus G_i ) \\
\leq & \, 
\frac{c(d)}{\varrho} 
\int _{B_3 (0)} \varepsilon _i \vert H^{\varepsilon _i} \nabla u ^{\varepsilon _i} \vert 
+ \vert \xi _{\varepsilon _i}\vert
+(1-\nu _d ^2) \varepsilon _i \vert \nabla u ^{\varepsilon _i} \vert^2 \, dx \to 0 \qquad \text{as} \ i\to \infty.
\end{split}
\end{equation}
Next we show that for sufficiently large $i$ 
\begin{equation}\label{eq:4.63}
\begin{split}
\frac{\mu _0 ^{\varepsilon _i} (B_r (x))}{ \omega _{d-1} r^{d-1} } \geq 1- 2s ,
\qquad \text{for any} \ x \in G_i \ \text{and} \ r \in [L\varepsilon _i , 1]
\end{split}
\end{equation}
Note that all the assumptions in Proposition \ref{prop7.6}
are satisfied by
Lemma \ref{lem:holder}, \eqref{eq:3.15}, and \eqref{eq:4.61}.
Thus we have \eqref{eq:4.63} with $r=L \varepsilon _i$.
By the integration by parts, we have
\begin{equation*}
\begin{split}
\frac{d}{d\tau } \left\{ \frac{1}{\tau ^{d-1}} \int _{B_\tau (x)} e_{\varepsilon _i} \, dy \right\}
+ \frac{1}{\tau ^{d}} \int _{B_\tau (x)} (\xi _{\varepsilon _i} 
+\varepsilon _i H^{\varepsilon _i} (y\cdot \nabla u ^{\varepsilon _i})) \, dy \\
-\frac{\varepsilon _i}{\tau ^{d+1}} \int _{\partial B_\tau (x)} (y\cdot \nabla u ^{\varepsilon _i}) ^2 
d \mathscr{H}^{d-1} (y)
=0. 
\end{split}
\end{equation*}
Thus we can compute
\begin{equation*}
\begin{split}
\frac{1}{ \sigma \tau ^{d-1}} \int _{B_\tau (x)} e_{\varepsilon _i} \, dy \Big \vert_{\tau=L\varepsilon _i} ^r
\geq & \,-
\int _{L\varepsilon _i} ^r \frac{1}{\sigma \tau ^{d}}
\int _{B_\tau (x)} \varepsilon _i H^{\varepsilon _i} (y\cdot \nabla u^{\varepsilon _i}) \, dy d\tau \\
\geq & \, - 
\int _{L\varepsilon _i} ^1 \frac{1}{\sigma \tau ^{d}}
\int _{B_\tau (x)} \varepsilon _i \tau \vert H^{\varepsilon _i} \nabla u^{\varepsilon _i} \vert \, dy d\tau \\
\geq & \, -\frac{\varrho D_2}{\sigma},
\end{split}
\end{equation*}
where we used \eqref{eq:3.22}, \eqref{eq:4.61}, and $\xi _{\varepsilon _i} \leq 0$.
Therefore we obtain \eqref{eq:4.63} for sufficiently large $i$
by restricting $\varrho$ to be small.
Let $\delta >0$ and $\phi \in C_c (B_3 (0))$ be a nonnegative test function such that
$\phi =1$ on $B_2(0) \cap \{ \vert x_d \vert > \delta \}$.
Then there exists $i_0 \geq 1$ such that
\begin{equation}\label{eq:4.64}
\mu _0 ^{\varepsilon _i} (\phi) \leq (1-2s) \omega_{d-1} \frac{\delta ^{d-1}}{2},
 \qquad \text{for any} \ i \geq i_0,
\end{equation}
since
$\mu _0 ^{\varepsilon _i} =\| V_0 ^{\varepsilon _i} \| \to \theta \mathscr{H}^{d-1} \lfloor _{P}$. 
Assume that $x \in G_i \cap \{ \vert x_d \vert > 2 \delta \}$ for $i \geq i_0$. 
Then \eqref{eq:4.63} and \eqref{eq:4.64} imply 
\[
(1-2s) \omega _{d-1} \delta ^{d-1}
\leq \mu _0 ^{\varepsilon _i} (B _{\delta_1} (x)) \leq \mu _0 ^{\varepsilon _i} (\phi)
\leq (1-2s) \omega_{d-1} \frac{\delta ^{d-1}}{2}, \qquad 
\text{for any} \ i \geq i_0.
\]
This is a contradiction. Thus 
\begin{equation}\label{eq:4.65}
\dist (P, G_i) \to 0 \qquad \text{as} \ i \to \infty.
\end{equation}
Set $Y:=P^{-1} (x) \cap G_i \cap \{ x \mid u ^{\varepsilon _i} (x) =l \}$ for $x \in P \cap B_1 (0)$.
Next we show that for sufficiently large $i$
\begin{equation}\label{eq:4.66}
\# Y \leq N-1, \quad
\text{for any} \ x\in P \cap B_1 (0) \ \text{and} \ \vert l \vert \leq 1-b. 
\end{equation}
For a contradiction, assume that $\# Y \geq N$ and choose $y_j \in Y$ for $j=1,2,\dots, N$.
We use Proposition \ref{prop7.5} with $R=1$, $a=L\varepsilon _i$ and
$Y' =\{ y_j \}_{j=1} ^N$ instead of $Y$.
Note that the smallness of $\diam Y'$ is true from \eqref{eq:4.65} and
$\vert y_j -y_k \vert >3 L\varepsilon _i$ for any $1\leq j < k \leq N$ holds by
\eqref{eq:7.19}. Then \eqref{eq:7.18} yields
\begin{equation}\label{eq:4.67}
\sum_{j=1} ^N \frac{1}{(L\varepsilon _i )^{d-1}} 
\mu _0 ^{\varepsilon _i} (B_{L\varepsilon _i} (y_j))
\leq s+ (1+s) \mu _0 ^{\varepsilon _i} 
(\{ z \mid \dist (z,Y') <1 \})
\end{equation}
for sufficiently large $i$. By \eqref{eq:4.65} and
$\mu _0 ^{\varepsilon _i} =\| V_0 ^{\varepsilon _i} \| \to \theta \mathscr{H}^{d-1} \lfloor _{P}$,
\[
\limsup_{i\to \infty} 
\mu _0 ^{\varepsilon _i} 
(\{ z \mid \dist (z,Y') <1 \})
\leq \theta \mathscr{H} ^{d-1} \lfloor_{P} (\overline{B_1 (0)}) = \theta \omega _{d-1}. 
\]
By this, $\# Y' = N$, \eqref{eq:4.63}, and \eqref{eq:4.67} we have
\[
N\omega_{d-1} (1-2s) \leq s+(1+s) \theta \omega_{d-1}.
\]
However, this contradicts \eqref{eq:4.59} by restricting $s$ to be small.
Thus \eqref{eq:4.66} holds for sufficiently large $i$.

Finally, we complete the proof. Set $\hat V_0 ^{\varepsilon _i} 
:= V_0 ^{\varepsilon _i} \lfloor_{\{ \vert x_d \vert \leq 1 \} \times \mathbb{G}(d,d-1)}$.
We regard $P$ as a diagonal matrix 
$(p_{jk})$ with $p_{kk}=1$ for $1\leq k\leq d-1$ and $p_{dd}=0$.
Then the push-forward of $\hat V_0 ^{\varepsilon _i}$ by $P$ is given by
\begin{equation*}
\begin{split}
P_{\#} \hat V_0 ^{\varepsilon _i} (\phi)
= & \, 
\int _{\{ \vert x_d \vert \leq 1 \}}
\phi (Px, \nabla Px \circ (I -\nu^i \otimes \nu ^i)) \vert \Lambda _{d-1} \nabla Px \circ (I -\nu^i \otimes \nu ^i) \vert 
\, d \mu _0 ^{\varepsilon _i} \\
= & \, 
\int _{\{ \vert x_d \vert \leq 1 \}}
\phi (Px, P \circ (I -\nu^i \otimes \nu ^i)) \vert \nu ^i _d \vert 
\, d \mu _0 ^{\varepsilon _i}
\end{split}
\end{equation*}
for any $\phi \in C_c (P \cap B_2 (0) \times \mathbb{G}(d,d-1))$.
Here $\vert \Lambda _{d-1} \nabla Px \circ (I -\nu^i \otimes \nu ^i) \vert$ is the Jacobian
and $\nu ^ i _d = \frac{\partial _{x_d} u ^{\varepsilon _i}}{\vert \nabla u ^{\varepsilon _i} \vert}$.
Due to \eqref{eq:4.54}, $P_\# \hat V_0 ^{\varepsilon _i} \to P_\# (\theta \mathscr{H}^{d-1} \lfloor_{P}) 
=\theta \mathscr{H}^{d-1} \lfloor_{P}$ as $i\to \infty$.
By \eqref{eq:4.55},
\begin{equation}\label{eq:4.68}
\lim_{i\to \infty} \int _{B_3 (0)}
\left\vert
\frac{\varepsilon_i \vert\nabla u ^{\varepsilon_i} \vert^2}{2} +
\frac{W( u ^{\varepsilon_i})}{\varepsilon_i }
- \sqrt{2W(u^{\varepsilon _i})} \vert \nabla u^{\varepsilon _i}\vert 
\right\vert
\, dx =0
\end{equation}
holds (see \eqref{eq:5.1} below).
We compute
\begin{equation}\label{eq:4.69}
\begin{split}
\omega _{d-1} \theta 
= & \,
\mathscr{H}^{d-1} \lfloor _{P} (B_1 (0)) 
\leq \liminf _{i\to \infty} \| P_\# \hat V_0 ^{\varepsilon _i} \| (B_1 (0)) 
= \liminf _{i\to \infty} \int _{B_1 (0)} \vert \nu _d ^i \vert \, d \mu _0 ^{\varepsilon _i} \\
\leq & \,
\liminf _{i\to \infty} \int _{B_1 (0) \cap G_i } 
\vert \nu _d ^i \vert \, d \mu _0 ^{\varepsilon _i} + 2s \\
\leq & \,
\liminf _{i\to \infty} \frac{1}{\sigma} \int _{B_1 (0) \cap G_i } 
\vert \nu _d ^i \vert \sqrt{2W(u^{\varepsilon _i})} \vert \nabla u^{\varepsilon _i} \vert  \, dx + 2s ,
\end{split}
\end{equation}
where we used \eqref{eq:4.60}, \eqref{eq:4.62}, and \eqref{eq:4.68}.
By the co-area formula and the area formula, we have
\begin{equation}\label{eq:4.70}
\begin{split}
& \int _{B_1 (0) \cap G_i } 
\vert \nu _d ^i \vert \sqrt{2W(u^{\varepsilon _i})} \vert \nabla u^{\varepsilon _i} \vert  \, dx \\
= & \, 
\int _{-1+b} ^{1-b} \sqrt{2W(\tau)} \int _{ B_1 (0) \cap G_i \cap \{ u ^{\varepsilon_i } =\tau \} }
\vert \Lambda _{d-1} \nabla Px \circ (I -\nu^i \otimes \nu ^i) \vert \, d\mathscr{H}^{d-1} (x) d\tau \\
= & \,
\int _{-1+b} ^{1-b} \sqrt{2W(\tau)} 
\int _{B_1 (0) \cap \{ x_d=0 \} } \mathscr{H}^{0} 
( \{ u ^{\varepsilon _i} =\tau \} \cap G_i \cap P^{-1} (x) )
 \, d\mathscr{H}^{d-1} (x) d\tau \\
\leq & \,
\int _{-1+b} ^{1-b} \sqrt{2W(\tau)} 
\int _{B_1 (0) \cap \{ x_d=0 \} } 
(N-1) 
 \, d\mathscr{H}^{d-1} (x) d\tau \\
 \leq & \, 
\sigma (N-1) \omega _{d-1} ,
\end{split}
\end{equation}
where we used \eqref{eq:4.66} and $\sigma = \int _{-1} ^{1} \sqrt{2W (\tau)} \, d\tau$.
Hence 
$
\theta \leq N-1
$
due to \eqref{eq:4.69} and \eqref{eq:4.70} and the arbitrariness of $s$. By this and \eqref{eq:4.59},
$\theta =N-1$.
\end{proof}


\section{Proofs of main theorems}

In this section we prove Theorem \ref{mainthm1} and Theorem \ref{mainthm2}
on the existence of the weak solution in the sense of $L^2$-flow and distributional $BV$-solution.
\begin{proof}[Proof of Theorem \ref{mainthm1}]
Let $\{ \varphi ^{\varepsilon _i} _0 \} _{i=1} ^\infty$ be a family of functions such that
all the claims of Proposition \ref{prop2.2} are satisfied.
Then one can check that all the assumptions in Section 3 and 4 are fulfilled.
Therefore (a) holds by propositions \ref{prop2.2}, \ref{prop3.11}, and \ref{prop3.12}.
By taking a subsequence $\varepsilon _i \to 0$, we obtain (b) 
(the proof is standard and is exactly the same as that in \cite{takasao2017}, so we omit it). 
By Lemma \ref{lem3.2} and the weak compactness of $L^2 (0,T)$, we may
take a subsequence $\varepsilon_i \to 0$
such that (c) holds (for the weak convergence for all $T>0$, 
we only need to use the diagonal argument).

Next we show (d). We compute
\begin{equation}\label{eq:5.1}
\begin{split}
\frac{\varepsilon \vert \nabla \varphi ^\varepsilon \vert^2}{2} 
+ \frac{W(\varphi ^\varepsilon )}{\varepsilon}
- \sqrt{2W(\varphi ^\varepsilon )} \vert\nabla \varphi ^\varepsilon \vert
=\left(\frac{\sqrt{\varepsilon} \vert\nabla \varphi ^\varepsilon \vert}{\sqrt{2}} 
- \frac{\sqrt{W(\varphi ^\varepsilon )}}{\sqrt{\varepsilon}}
 \right) ^2 \leq \vert \xi _\varepsilon\vert .
\end{split}
\end{equation}
Set $d \hat{\mu} ^\varepsilon := \frac{1}{\sigma}\sqrt{2W(\varphi ^\varepsilon )} 
\vert \nabla \varphi ^\varepsilon \vert \, dxdt$.
By \eqref{eq:5.1}, Proposition \ref{prop3.11}, and Theorem \ref{thm4.5}, we have
\begin{equation}\label{eq:5.2}
\hat{\mu} ^\varepsilon \to \mu \qquad \text{as Radon measures}, 
\end{equation}
where $d\mu := d\mu _t dt$. By \eqref{eq:3.1}, \eqref{eq:3.5},
 \eqref{eq:5.1}, and \eqref{eq:5.2}, we obtain
\[
\sup _{i \in \N} \int _{\Omega \times (0,T)} \vert \lambda ^{\varepsilon _i} \vert^2 \, d \hat{\mu} ^{\varepsilon_i} 
\leq
\sup _{i \in \N} \int _{\Omega \times (0,T)} \vert \lambda ^{\varepsilon _i} \vert^2 \, d {\mu} ^{\varepsilon_i} _t dt
\leq D_1 \Cr{const:3.1} (1+T).
\]
Then there exist $\vec{f} \in (L_{loc} ^2 (\mu))^d$ and the subsequence $\varepsilon_i \to 0$ such that
\begin{equation*}
\begin{split}
\frac{1}{\sigma}
\int _{\Omega \times (0,T) } - \lambda^{\varepsilon_i}
\sqrt{2W(\varphi ^{\varepsilon_i} )} \nabla \varphi ^{\varepsilon_i} \cdot \vec{\phi} \, dxdt
= & \, 
\int _{\Omega \times (0,T) \cap \{ \vert \nabla \varphi ^{\varepsilon _i}\vert \not = 0 \} } 
- \lambda^{\varepsilon_i}
\frac{\nabla \varphi ^{\varepsilon_i}}{\vert \nabla \varphi ^{\varepsilon_i}\vert} \cdot \vec{\phi} 
\, d\hat{\mu} ^{\varepsilon _i}\\
\to & \, \int _{\Omega \times (0,T)} \vec{f} \cdot \vec{\phi} \, d \mu _t dt
\end{split}
\end{equation*} 
for any $\vec{\phi} \in C_c (\Omega\times [0,T) ; \R^d)$ (see \cite[Theorem 4.4.2]{hutchinson}).
Moreover, if $\vec {\phi} $ is smooth,
we have
\begin{equation*}
\begin{split}
& \int _{\Omega\times (0,T)} \vec{f} \cdot \vec{\phi} \, d \mu
=
\lim_{i\to \infty}
\frac{1}{\sigma}
\int _{\Omega \times (0,T) } - \lambda^{\varepsilon_i}
\sqrt{2W(\varphi ^{\varepsilon_i} )} \nabla \varphi ^{\varepsilon_i} \cdot \vec{\phi} \, dxdt \\
= & \,
\lim_{i\to \infty}
\frac{1}{\sigma}
\int _{\Omega \times (0,T) } \lambda^{\varepsilon_i}
k( \varphi ^{\varepsilon_i}) \div \vec{\phi} \, dxdt
= \int _0 ^T \lambda \int _{\Omega} \psi \div \vec{\phi} \, dxdt \\
= & \, - \int _0 ^T \lambda \int _{\Omega} \vec{\phi} \cdot \nu \, d \| \nabla \psi (\cdot, t)\| dt,
\end{split}
\end{equation*}
where we used (c), $k'(s) =\sqrt{2W(s)}$, 
$\lim _{i\to \infty} k(\varphi ^{\varepsilon_i})
= \sigma (\psi -\frac12) $ for a.e. $(x,t)$, and the dominated convergence theorem.
Hence we have \eqref{claim-e2}.

Now we prove (e). 
By replacing $d\hat \mu ^{\varepsilon}$ with 
$d \tilde \mu^\varepsilon := \frac{\varepsilon }{\sigma} \vert \nabla \varphi ^{\varepsilon}\vert^2 \, dxdt$,
the convergence \eqref{claim-e} is obtained in the same way as (d).
In addition, for any $\vec{\phi} \in C_c (\Omega\times [0,T) ; \R^d)$, we compute
\begin{equation*}
\begin{split}
&\int_{0} ^T \int _{\Omega} \vec{v} \cdot \vec{\phi} \, d \mu_t dt
= 
\lim_{i \to \infty}
\int_0 ^T \int _{\Omega }
\vec{v}^{ \, \varepsilon _i} \cdot \vec{\phi} \, d\mu _t ^{\varepsilon _i} dt
=
\lim_{i \to \infty}
\int_0 ^T \int _{\Omega }
\vec{v}^{ \, \varepsilon _i} \cdot \vec{\phi} \, d \tilde  \mu ^{\varepsilon _i}\\
= & \,
\lim_{i \to \infty}
\int_0 ^T \int _{\Omega }
-\varepsilon _i 
\left(
\Delta \varphi ^{\varepsilon_i} -\frac{W' (\varphi ^{\varepsilon_i})}{\varepsilon_i ^2} 
+ \lambda ^{\varepsilon _i} \frac{\sqrt{2W(\varphi ^{\varepsilon_i})}}{\varepsilon _i} 
\right) \nabla \varphi ^{\varepsilon_i}  \cdot \vec{\phi} 
 \, dxdt
 \\
= & \,
\int _0 ^T \int _{\Omega} (\vec{h} +\vec{f}) \cdot \vec{\phi} \, d\mu _t dt ,
\end{split}
\end{equation*}
where we used Theorem \ref{thm4.6} and (d). Thus $\vec{v}=\vec{h} +\vec{f}$. 
One can check that $\{ \mu _t \}_{t \in [0,\infty)}$ is an $L^2$-flow
with the generalized velocity vector $\vec{v}$
(see \cite[Proposition 4.3]{takasao2017} for the inequality \eqref{ineq-L2} and \cite[Lemma 6.3]{MR2383536}
for the perpendicularity).
\end{proof}


To prove Theorem \ref{mainthm2}, 
we use the following proposition and lemmas.
In the original proof of the following proposition, $2\leq d\leq 3$ is assumed to
use the results of \cite[Proposition 4.9, Theorem 5.1]{roger-schatzle}.
However, we already know that $\vert \xi \vert=0$ and $\mu_t$ is integral for a.e. $t$, 
so we can show the claim in the same way.
\begin{proposition}[See Proposition 4.5 of \cite{MR2383536}]\label{prop5.1}
Let $\psi$, $\vec{v}$, and $\vec{\nu}$ are given by Theorem \ref{mainthm1}.
Then $\int  \vert \vec{v} \cdot \vec{\nu} \vert d \| \nabla \psi (\cdot, t)\| dt <\infty$ and
\begin{equation}\label{eq:5.3}
\int _0 ^T \int _{\Omega} \phi \vec{v} \cdot \vec{\nu} \, d \| \nabla \psi (\cdot, t)\| dt
= \int _0 ^T \int _{\Omega} \phi _t \psi  \, dxdt
\end{equation}
for any $\phi \in C_c ^1 (\Omega\times (0,T))$ and for any $T>0$.
\end{proposition}

\begin{proof}
Set $\nabla_{x,t} = (\nabla, \partial _t)$ in the sense of BV.
One can check that $\| \nabla _{x,t} \psi \|\ll \mu$, 
$ \mu \lfloor_{\partial ^\ast \{ \psi=1 \}}$ is rectifiable, 
$\int  \vert \vec{v} \cdot \vec{\nu} \vert \, d \| \nabla \psi (\cdot, t)\| dt <\infty$,
and 
the approximate tangent space coincides with that of $\| \nabla_{x,t} \psi\|$ for $\mu$-a.e.
and $\|\nabla_{x,t} \psi\|$-a.e. (see \cite[Proposition 8.1--8.3]{MR2383536} and 
\cite[Proposition 2.85]{MR1857292}). 
By this and Proposition \ref{prop2.3},
we have
\begin{equation*}
\begin{split}
0 = 
\int _0 ^T \int _{\Omega}  \phi (\vec{v},1) \cdot \vec \nu_{x,t} \, d\|\nabla _{x,t} \psi \| 
=
\int _0 ^T \int _{\Omega} \phi \vec{v} \cdot \vec{\nu} \, d \| \nabla \psi (\cdot, t)\| dt
+
\int _0 ^T \int _{\Omega} \phi \psi _t \, dxdt
\end{split}
\end{equation*}
for any $\phi \in C_c ^1 (\Omega\times (0,T))$,
where $\vec \nu_{x,t}$ is the inner unit normal vector of $\{ (x,t) \mid \psi (x,t)=1 \}$.
Therefore we obtain \eqref{eq:5.3}.
\end{proof}

\begin{lemma}\label{lem5.2}
Let $\gamma$ and $\delta$ be positive constants with $\delta < \gamma$.
Under the same assumptions of Theorem \ref{mainthm2}, there exist $T_2 \in( 0,1)$
 and $\epsilon_5 \in (0,1)$
depending only on $\gamma$, $\delta$, and $\Cr{const:3.1}(\omega,d ,D' _1)$
with the following property.
Let $g: \R \to [0,\infty)$ be a smooth even function such that $g(0)=0$, 
$0 \leq g'' (s) \leq 2$ for any $s \in \R$, and $g(s)=\vert s \vert -\frac12$ if $\vert s \vert \geq 1$.
Set $g^\delta (s) := \delta g(s/\delta)$ and 
define $\tilde r^{\varepsilon,\delta} \in C^\infty (\R^d \times [0,\infty))$ by
\[
\tilde r^{\varepsilon,\delta} (x,t) := g^\delta (x_1) + \int _0 ^t \lambda ^\varepsilon (\tau) \, d\tau 
+ 2\delta ^{-1} t - \gamma, 
\]
where $\lambda ^\varepsilon$ is given by 
\eqref{lambdadef}. Set $\tilde \phi^{\varepsilon,\delta} := q^\varepsilon (\tilde r^{\varepsilon,\delta} )$
and assume that $\tilde \phi ^\varepsilon (x,0) \geq \varphi ^\varepsilon _0 (x) $
for any $x \in \R^d$. Then 
\begin{equation}\label{eq:5.6}
\tilde \phi ^{\varepsilon,\delta}  \geq \varphi ^\varepsilon
\qquad \text{in} \ \R^d \times [0,T_2) 
\end{equation}
for any $\varepsilon \in (0,\epsilon_5)$.
\end{lemma}

\begin{proof}
We denote $\tilde r =\tilde r^{\varepsilon,\delta}$ for simplicity.
By \eqref{ac-r} and the comparison principle,
we need only prove 
\begin{equation}
\tilde  r  _t
\geq 
\Delta \tilde r
- \dfrac{2 q ^\varepsilon (\tilde r) }{\varepsilon } ( \vert \nabla \tilde r \vert ^2 -1) 
+ \lambda ^\varepsilon  
\qquad \text{in} \ \R^d \times (0,T_2)
\label{ac-r2}
\end{equation}
for sufficiently small $T_2>0$ and $\varepsilon>0$, since
$\tilde \phi ^{\varepsilon,\delta} \geq \varphi ^\varepsilon$ if and only if
$\tilde r  \geq r ^\varepsilon$.
In the case of $\vert x_1 \vert \geq \delta$, \eqref{ac-r2} holds by 
$\tilde r_t = 2\delta^{-1} + \lambda ^\varepsilon $,
$\vert \nabla \tilde r \vert =1$, and $\Delta \tilde r =0$.
Next we consider the case of $\vert x_1 \vert \leq \delta$.
Set $O_\delta :=\{ x \in \R^d \mid \vert x_1 \vert \leq \delta \}$.
By $\tilde r (x ,0) \leq -\gamma +\frac{\delta}{2} $ on $O_\delta$ and
\[
\vert \tilde r (x,t) -\tilde r (x,0) \vert
\leq \int _0 ^t \vert \lambda ^\varepsilon (\tau) \vert \, d\tau + 2\delta^{-1} t 
\leq \sqrt{\Cr{const:3.1} (1+t)} \sqrt{t} + 2\delta^{-1} t,
\]
there exists $T_2 >0$ such that 
\begin{equation}\label{eq:5.7}
\tilde r (x,t) \leq -\frac{\gamma}{4}<0
\qquad \text{for any} \ (x,t) \in O_\delta \times [0,T_2).
\end{equation}
By \eqref{eq:5.7} and $\vert \nabla \tilde r \vert \leq 1$, 
$\dfrac{2 q ^\varepsilon (\tilde r) }{\varepsilon } ( \vert \nabla \tilde r \vert ^2 -1) \geq 0$. 
By using this, for any $(x,t) \in O_\delta \times [0,T_2)$,
\begin{equation}\label{ac-r3}
\tilde  r  _t
-
\Delta \tilde r
+ \dfrac{2 q ^\varepsilon (\tilde r) }{\varepsilon } ( \vert \nabla \tilde r \vert ^2 -1) 
- \lambda ^\varepsilon  
\geq 2\delta^{-1} - \Delta \tilde r \geq 0,
\end{equation}
where we used
$\Delta \tilde r \leq \delta^{-1} g'' (x_1 /\delta) \leq 2 \delta^{-1}$.
Therefore we obtain \eqref{ac-r2}.
\end{proof}

\begin{lemma}\label{lem5.3}
Let $r \in (0,\frac{1}{4})$. Then there exists $T_3 >0$ depending only on $d$ and $r$
with the following property.
Let $U_0 \subset \subset (\frac14,\frac34)^d$ satisfies 
$\mathscr{L}^d (U_0) =\mathscr{L}^d (B_{r} (0))$ and has a 
$C^1$ boundary $M_0$ with \eqref{iso} for $\delta_1 >0$.
In addition, we assume $\mathscr{H}^{d-1} (M_0) \leq 2 \mathscr{H}^{d-1} (\partial B_{r} (0))$.
Then we have 
\begin{equation}\label{eq:5.10}
0\leq \mu _0 (\Omega) - \mu _t ( \Omega ) \leq \delta_1 \qquad \text{for any} \ t \in [0,T_3),
\end{equation}
where $\mu_t$ is a weak solution to \eqref{vpmcf} with initial data $M_0$.
\end{lemma}

\begin{proof}
First we claim that there exists $T_3 >0$ depending only on $d$ and $r$ such that
\begin{equation}\label{eq:5.11}
U_t = \{ x \in (0,1)^d \mid \psi (x,t)=1 \} \subset \left( \frac{1}{10}, \frac{9}{10}\right)^d
\end{equation}
for any $ t \in[0,T_3)$, where 
$\psi = \lim _{i\to \infty} \psi ^{\varepsilon _i} = \lim _{i\to \infty}  \frac12 ( \varphi ^{\varepsilon _i} +1)$.
Let $\tilde \phi ^{\varepsilon ,\delta}$ be a function given by Lemma \ref{lem5.2}
with $\gamma= \frac{1}{10}$ and $\delta = \frac{1}{20}$. 
By \eqref{eq:5.6} and \eqref{eq:5.7}, one can check that
there exists $T_3=T_3 (\Cr{const:3.1}(\omega,d ,D' _1)) >0$
such that
$\lim _{ i\to \infty }\varphi^{\varepsilon _i} (x,t) = -1 $
on $\{ x \in \R^d \mid \vert x_1 \vert \leq \frac{1}{10} \}$
for any $t \in [0,T_3)$.
Note that $\omega$ and $D' _1$ depend only on $r$ by  
$\mathscr{L}^d (U_0) =\mathscr{L}^d (B_{r} (0))$ and
$\mathscr{H}^{d-1} (M_0) \leq 2 \mathscr{H}^{d-1} (\partial B_{r} (0))$.
Hence $T_3$ depends only on $d$ and $r$.
Therefore $U_t \cap \{ x \in \R^d \mid \vert x_1 \vert \leq \frac{1}{10} \} =\emptyset$ for any $t \in [0,T_3)$.
Similarly we have \eqref{eq:5.11}.
Thus $\partial ^\ast (U_t \cap (0,1)^d) = \partial ^\ast U_t$ for any $t \in [0,T_3)$.
Hence, by using the isoperimetric inequality for Caccioppoli sets (see \cite{MR2147710,MR1242977}), 
and (b3) and (b4) of Theorem \ref{mainthm1}, we have
\begin{equation}\label{iso2}
d \omega _d ^{\frac{1}{d}} (\mathscr{L}^d (U_0))^{\frac{d-1}{d}}=
d \omega _d ^{\frac{1}{d}} (\mathscr{L}^d (U_t))^{\frac{d-1}{d}} \leq \mathscr{H}^{d-1} (\partial ^\ast U_t) 
\leq \mu _t (\Omega) \qquad \text{for any} \ t \in [0,T_3).
\end{equation}
By \eqref{iso} and \eqref{iso2}, we obtain \eqref{eq:5.10}.
\end{proof}

\bigskip

Finally we prove Theorem \ref{mainthm2}.
\begin{proof}[Proof of Theorem \ref{mainthm2}]
First we show (a). 
From \eqref{eq:3.5.additional}, 
\begin{equation*}
\int_ {0} ^{T} \vert \lambda ^\varepsilon (t) \vert^2 \, dt 
\leq \Cr{const:3.1-2} (\mu _0 ^\varepsilon (\Omega) -\mu _{T} ^\varepsilon (\Omega) + T )
\end{equation*}
for any $T>0$ and for any $\varepsilon \in (0,\epsilon_1)$. By this and \eqref{eq:5.10},
we can choose $\delta_1 = \delta _1 (\Cr{const:3.1-2}(\omega,d ,D' _1)) >0$ so that 
\begin{equation}\label{eq:5.12}
\limsup _{i\to \infty} e^{\frac12 \int_0 ^{T_4} \vert \lambda ^{\varepsilon_i} \vert ^2 \, dt} 
\leq \limsup _{i \to \infty} e^{ \frac12 \Cr{const:3.1-2} \delta_1 } e^{\frac12 \Cr{const:3.1-2} T_4}
\leq \frac{10}{9},
\end{equation}
where $T_4 = T_4 (d,r) = \min \{ T_3,  \frac{2}{ \Cr{const:3.1-2} } \log \frac{100}{99} \} >0$
and $\delta_1$ also depends only on $d$ and $r$ since
$\mathscr{L}^d (U_0) =\mathscr{L}^d (B_{r} (0))$ and
$\mathscr{H}^{d-1} (M_0) \leq 2 \mathscr{H}^{d-1} (\partial B_{r} (0))$.
Then \eqref{eq:3.20} and \eqref{eq:5.12} imply
\begin{equation}\label{eq:5.13}
\int_{\mathbb{R}^d} \rho _{(y,s)} (x,t) \, d\mu _t  (x) 
\leq  
\frac{10}{9} \int_{\mathbb{R}^d} \rho _{(y,s)} (x,0) \, d\mu _0  (x) 
\end{equation}
for any $y \in \R^d$, $t\in [0,T_4)$, and $s>0$ with $0\leq t <s$.
Recall that $\rho _{(y,s)} (x,0)$ converges to $(d-1)$-dimensional 
delta function at $y$ as $s\downarrow 0$.
Therefore, since $M_0$ is $C^1$, we may assume that there exists $s_0>0$ depending only on $M_0$
such that
\begin{equation}\label{eq:5.4}
\int _{\R^d} \rho _{(y,s)} (x,0) \, d \mu _0
\leq \frac32
\end{equation}
for any $y\in \R^d$ and $s \in (0,s_0)$. Set $T_1= T_1 (d,r,M_0) := \min \{ T_4 , s_0 \}$.
Let $t_0 \in (0,T_1)$ be a number such that $\mu _{t_0}$ is integral.
Then there exist a countably $(d-1)$-rectifiable set $M_{t_0}$ and 
$\theta_{t_0} \in L ^1 _{loc} (\mathscr{H}^{d-1} \lfloor_{M_{t_0}})$ such that 
$\mu _{t_0} = \theta _{t_0} \mathscr{H}^{d-1} \lfloor_{M_{t_0}}$.
Assume that there exist $x_0 \in M_{t_0}$ and $N\geq 2$ such that
$M_{t_0}$ has an approximate tangent space at $x_0$ and
$\displaystyle \lim _{r\to 0} \frac{ \mu _{t_0} (B_r (x_0)) }{\omega _{d-1} r^{d-1}} =\theta _{t_0} (x_0) =N$.
Set $r=\sqrt{2(s-t_0)}$ for $t_0<s <T_1$.
Using the same calculation as \eqref{eq:7.15}, for any $\delta \in (0,1)$, we obtain
\begin{equation*}
\begin{split}
& \int _{\R^d} \rho _{(x_0,s)} (x,t) \, d\mu_{t_0}
\geq 
\frac{1}{(\sqrt{2\pi} r)^{d-1}} \int _\delta ^1 \mu_{t_0} (B_{\sqrt{2r^2 \log \frac{1}{k} }} (x_0)) \, dk \\
\to & \,
\frac{N \omega_{d-1}}{\pi ^{\frac{d-1}{2}}} 
\int _\delta ^1 \left(  \log \frac{1}{k} \right)^{\frac{d-1}{2}} \, dk
\qquad \text{as} \ r\to 0 \ (s\downarrow t_0).
\end{split}
\end{equation*}
By this and 
$
\int _0 ^1 \left( \log \frac{1}{k} \right)^{\frac{d-1}{2}} \, dk
=\Gamma (\frac{d-1}{2} +1)
=\pi ^{\frac{d-1}{2}} /\omega_{d-1}
$, we have
\begin{equation*}\label{eq:5.8}
\begin{split}
\lim_{s\downarrow t_0} \int _{\R^d} \rho _{(x_0,s)} (x,t) \, d\mu_{t_0} =N.
\end{split}
\end{equation*}
Then we have a contradiction by \eqref{eq:5.13} and \eqref{eq:5.4}.
Hence we obtain (a).

The claim (b1) and (b2) are clear and (b3) is also obvious by Remark \ref{rem2.6} and 
$\mu_t= \|\nabla \psi (\cdot,t)\|$ for a.e. $t \in (0,T_1)$. By \eqref{eq:5.1}, we have (b4).

Next  we prove (b5).
By \eqref{eq:5.3}, for any $\phi \in C_c ^1 ((0,T))$, we compute
\begin{equation*}
\begin{split}
\int _0 ^{T_1} \phi
\int _{\Omega} \vec v \cdot \vec \nu \, d \| \nabla \psi (\cdot ,t) \| dt
=\int _0 ^{T_1} \phi _t \int _{\Omega} \psi \, dx dt
=0,
\end{split}
\end{equation*}
where we used (b3) of Theorem \ref{mainthm1}.
Thus $\int _{\Omega} \vec v \cdot \vec \nu \, d \| \nabla \psi (\cdot ,t) \|=0$ for a.e. $t \in (0,T_1)$.

Now we prove (b6). Set $d\nu:=d\mathscr{H}^{d-1} \lfloor_{\partial ^\ast U_t} dt$.
Since the space $C_c (\Omega)$ is dense in $L^2 (\nu)$, for any $\eta \in C_c ((0,T_1))$ we have
\begin{equation*}
\begin{split}
0 = & \, \int _0 ^{T_1} \int _{\partial ^\ast U_t} \{ \vec v -\vec h +\lambda \vec \nu \}
\cdot \vec \nu \eta \, d \mathscr{H}^{d-1} dt \\
= & \, \int _0 ^{T_1}
\left( - \int _{\partial ^\ast U_t} \vec h\cdot \vec \nu \, d \mathscr{H}^{d-1} 
+ \lambda  \mathscr{H}^{d-1} (\partial ^\ast U_t)  \right) \eta \, dt,
\end{split}
\end{equation*}
where we used (b3) and (b5). Hence we obtain (b6).
\end{proof}

\section{Appendix}

\begin{proposition}[See Lemma 3.24 in \cite{MR1857292}]\label{prop7.2}
Let $U \subset \R^d$ be an open set. 
Assume that $u \in BV(U)$ and $K \subset U$ is a compact set. Then
\[
\int _K \vert u\ast \eta_\delta -u \vert \, dx \leq \delta \| \nabla u \| (U)
\]
for all $\delta \in (0, \dist (K,\partial U))$, where $\eta _\delta$ is the standard mollifier defined in Section 3.
\end{proposition}

%
%

\bigskip

As in Lemma \ref{lem3.2}, we can obtain the following estimate for the classical solution to the volume preserving mean curvature flow.
\begin{proposition}\label{prop7.3}
Let $\Omega = \T^d$ and $U_t \subset \Omega$ be a bounded open set with smooth boundary $M_t$
for any $t\ \in [0,T)$ and $0 < \mathscr{L}^d (U_0) < \mathscr{L}^d(\Omega)$. Assume that $\{M_t \} _{t\in [0,T)}$ is the volume preserving mean curvature
flow. Then there exists $C_\lambda >0$ depending only on $d$, 
$\mathscr{L}^d(U_0)$, and $\mathscr{H}^{d-1}(M_0)$ such that
\begin{equation}
\int _0 ^s \vert \lambda (t) \vert ^2 \, dt \leq C_\lambda (1+s),\qquad s\in (0,T) ,
\label{eq:7.5}
\end{equation}
where $\lambda (t) =\frac{1}{\mathscr{H}^{d-1} (M_t)} \int _{M_t} \vec{h} \cdot \vec{\nu} \, d \mathscr{H}^{d-1}$.
\end{proposition}
\begin{proof}
Let $\vec{\zeta} :\Omega \times [0,\infty) \to \R ^d$ be a smooth periodic function.
By \eqref{vpmcf}, \eqref{vpproperty}, the divergence theorem, and the property of the mean curvature, we have
\begin{equation}
\frac{d}{dt} \mathscr{H}^{d-1} (M_t) 
=
- \int _{M_t} \vec{h} \cdot \vec{v} \, d \mathscr{H}^{d-1} 
= 
- \int _{M_t} \vert \vec{v} \vert ^2 \, d \mathscr{H}^{d-1} \leq 0
\label{eq:7.6}
\end{equation}
and
\begin{equation}
\begin{split}
\int _{M_t} \vec{v} \cdot \vec{\zeta} \, d \mathscr{H}^{d-1} 
= & \, 
\int _{M_t}\vec{h} \cdot \vec{\zeta} \, d \mathscr{H}^{d-1} 
-
\lambda \int _{M_t} \vec{\nu} \cdot \vec{\zeta} \, d \mathscr{H}^{d-1} \\
= & \,
- \int _{M_t} \div \!_{M_t} \vec{\zeta} \, d \mathscr{H}^{d-1} 
+
\lambda \int _{U_t} \div \vec{\zeta} \, dx.
\end{split}
\label{eq:7.7}
\end{equation}
By \eqref{eq:7.6} and \eqref{eq:7.7}, we obtain
\begin{equation}
\begin{split}
\vert \lambda  \vert \left\vert \int _{U_t} \div \vec{\zeta} \, dx \right\vert
\leq & \, 
\| \vec{\zeta} (\cdot,t) \|_{C^1} 
\left( \mathscr{H}^{d-1} (M_t) + \int _{M_t} \vert \vec{v} \vert \, d\mathscr{H}^{d-1} \right) \\
\leq & \,
\| \vec{\zeta} (\cdot,t) \|_{C^1} 
\left( \mathscr{H}^{d-1} (M_0) + \int _{M_t} \vert \vec{v} \vert \, d\mathscr{H}^{d-1} \right).
\end{split}
\label{eq:7.8}
\end{equation}
Let $\alpha, \delta \in (0,1)$ and $u=u(x,t)$ be a periodic solution to
\begin{equation*}
\left\{ 
\begin{array}{ll}
-\Delta u &=\chi _{U_t} \ast \eta_\delta - \dashint _\Omega (\chi _{U_t} \ast \eta_\delta) \qquad \text{in} \ \Omega, \\
\int _\Omega u \, dx &=0.
\end{array} \right.
\end{equation*}
Then the standard PDE arguments imply the existence and uniqueness of the solution $u$ and
\[
\| u(\cdot, t) \|_{C^{2,\alpha} (\Omega)} \leq C_\delta , \qquad t \in [ 0, T),
\] 
where $C_\delta >0$ depends only on $d$ and $\delta$. 
Set $\vec{\zeta} (x,t) = \nabla u(x,t)$. We compute that
\begin{equation*}
\begin{split}
&-\int _{U_t} \div \vec{\zeta} \, dx = \int _{U_t} (-\Delta u) \, dx \\
=& \,
\int _{U _t}  \left( \chi _{U_t} \ast \eta_\delta - \dashint _\Omega (\chi _{U_t} \ast \eta_\delta) \right) \, dx \\
=& \,
\int _{\Omega} \chi _{U_t} (\chi _{U_t} \ast \eta_\delta) \,dx - \frac{\mathscr{L}^d (U_t)}{\mathscr{L}^d (\Omega)} \int _\Omega (\chi _{U_t} \ast \eta_\delta)  \, dx \\
=& \,
\int _{\Omega} \chi _{U_t} (\chi _{U_t} \ast \eta_\delta -\chi _{U_t}) \,dx + \int _{\Omega} \chi _{U_t} ^2 \, dx\\
& \, - \frac{\mathscr{L}^d (U_t)}{\mathscr{L}^d (\Omega)} \int _\Omega (\chi _{U_t} \ast \eta_\delta -\chi _{U_t})  \, dx 
-\frac{\mathscr{L}^d (U_t)}{\mathscr{L}^d (\Omega)} \int _\Omega \chi _{U_t} \, dx \\
\geq & \, -C\delta \mathscr{H}^{d-1} (M_t) + \mathscr{L}^d (U_t) \left( 1- \frac{\mathscr{L}^d (U_t)}{\mathscr{L}^d (\Omega)} \right) \\
\geq & \, -C\delta \mathscr{H}^{d-1} (M_0) + \mathscr{L}^d (U_0) \left( 1- \frac{\mathscr{L}^d (U_0)}{\mathscr{L}^d (\Omega)} \right),
\end{split}
\end{equation*}
where we used Proposition \ref{prop7.2}, $\|\nabla \chi _{U_t}\| (\Omega)= \mathscr{H}^{d-1} (M_t)$, 
\eqref{eq:7.6}, and the volume preserving property. 
We choose $\delta >0$ such that
\[
-\int _{U_t} \div \vec{\zeta} \, dx 
\geq
\frac12 \mathscr{L}^d (U_0) \left( 1- \frac{\mathscr{L}^d (U_0)}{\mathscr{L}^d (\Omega)} \right) >0.
\]
By this and \eqref{eq:7.8}, 
\begin{equation}
\vert \lambda  \vert
\leq
\frac{C_\delta }{\omega'}
\left( \mathscr{H}^{d-1} (M_0) + \int _{M_t} \vert \vec{v}\vert \, d\mathscr{H}^{d-1} \right),
\label{eq:7.9}
\end{equation}
where $\omega' = \frac12 \mathscr{L}^d (U_0) \left( 1- \frac{\mathscr{L}^d (U_0)}{\mathscr{L}^d (\Omega)} \right)$.
The equality \eqref{eq:7.6} implies 
\begin{equation}
\int _0 ^s \int _{M_t} \vert \vec{v} \vert ^2 \, d \mathscr{H}^{d-1} \, dt \leq \mathscr{H}^{d-1} (M_0), \qquad s \in [0,T).
\label{eq:7.10}
\end{equation}
Therefore we obtain \eqref{eq:7.5} by \eqref{eq:7.9} and \eqref{eq:7.10}.

\end{proof}

Next we show some properties of the backward heat kernel.
\begin{lemma}[See\cite{ilmanen1993}]\label{estilmanen}
Let $D>0$ and $\nu$ be a Radon measure on $\R^d$ satisfying 
\begin{equation}\label{eq:7.12}
\sup _{R>0 ,x \in \mathbb{R}^d} \frac{ \nu (B_R(x)) }{ \omega _{d-1} R^{d-1} } \leq D.
\end{equation}
Then the following hold:
\begin{enumerate}
\item For any $a>0$ there exists $\gamma_1 =\gamma_1 (a)>0$ such that for any $r>0$ and 
for any $x,x_1 \in \mathbb{R}^d$ with $\vert x-x_1 \vert \leq \gamma_1 r$, we have the estimate
\begin{equation}
\int_{\mathbb{R}^d} \rho_{x_1 } ^r (y) \,d\nu (y) \leq \int _{\mathbb{R}^d} \rho_{x} ^r \,d\nu (y) +aD,
\label{rhoest1}
\end{equation} 
where $\rho _x ^r$ is given by \eqref{defrhor}. 
\item For any $r,R>0$ and for any $x\in \mathbb{R}^d$, we have
\begin{equation}
\int_{\mathbb{R}^d \setminus B_R(x)} \rho^r _x (y) \,d\nu (y) \leq 2^{d-1} e^{-3R^2 /8r^2}D.
\label{rhoest2}
\end{equation} 
\end{enumerate}
\begin{proof}
We only show \eqref{rhoest1} here (the estimate \eqref{rhoest2} can be shown more easily).
For $\beta \in (0,1)$, we have
\begin{equation}\label{eq:7.15}
\begin{split}
& \int_{\mathbb{R}^d} \rho_{x_1 } ^r (y) \,d\nu (y)
= 
\frac{1}{(\sqrt{2\pi} r)^{d-1}} 
\int_{\mathbb{R}^d} e^{- \frac{\vert x_1 -y \vert^2 }{2r^2}} \,d\nu (y) \\
= & \,
\frac{1}{(\sqrt{2\pi} r)^{d-1}} 
\int_0 ^1\nu (\{ y \mid e^{- \frac{\vert x_1 -y \vert^2 }{2r^2}} >k \} ) \, dk
= 
\frac{1}{(\sqrt{2\pi} r)^{d-1}} 
\int_0 ^1 \nu ( B_{\sqrt{2r^2 \log \frac{1}{k}}} (x_1) ) \, dk \\
\leq & \,
\omega ^{d-1} D 
\int_0 ^\beta \left( \log \frac{1}{k} \right)^{\frac{d-1}{2}}  \, dk
+
\frac{1}{(\sqrt{2\pi} r)^{d-1}} 
\int_\beta ^1 \nu ( B_{ r \sqrt{2 \log \frac{1}{k}}} (x_1) ) \, dk,
\end{split}
\end{equation}
where we used \eqref{eq:7.12}.
By
$
\int _0 ^1 \left( \log \frac{1}{k} \right)^{\frac{d-1}{2}} \, dk
=\Gamma (\frac{d-1}{2} +1)
=\pi ^{\frac{d-1}{2}} /\omega_{d-1}
$, we have
\begin{equation}\label{eq:7.16}
\int_0 ^\beta \left( \log \frac{1}{k} \right)^{\frac{d-1}{2}}  \, dk
\to 0 \qquad 
\text{as}
\qquad 
\beta \to 0.
\end{equation}
We choose $\gamma _1 >0$ depending only on $\beta$ such that
\[
\sqrt{2 \log \frac{1}{k}} + \gamma_1 
\leq 
\sqrt{2 \log \frac{1}{k-\beta}}
\qquad 
\text{for any} \quad k \in (\beta ,1].
\]
For any $x \in \R ^d$ with $\vert x-x_1 \vert \leq \gamma_1 r$, we have
$
B_{r \sqrt{2 \log \frac{1}{k}}} (x_1) 
\subset 
B_{(\sqrt{2 \log \frac{1}{k}} + \gamma_1) r} (x)
\subset
B_{r \sqrt{2 \log \frac{1}{k-\beta}}} (x)
$ for $k \in (\beta,1]$.
Therefore 
\begin{equation}\label{eq:7.17}
\begin{split}
& \frac{1}{(\sqrt{2\pi} r)^{d-1}} \int_\beta ^1 \nu ( B_{ r \sqrt{2 \log \frac{1}{k}}} (x_1) ) \, dk
\leq
\frac{1}{(\sqrt{2\pi} r)^{d-1}} \int_\beta ^1 \nu ( B_{ r \sqrt{2 \log \frac{1}{k-\beta}}} (x) ) \, dk\\
\leq & \,
\frac{1}{(\sqrt{2\pi} r)^{d-1}} \int_0 ^1 \nu ( B_{ r \sqrt{2 \log \frac{1}{k'}}} (x) ) \, dk'
=
\int_{\mathbb{R}^d} \rho_{x } ^r (y) \,d\nu (y).
\end{split}
\end{equation}
Hence \eqref{eq:7.15}, \eqref{eq:7.16}, and \eqref{eq:7.17} imply \eqref{rhoest1}.
\end{proof}
\end{lemma}

Let $u ^\varepsilon=u ^\varepsilon (x)$ be a smooth function and define
$$
e _\varepsilon (x) = \frac{\varepsilon \vert \nabla u^\varepsilon (x) \vert ^2 }{2} 
+ \frac{W(u ^\varepsilon (x))}{\varepsilon}, \qquad
\xi _\varepsilon (x) = \frac{\varepsilon \vert \nabla u^\varepsilon (x) \vert ^2 }{2} 
- \frac{W(u^\varepsilon (x))}{\varepsilon}.
$$
The following propositions are used in the proof of the integrality of $\mu _t$.
\begin{proposition}[See \cite{MR1803974, takasao-tonegawa, MR2040901}]\label{prop7.5}
For any $R \in (0,\infty)$, $E_0 \in (0,\infty)$, $s\in (0,1)$, and $N \in \N$,
there exists $\varrho \in (0,1)$ with the following property:
Assume that a set $Y \in \R^d$ has no more than $N+1$ elements and
$Y \subset \{ (0,\dots ,0, x_d) \in \R^d \mid x_d \in \R \}$, $\diam Y \leq \varrho R$,
and there exists $a \in (0,R) $ such that $\vert y-z \vert >3a$ holds for any $y,z \in Y$ with $y\not=z$.
Moreover, we assume the following.
\begin{enumerate}
\item $u ^\varepsilon \in C^2 (\{ y \in \R ^d \mid \dist (y, Y) <R \})$.
\item For any $x \in Y$ and $r \in [a,R]$,
\[
\int _{B_r (x)} \vert \xi _\varepsilon \vert
+ (1-(\nu_d) ^2) \varepsilon \vert \nabla u ^\varepsilon \vert^2
+ \varepsilon \vert \nabla u ^\varepsilon \vert 
\left\vert \Delta u ^\varepsilon - \frac{W'(u ^\varepsilon)}{\varepsilon ^2} \right\vert 
\, dy \leq \varrho r^{d-1}.
\]
Here $\nu = (\nu _1,\dots , \nu _d) = \frac{\nabla u ^\varepsilon}{\vert \nabla u ^\varepsilon \vert}$.
\item For any $x \in Y$, 
\[
\int _ a ^R \frac{d \tau}{\tau ^d} \int _{B_\tau (x)} (\xi _\varepsilon)_{+} \, dy \leq \varrho.
\]
\item For any $x \in Y$ and $ r \in [a,R]$,
\[
\int _{B_r (x)} \varepsilon \vert \nabla u ^\varepsilon \vert^2 \, dy 
\leq E_0 r^{d-1}.
\]
\end{enumerate}
Then, we have
\begin{equation}\label{eq:7.18}
\sum_{x \in Y} \frac{1}{a^{d-1}}
\int_{B_a (x)} e_\varepsilon \leq
s+ \frac{1+s}{R^{d-1}} 
\int_{\{ x \mid \dist (x,Y) <R\}} e_\varepsilon .
\end{equation}
\end{proposition}

\begin{proposition}[See \cite{MR1803974, takasao-tonegawa, MR2040901}]\label{prop7.6}
For any $s,b,\beta \in (0,1)$, and $c \in (1,\infty)$, there exist
$\varrho , \epsilon \in (0,1)$ and $L \in (1,\infty) $ with the following property:
Assume that $\varepsilon \in (0,\epsilon)$, $u ^\varepsilon \in C^2 (B_{4\varepsilon L} (0))$
and
\[
\sup _{B_{4\varepsilon L} (0)} \varepsilon \vert \nabla u ^\varepsilon \vert
\leq c,
\quad
\sup_{x,y \in B_{4\varepsilon L} (0), x\not=y} \varepsilon ^{\frac32}
\frac{ \vert \nabla u ^\varepsilon (x) -\nabla u ^\varepsilon (y) \vert }{\vert x-y \vert^{\frac12}}
\leq c, \quad
\vert u ^\varepsilon (0) \vert<1-b,
\]
\[
\int _{B_{4\varepsilon L} (0)} (\vert \xi _\varepsilon \vert +
(1-(\nu_d)^2) \varepsilon \vert \nabla u ^\varepsilon \vert^2) \, dx
\leq \varrho (4\varepsilon L)^{d-1},
\]
and
\[
\sup _{B_{4\varepsilon L} (0)} (\xi _\varepsilon)_+ \leq \varepsilon ^{-\beta}.
\]
Then we have
\begin{equation}\label{eq:7.19}
[-1+b, 1-b] \subset u ^\varepsilon (J)
\qquad \text{and} \qquad
\inf _{x \in J} \partial _{x_d} u ^\varepsilon (x) >0 
\ \ \text{or} \ \
\sup _{x \in J} \partial _{x_d} u ^\varepsilon (x) <0,
\end{equation}
where $J= B_{3\varepsilon L} (0) \cap \{ (0,\dots , 0, x_d) \in \R^d \mid x_d \in \R \}$.
In addition, we have
\[
\left \vert 
\sigma - \frac{1}{\omega _{d-1} (L\varepsilon )^{d-1}} \int _{B_{\varepsilon L}(0)} e_{\varepsilon}
\right \vert \leq s.
\]
\end{proposition}

\begin{remark}
Note that the assumptions for $(\xi _\varepsilon )_+$ in the propositions above
hold for the solution to \eqref{ac} with
suitable initial data (see \eqref{eq:3.15}).
\end{remark}

\bigskip

\section*{Acknowledgement}

The author would like to thank Professor Takashi Kagaya for his helpful comments.
This work was supported by JSPS KAKENHI
Grant Numbers JP20K14343, JP18H03670, 
and JSPS Leading Initiative for Excellent Young Researchers (LEADER) operated by Funds for the Development of Human Resources in Science and Technology.

\end{document}